\documentclass[amsfonts,12pt]{amsart}

\usepackage{epsfig}

\newtheorem{thm}{Theorem}[section]

\newtheorem{lem}[thm]{Lemma}
\newtheorem{cor}[thm]{Corollary}
\newtheorem{rem}[thm]{Remark}
\newtheorem{prp}[thm]{Proposition}
\newtheorem{ex}[thm]{Example}

\newcommand{\inte}{{\mathrm{int}}\,}

\newcommand{\conv}{{\mathrm{conv}}\,}

\newcommand{\lin}{{\mathrm{lin}}\,}

\newcommand{\ee}{\varepsilon}

\newcommand{\R}{\mathbb{R}}

\newcommand{\N}{\mathbb{N}}

\newcommand{\K}{\mathcal{K}}
\newcommand{\di}{\diamondsuit}

\begin{document}
\hfill\today
\bigskip

\title{OPERATIONS BETWEEN SETS IN GEOMETRY}
\author[Richard J. Gardner, Daniel Hug, and Wolfgang Weil]
{Richard J. Gardner, Daniel Hug, and Wolfgang Weil}
\address{Department of Mathematics, Western Washington University,
Bellingham, WA 98225-9063,USA} \email{Richard.Gardner@wwu.edu}
\address{Department of Mathematics, Karlsruhe Institute of Technology,
D-76128 Karlsruhe, Germany}
\email{daniel.hug@kit.edu}
\address{Department of Mathematics, Karlsruhe Institute of Technology,
D-76128 Karlsruhe, Germany} \email{wolfgang.weil@kit.edu}
\thanks{First author supported in
part by U.S.~National Science Foundation Grant DMS-1103612 and by the Alexander von Humboldt Foundation}
 \maketitle
\begin{abstract}
An investigation is launched into the fundamental characteristics of operations on and between sets, with a focus on compact convex sets and star sets (compact sets star-shaped with respect to the origin) in $n$-dimensional Euclidean space $\R^n$.  It is proved that if $n\ge 2$, with three trivial exceptions, an operation between origin-symmetric compact convex sets is continuous in the Hausdorff metric, $GL(n)$ covariant, and associative if and only if it is $L_p$ addition for some $1\le p\le\infty$.  It is also demonstrated that
if $n\ge 2$, an operation $*$ between compact convex sets is continuous in the Hausdorff metric, $GL(n)$ covariant, and has the identity property (i.e., $K*\{o\}=K=\{o\}*K$ for all compact convex sets $K$, where $o$ denotes the origin) if and only if it is Minkowski addition.  Some analogous results for operations between star sets are obtained.  Various characterizations are given of operations that are projection covariant, meaning that the operation can take place before or after projection onto subspaces, with the same effect.

Several other new lines of investigation are followed.  A relatively little-known but seminal operation called $M$-addition is generalized and systematically studied for the first time.  Geometric-analytic formulas that characterize continuous and $GL(n)$-covariant operations between compact convex sets in terms of $M$-addition are established.  It is shown that if $n\ge 2$, an $o$-symmetrization of compact convex sets (i.e., a map from the compact convex sets to the origin-symmetric compact convex sets) is continuous in the Hausdorff metric, $GL(n)$ covariant, and translation invariant if and only if it is of the form $\lambda DK$ for some $\lambda\ge 0$, where $DK=K+(-K)$ is the difference body of $K$.  The term ``polynomial volume" is introduced for the property of operations $*$ between compact convex or star sets that the volume of $rK*sL$, $r,s\ge 0$, is a polynomial in the variables $r$ and $s$. It is proved that if $n\ge 2$, with three trivial exceptions, an operation between origin-symmetric compact convex sets is continuous in the Hausdorff metric, $GL(n)$ covariant, associative, and has polynomial volume if and only if it is Minkowski addition.
\end{abstract}

\section{Introduction}
One of the most prevalent operations in mathematics is vector addition.  As an operation between sets $K$ and $L$, defined by
\begin{equation}\label{introMink1}
K+L=\{x+y: x\in K, y\in L\},
\end{equation}
it is usually called Minkowski addition and continues to find wide utility in science. For example, \cite{BDD} provides references to its application in computer-aided design and manufacturing, computer
animation and morphing, morphological image analysis, robot motion planning, and solid modeling. In geometry, when the sets $K$ and $L$ are usually compact convex sets in $\R^n$, the operation is a cornerstone of the Brunn-Minkowski theory, a profound and powerful apparatus developed by Minkowski, Blaschke, Aleksandrov, Fenchel, and others. Indeed, the whole theory can be said to arise from combining two concepts: volume and Minkowski addition.  This stems from a fundamental observation of Minkowski, who showed that if $K_1,\dots,K_m$ are compact convex sets in $\R^n$, and $t_1,\dots,t_m\ge 0$, the volume ${{\mathcal{H}}^n}(t_1K_1+\cdots +t_mK_m)$ is a homogeneous polynomial of degree $n$ in the variables $t_1,\dots,t_m$.  The coefficients in this polynomial are called mixed volumes. When $m=n$, $K_1=\cdots=K_i=K$, and $K_{i+1}=\cdots=K_n=B^n$, the unit ball in $\R^n$, then, up to constant factors, the mixed volumes turn out to be averages of volumes of orthogonal projections of $K$ onto subspaces, and include the volume, surface area, and mean width of $K$, as special cases.  In this way, Minkowski created a unified framework for solving problems of a metrical character.  The influence of the Brunn-Minkowski theory is felt not only in geometry, but in many other areas both in and outside mathematics.   To hint at these, we recall that zonoids (limits in the Hausdorff metric of finite Minkowski sums of line segments) alone have found application in stochastic geometry \cite[Section~4.6]{SchW}, random determinants \cite{Vit91}, Hilbert's fourth problem \cite{Ale88}, data analysis and mathematical economics \cite{Mos02}, and mathematical finance \cite{MolS}.  The classic treatise of Schneider \cite{Sch93} provides a detailed survey of the Brunn-Minkowski theory and a host of references.

There are several (though surprisingly few) other ways of combining sets that have found application in geometry and beyond. However, only the imagination limits the number of different operations that might be considered.  In this paper we initiate an investigation motivated, at the first level, by the simple yet fundamental question:  What is so special about the known operations, Minkowski addition in particular?

Before outlining our results, we briefly survey some other useful operations. Two of these underpin far-reaching extensions to the classical Brunn-Minkowski theory.  The first, $L_p$ addition, $1<p\le \infty$, introduced by Firey \cite{Fir61}, \cite{Fir62} and denoted by $+_p$, is defined by
\begin{equation}\label{introMink2}
h_{K+_pL}(x)^p=h_K(x)^p+h_L(x)^p,
\end{equation}
for all $x\in\R^n$ and compact convex sets $K$ and $L$ containing the origin, where the functions are the support functions of the sets involved.  (When $p=\infty$, (\ref{introMink2}) is interpreted as $h_{K+_{\infty}L}(x)=\max\{h_K(x),h_L(x)\}$, as is customary. It is possible to extend $L_p$ addition to arbitrary compact convex sets; see Section~\ref{subLp}.) Note that when $p=1$, (\ref{introMink2}) is equivalent to (\ref{introMink1}) for compact convex $K$ and $L$, so we regard Minkowski addition as the case $p=1$ of $L_p$ addition.  The rich theory that emerges is called the $L_p$-Brunn-Minkowski theory.  The second, radial addition, denoted by $\widetilde{+}$, is just Minkowski addition restricted to lines through the origin $o$: One defines $x\widetilde{+}y=x+y$ if $x$, $y$, and $o$ are collinear, $x\widetilde{+}y=o$, otherwise, and
$$K\widetilde{+}L=\{x\widetilde{+}y: x\in K, y\in L\},$$
for star sets (compact sets star-shaped at the origin) $K$ and $L$ in $\R^n$.  When combined with volume, radial addition gives rise to another substantial appendage to the classical theory, called the dual Brunn-Minkowski theory.  Indeed, in founding this theory, Lutwak \cite{L4} proved that if $K_1,\dots,K_m$ are star sets in $\R^n$, and $t_1,\dots,t_m\ge 0$, the volume ${{\mathcal{H}}^n}(t_1K_1\widetilde{+}\cdots\widetilde{+}t_mK_m)$ is a homogeneous polynomial of degree $n$ in the variables $t_1,\dots,t_m$, a perfect analog of Minkowski's theorem.  He called the coefficients of this polynomial dual mixed volumes and showed that up to constant factors, they include averages of volumes of sections of a star set by subspaces.

The significance of these two extensions of the classical Brunn-Minkowski theory cannot be overstated.   The $L_p$-Brunn-Minkowski theory has allowed many of the already potent sharp affine isoperimetric inequalities of the classical theory, as well as related analytic inequalities, to be strengthened; provided tools for attacks on major unsolved problems such as the slicing problem of Bourgain; and consolidated connections between convex geometry and information theory.  See, for example, \cite{DP}, \cite{HS1}, \cite{HS2}, \cite{LYZ1}, \cite{LYZ2}, \cite{LYZ3}, \cite{LYZ4}, \cite{LYZ5}, and \cite{LYZ6}.  The dual Brunn-Minkowski theory can count among its successes the solution of the Busemann-Petty problem in \cite{G3}, \cite{GKS}, \cite{L5}, and \cite{Z1}. It also has connections and applications to integral geometry, Minkowski geometry, the local theory of Banach spaces, and stereology; see \cite{Gar06} and the references given there.

As well as $L_p$ and radial addition, there are a few other operations familiar to many geometers:  $p$th radial addition $\widetilde{+}_p$, $-\infty\le p\neq 0\le \infty$ (which is, roughly, to radial addition $\widetilde{+}$ as $+_p$ is to Minkowski addition $+$), polar $L_p$ addition, and Blaschke addition.  These are all defined and discussed in Section~\ref{examples}.

Faced with our motivating question above---what is so special about these operations?---it is reasonable to compile a list of properties that they may enjoy.  Algebraic properties such as commutativity and associativity immediately come to mind.  Several properties can be considered that express good behavior under natural geometrical operations, such as continuity in an appropriate metric or covariance with respect to nonsingular linear transformations.  These and other properties that feature in our study are defined formally in Section~\ref{properties}.  They include two worthy of special mention: the identity property, meaning that addition of the single point at the origin leaves a set unchanged, and projection covariance, which states that the operation can take place before or after projection onto subspaces, with the same effect.

Projection covariance plays a star role in this paper.  Note that it is completely natural, being automatically satisfied when the operation takes place between compact convex sets and is both continuous with respect to the Hausdorff metric and $GL(n)$ covariant (see Lemma~\ref{CGimpliesP}).  (Here, $GL(n)$ covariant means that the operation can take place before or after the sets concerned undergo the same transformation in $GL(n)$.)  Generally, the Brunn-Minkowski theory caters readily for projections.  Similarly, intersections with subspaces are handled efficiently by the dual Brunn-Minkowski theory, and for this reason we also consider the corresponding property, section covariance. Both projections and sections are of course of prime importance in geometry, but besides, as Ball states in his MathSciNet review of the book \cite{Ko}: ``A variety of problems from several areas of mathematics, including probability theory, harmonic analysis, the geometry of numbers, and linear programming, can be couched as questions about the volumes of sections or projections of convex bodies."

At this point we can state two of our main results.

\smallskip

{\bf Theorem A}. {\em If $n\ge 2$, with three trivial exceptions, an operation between $o$-symmetric compact convex sets is continuous in the Hausdorff metric, $GL(n)$ covariant, and associative if and only if it is $L_p$ addition for some $1\le p\le\infty$.}  (See Theorem~\ref{thm1}; here, $o$-symmetric means symmetric with respect to the origin.)

\smallskip

{\bf Theorem B}. {\em If $n\ge 2$, an operation between compact convex sets is continuous in the Hausdorff metric, $GL(n)$ covariant, and has the identity property if and only if it is Minkowski addition.} (See Corollary~\ref{cor277}.)

\smallskip

Note that in both cases only one extra and quite weak property beyond continuity and $GL(n)$ covariance is required for these strong classification theorems.  Moreover, none of the properties can be omitted, as we show with various examples.

In both of these results, it suffices to assume projection covariance instead of continuity and $GL(n)$ covariance. In fact, it is a consequence of our work that an operation between compact convex sets is projection covariant if and only if it is continuous and $GL(n)$ covariant; see Corollaries~\ref{contaff} and~\ref{contaff7}.  It is remarkable that this single property should have such dramatic consequences.  Section covariance, for example, which at first sight may seem to be of approximately the same strength, is decidedly weaker.  Thus we prove:

\smallskip

{\bf Theorem C}. {\em If $n\ge 2$, with three trivial exceptions, an operation between $o$-symmetric star sets is section covariant, continuous in the radial metric, rotation covariant, homogeneous of degree 1, and associative if and only if it is $p$th radial addition for some $-\infty\le p\neq 0\le \infty$.}  (See Theorem~\ref{thm1a}.)

\smallskip

The meaning of the term radial metric can be found, along with other basic definitions and notation, in Section~\ref{prelim}. Though the list of assumptions in Theorem~C is longer than in Theorem~A, we again show by examples that none can be omitted. The analogy between Theorems A and C is yet another instance of a still mysterious, imperfect duality between projections onto subspaces and sections by subspaces, discussed at length in \cite{Gar06} and \cite{Ko}, for example.

Notice that a consequence of Theorems A and C is that operations with these properties must be commutative, even though it is associativity that is assumed. This extraordinary effect of associativity under certain circumstances actually has a long history, going back at least to Abel's pioneering work on the so-called associativity equation. A full discussion would take us too far afield; we refer the reader to the books \cite{Acz66} and \cite{Acz89} on functional equations and \cite{AFS06} and \cite{KMP} on triangular norms and copulas.  It is also a well-known phenomenon in semiring theory.  In fact, associativity is brought to bear in Theorems A and C via a result of Pearson \cite{Pea66} (see Proposition~\ref{prp2}).  Pearson's work is more suited to our purpose than the closely related and earlier articles of Acz\'{e}l \cite{Acz55} and Bohnenblust \cite{Boh40}.  In Section~\ref{background} we provide a minor service by clarifying the relationship between the three results.

Our study has more to offer than merely determining minimal lists of properties that characterize known operations. En route to Theorem A we establish a complete geometric-analytic characterization of all continuous and $GL(n)$-covariant operations (equivalently, of all projection covariant operations) between $o$-symmetric compact convex sets in $\R^n$ for $n\ge 2$, by proving in Theorem~\ref{thm12} that such operations are precisely those corresponding to $M$-addition for some compact convex set $M$ that is 1-unconditional (symmetric with respect to the coordinate axes) in $\R^2$.  This means that for all $o$-symmetric compact convex sets $K$ and $L$ in $\R^n$, the operation is defined by
\begin{equation}\label{intro1}
K\oplus_M L= \{ a x + b y : x\in K, y\in L, (a,b )\in M\}.
\end{equation}
Surprisingly, this very natural generalization of Minkowski and $L_p$ addition (which correspond here to taking $M=[-1,1]^2$ and $M$ equal to the unit ball in $l^2_{p'}$, $1/p+1/p'=1$, respectively) appears to have been introduced only quite recently, by Protasov \cite{Pro97}, inspired by work on the joint spectral radius in the theory of normed algebras.

We actually discovered Protasov's work after finding what turns out to be an equivalent version of Theorem~\ref{thm12}:  Projection covariant (equivalently, continuous and $GL(n)$ covariant) operations $*$ between $o$-symmetric compact convex sets are precisely those given by the formula
\begin{equation}\label{intro2}
h_{K*L}(x)=h_{M}\left(h_K(x),h_L(x)\right),
\end{equation}
for all $x\in \R^n$ and some $1$-unconditional compact convex set $M$ in $\R^2$.  The obviously fundamental character of $M$-addition and the equivalence of (\ref{intro1}) and (\ref{intro2}) in the given context prompted us to simultaneously generalize $M$-addition (so that the sets $K$, $L$, and $M$ are arbitrary sets in the appropriate Euclidean spaces) and initiate a thorough investigation into its properties.  In fact, we further extend the scope by considering the analogous $m$-ary operation between $m$ sets in $\R^n$, which we call $M$-combination (the natural extension of (\ref{intro1}), where $M$ is a subset of $\R^m$).  Our results are set out in Section~\ref{Maddition}.

In the transition from $o$-symmetric to general sets, we begin by observing that if an operation between arbitrary compact convex sets satisfies the hypotheses of Theorem A when restricted to the $o$-symmetric sets, then this restriction must be $L_p$ addition for some $1\le p\le\infty$ (see Corollary~\ref{corrlem01arb}).  However, Theorem A itself does not hold for operations between arbitrary compact convex sets and in this regard, the following example is instructive.  Define
\begin{equation}\label{introex}
K*L=(1/2)DK+(1/2)DL,
\end{equation}
for all compact convex sets $K$ and $L$ in $\R^n$,
where $DK$ is the difference body $K+(-K)$ of $K$.  When $K$ and $L$ are $o$-symmetric, they coincide with their reflections $-K$ and $-L$ in the origin, so $*$ is the same as Minkowski addition, but for general $K$ and $L$, it is not equal to $L_p$ addition for any $p$. The formula (\ref{introex}) leads to another new investigation, in Section~\ref{Symmetrization}, on $o$-symmetrizations, i.e., maps from the compact convex sets (or star sets) to the $o$-symmetric compact convex sets (or $o$-symmetric star sets, respectively).  We obtain the following characterization of the difference body operator.

\smallskip

{\bf Theorem D}. {\em If $n\ge 2$, an $o$-symmetrization of compact convex sets is continuous in the Hausdorff metric, $GL(n)$ covariant, and translation invariant if and only if it is of the form $\lambda DK$ for some $\lambda\ge 0$.}  (See Corollary~\ref{corthm2s}.)

\smallskip

Once again, none of the assumptions can be omitted and the version stated in Corollary~\ref{corthm2s} is only apparently more general, since an $o$-symmetrization of compact convex sets is projection covariant if and only if it is both continuous and $GL(n)$ covariant, by Lemma~\ref{diCGimpliesP} and Corollary~\ref{contaffs}.

In the process of proving Theorem~B, we show in Theorem~\ref{thm275} that projection covariant (equivalently, continuous and $GL(n)$ covariant) operations $*$ between arbitrary compact convex sets are precisely those given by the formula
\begin{equation}\label{intro27}
h_{K*L}(x)=h_{M}\left(h_K(-x),h_K(x),h_L(-x),h_L(x)\right),
\end{equation}
for all $x\in \R^n$ and some closed convex set $M$ in $\R^4$. (More work remains to be done to understand which such sets $M$ give rise via (\ref{intro27}) to valid operations, but we show that this is true if $M$ is any compact convex subset of $[0,\infty)^4$, in which case $K*L$ is just the $M$-combination of $K$, $-K$, $L$, and $-L$.)  A similar result for section covariant operations $*$ between arbitrary star sets is given in Theorem~\ref{lem01788}; however, in this case the natural analog of Theorem~B fails to hold.

In Section~\ref{Minkowski}, we set off in yet another new direction by focusing on operations between $o$-symmetric compact sets having polynomial volume.  This means that for all $r,s\ge 0$,
\begin{equation}\label{intropol}
{{\mathcal{H}}^n}(rK*sL)=\sum_{i,j=0}^{m(K,L)} a_{ij}(K,L)r^{i}s^j,
\end{equation}
for some real coefficients $a_{ij}(K,L)$, some $m(K,L)\in\N\cup\{0\}$, and all $o$-symmetric compact convex or star sets $K$ and $L$ in $\R^n$.
It was mentioned above that both Minkowski and radial addition have this property, the coefficients being mixed or dual mixed volumes, respectively, but note that here the polynomial need not be homogeneous and its degree may depend on $K$ and $L$.  In Theorem~\ref{thm5} we solve a problem known to experts by showing that $+_p$ does not have polynomial volume for $p>1$.  (This effectively implies that the full set of ``$L_p$-mixed volumes" is not available unless $p=1$.)  As a result, we obtain the following characterization of Minkowski addition as an operation between $o$-symmetric sets.

\smallskip

{\bf Theorem~E}. {\em If $n\ge 2$, with three trivial exceptions, an operation between $o$-symmetric compact convex sets is projection covariant, associative, and has polynomial volume if and only if it is Minkowski addition.}  (See Corollary~\ref{cor6}.)

\smallskip

Further examples show that none of the assumptions can be removed. An analog for operations between star sets is stated in Corollary~\ref{cor8}, but, interestingly, here $p$th radial addition is allowed, provided that $p\in \N$ and $p$ divides the dimension $n$.

This work raises many natural questions and invites extensions in several different directions, so we regard it as the first stage in an extensive program. Rather than elaborate on this in detail here, we mention only that some of the results are applied in \cite{GHW} to further the new Orlicz-Brunn-Minkowski theory (see \cite{LYZ7}, \cite{LYZ8}).

The paper is organized as follows.  After Section~\ref{prelim} giving definitions and notation, Section~\ref{background} mainly concerns functions satisfying the associativity equation and different types of homogeneity.  The properties of operations and $o$-symmetrizations with which we work are listed and defined in Section~\ref{properties}, and some basic lemmas relating them are proved. In Section~\ref{examples}, the various examples of useful operations and $o$-symmetrizations are defined and their properties are discussed, including some new observations.  Our generalization and extension of $M$-addition is presented in Section~\ref{Maddition} and several fundamental results are established that shed light on its behavior.  Section~\ref{results} focuses on operations between $o$-symmetric compact convex or star sets and Section~\ref{Symmetrization} sets out the results on $o$-symmetrizations.  The symmetry restriction is discarded in Section~\ref{arbresults}, where we deal with operations between arbitrary compact convex or star sets.  In the final Section~\ref{Minkowski}, we state our results on the polynomial volume property.

The first author acknowledges discussions with Mathieu Meyer concerning Theorem~\ref{thm5}.

\section{Definitions and preliminaries}\label{prelim}

As usual, $S^{n-1}$ denotes the unit sphere and $o$ the origin in Euclidean
$n$-space $\R^n$.  The unit ball in $\R^n$ will be denoted by $B^n$. The
standard orthonormal basis for $\R^n$ will be $\{e_1,\dots,e_n\}$.  Otherwise, we usually denote the coordinates of $x\in \R^n$ by $x_1,\dots,x_n$.  We write
$[x,y]$ for the line segment with endpoints $x$ and $y$. If $x\in \R^n\setminus\{o\}$, then $x^{\perp}$ is the $(n-1)$-dimensional subspace orthogonal to $x$ and $l_x$ is the line through $o$ containing $x$.  (Throughout the paper, the term {\em subspace} means a linear subspace.)

If $X$ is a set,  we denote by $\partial X$, $\inte X$, $\lin X$, $\conv X$, and $\dim X$ the {\it boundary}, {\it interior}, {\it linear hull}, {\it convex hull}, and {\it dimension} (that is, the dimension of the affine hull) of $X$, respectively.  If $S$ is a subspace of $\R^n$, then $X|S$ is the (orthogonal) projection of $X$ onto $S$ and $x|S$ is the projection of a vector $x\in\R^n$ onto $S$.

If $t\in\R$, then $tX=\{tx:x\in X\}$. When $t>0$, $tX$ is called a {\em dilatate} of $X$.  The set $-X=(-1)X$ is the {\em reflection} of $X$ in the origin.

A {\it body} is a compact set equal to the closure of its interior.

We write ${\mathcal{H}}^k$ for $k$-dimensional Hausdorff measure in $\R^n$,
where $k\in\{1,\dots, n\}$. The notation $dz$ will always
mean $d{\mathcal{H}}^k(z)$ for the appropriate $k=1,\dots, n$.

We follow Schneider \cite{Sch93} by writing $\kappa_n$ for the volume ${{\mathcal{H}}^n}(B^n)$ of
the unit ball in $\R^n$, so that $\kappa_n=\pi^{n/2}/\,\Gamma(1+n/2)$.

The Grassmannian of $k$-dimensional subspaces in $\R^n$ is denoted by ${\mathcal{G}}(n,k)$.

A set is {\it $o$-symmetric} if it is centrally symmetric, with center at the
origin.  We shall call a set in $\R^n$ {\em $1$-unconditional} if it is symmetric with respect to each coordinate hyperplane; this is traditional in convex geometry for compact convex sets. If $X$ is a set in $\R^n$, we denote by
\begin{equation}\label{1uncond}
\widehat{X}=\{(\alpha_1x_1,\alpha_2x_2,\dots,\alpha_n x_n): (x_1,x_2,\dots,x_n)\in X, |\alpha_i|\le 1, i=1,\dots,n\}
\end{equation}
its {\em 1-unconditional hull}.  Geometrically, this is the union of all $o$-symmetric coordinate boxes that have at least one vertex in $X$.

Let ${\mathcal K}^n$ be the class of nonempty compact convex
subsets of  $\R^n$, let ${\mathcal K}^n_s$ denote the class of $o$-symmetric members of ${\mathcal K}^n$, let ${\mathcal{K}}_o^n$ be the class of members of ${\mathcal K}^n$ containing the origin, and let ${\mathcal{K}}_{oo}^n$ be those sets in ${\mathcal K}^n$ containing the origin in their interiors. A set $K\in {\mathcal K}^n$ is called a {\em
convex body} if its interior is nonempty.

If $K$ is a nonempty closed (not necessarily bounded) convex set, then
\begin{equation}\label{suppf}
h_K(x)=\sup\{x\cdot y: y\in K\},
\end{equation}
for $x\in\R^n$, is its {\it support function}. A nonempty closed convex set is uniquely determined by its support function.   Support functions are {\em homogeneous of degree 1}, that is,
$$h_K(rx)=rh_K(x),$$
for all $x\in \R^n$ and $r\ge 0$, and are therefore often regarded as functions on $S^{n-1}$.  They are also {\em subadditive}, i.e.,
$$h_K(x+y)\le h_K(x)+h_K(y),$$
for all $x,y\in \R^n$.  Any real-valued function on $\R^n$ that is {\em sublinear}, that is, both homogeneous of degree 1 and subadditive, is the support function of a unique compact convex set.
The {\em Hausdorff distance} $\delta(K,L)$ between sets $K,L\in {\mathcal K}^n$
can be conveniently defined by
\begin{equation}\label{HD}
\delta(K,L)=\|h_K-h_L\|_{\infty},
\end{equation}
where $\|\cdot\|_{\infty}$ denotes the $L_{\infty}$ norm on $S^{n-1}$.  (For compact convex sets, this is equivalent to the alternative definition
$$\delta(K,L)=\max\{\max_{x\in K}d(x,L),\max_{x\in L}d(x,K)\}$$
that applies to arbitrary compact sets, where $d(x,E)$ denotes the distance
from the point $x$ to the set $E$.)  Proofs of these facts can be found in \cite{Sch93}.  Gruber's book \cite{Gru07} is also a good general reference for convex sets.

Let $K$ be a nonempty, closed (not necessarily bounded) convex set. If $S$ is a subspace of $\R^n$, then it is easy to show that
\begin{equation}\label{hproj}
h_{K|S}(x)=h_K(x|S),
\end{equation}
for each $x\in \R^n$.  The formula (see \cite[(0.27), p.~18]{Gar06})
\begin{equation}\label{suppchange}
h_{\phi K}(x)=h_K(\phi^t x),
\end{equation}
for $x\in \R^n$, gives the change in a support function under a transformation $\phi\in GL(n)$, where $\phi^t$ denotes the linear transformation whose standard matrix is the transpose of that of $\phi$.  (Equation (\ref{suppchange}) is proved in \cite[p.~18]{Gar06} for compact sets, but the proof is the same if $K$ is unbounded.)

The {\em polar set} of an arbitrary set $K$ in $\R^n$ is
$$
K^{\circ}=\{x\in \R^n:\,x\cdot y\le 1 {\mathrm{~for~all~}} y\in K\}.$$
See, for example, \cite[p.~99]{Web}.

Recall that $l_x$ is the line through the origin containing $x\in\R^n\setminus\{o\}$. A set $L$ in $\R^n$ is {\it star-shaped at $o$} if $L\cap l_u$ is either empty
or a (possibly degenerate) closed line segment for each $u\in
S^{n-1}$. If $L$ is star-shaped at $o$, we define its {\it radial
function} $\rho_L$ for $x\in \R^n\setminus\{o\}$ by
$$\rho_{L}(x)=\left\{\begin{array}{ll} \max\{c:cx\in L\}, &
\mbox{if $L\cap l_x\neq\emptyset$,}\\ 0, & \mbox{otherwise.}
\end{array}\right.
$$ This definition is a slight modification of \cite[(0.28)]{Gar06};
as defined here, the domain of $\rho_L$ is always $\R^n\setminus\{o\}$.  Radial functions are {\em homogeneous of degree $-$1}, that is,
$$\rho_L(rx)=r^{-1}\rho_L(x),$$
for all $x\in \R^n\setminus\{o\}$ and $r>0$, and are therefore often regarded as functions on $S^{n-1}$.

In this paper, a {\it star set} in $\R^n$ is a compact set that is star-shaped at $o$ and contains $o$.  (Other definitions have been used; see, for example
\cite[Section~0.7]{Gar06} and \cite{GV}.)  We denote the class of star sets in
$\R^n$ by ${\mathcal S}^n$ and the subclass of such sets that are $o$-symmetric by ${\mathcal S}^n_{s}$. Note that each of these two classes is closed under finite unions, countable intersections, and intersections with subspaces.
The {\em radial metric} $\widetilde{\delta}$ defines the distance between star sets $K,L\in {\mathcal S}^n$ by
$$\widetilde{\delta}(K,L)=\|\rho_K-\rho_L\|_{\infty}=\sup_{u\in S^{n-1}}|\rho_K(u)-\rho_L(u)|.$$
Observe that this differs considerably from the Hausdorff metric; for example, the radial distance between any two different $o$-symmetric line segments containing the origin and of length two is one.

Let $\mathcal{C}$ be a class of sets in $\R^n$ and let $\mathcal{C}_s$ denote the subclass of $o$-symmetric members of $\mathcal{C}$. We call a map $\di:{\mathcal{C}}\rightarrow{\mathcal{C}}_s$ an {\em $o$-symmetrization} on $\mathcal{C}$, and for $K\in \mathcal{C}$, we call $\di K$ an {\em $o$-symmetral}.

\section{Some background results}\label{background}

The following result is due to Bohnenblust \cite{Boh40}.

\begin{prp}\label{prp1}
Let $f:[0,\infty)^2\to\R$  satisfy the following conditions:

{\em (i)} $f(rs,rt)=rf(s,t)$ for $r,s,t\ge 0$;

{\em (ii)} $f(s,t)\le f(s',t')$ for $0\le s\le s'$ and $0\le t\le t'$;

{\em (iii)} $f(s,t)=f(t,s)$;

{\em (iv)} $f(0,1)=1$;

{\em (v)} $f(s,f(t,u))=f(f(s,t),u)$ for $s,t,u\ge 0$.

Then there exists $p$, $0<p\le \infty$, such that
\begin{equation}\label{2}
f(s,t)=(s^p+t^p)^{1/p},
\end{equation}
where, in the case $p=\infty$, we mean $f(s,t)=\max\{s,t\}$.
\end{prp}

The equation in (v) is called the {\em associativity equation} and has generated a large literature; see, for example, \cite{Acz66}, \cite{Acz89}, and \cite{AFS06}.

In \cite[Theorem~4]{Fle07}, Fleming states: {\em The conclusion of Bohnenblust's theorem remains true even with condition {\em (iii)} of the hypotheses removed}.  He means to say also that (iv) should be replaced  by $f(1,0)=f(0,1)=1$ (or else the function $f(s,t)=t$ for all $s,t\ge 0$ would be a counterexample).
Fleming ascribes this result to B.~Randrianantoanina in a personal communication.  See also \cite[Theorem~9.5.3]{FleJ08}.

For $s,t\ge 0$, let

\begin{eqnarray*}
f_1(s,t)&=&\log\left(e^s+e^t-1\right);\\
f_2(s,t)&=& \begin{cases}
\min\{s,t\},& {\text{ if $s>0$ and $t>0$,}}\\
\max\{s,t\},& {\text{ if $s=0$ or $t=0$}};
\end{cases}\\
f_3(s,t)&=&t;\\
f_4(s,t)&=&\min\{s,t\};\\
f_5(s,t)&=&s+t+\sqrt{st}.\\
\end{eqnarray*}

Then one can check that for $i=1,\dots,5$, the function $f_i(s,t)$ satisfies all but the $i$th of the five hypotheses of Proposition~\ref{prp1}.

A related result is due to Acz\'{e}l \cite[Theorem~2]{Acz55}.  He shows that if $f$ is continuous and satisfies only the hypotheses (i), (ii) (but with {\em strict} inequalities), and (v) in Bohnenblust's theorem, then $f$ is given by (\ref{2}).  Incidentally, according to Acz\'{e}l, (ii) with strict inequalities is equivalent to the cancellation law ($f(s,t)=f(s,u)\Rightarrow t=u$ and $f(s,t)=f(u,t)\Rightarrow s=u$).

The following result, stronger than Acz\'{e}l's, was proved by Pearson \cite[Theorem~2]{Pea66} in a paper on topological semigroups.

\begin{prp}\label{prp2}
Let $f:[0,\infty)^2\rightarrow [0,\infty)$ be a continuous function satisfying conditions {\em (i)} and {\em (v)} of Proposition~\ref{prp1}.  Then either $f(s,t)=0$, or $f(s,t)=s$, or $f(s,t)=t$, or there exists $p$, $0<p\le \infty$, such that $f$ is given by (\ref{2}), or
$$
f(s,t)=\begin{cases}
(s^p+t^p)^{1/p},& {\text{ if $s>0$ and $t>0$,}}\\
0,& {\text{ if $s=0$ or $t=0$}},
\end{cases}
$$
for some $p<0$, or $f(s,t)=\min\{s,t\}$ (the case $p=-\infty$).
\end{prp}

The functions $f_1(s,t)$ and $f_5(s,t)$ above are continuous and show that both conditions (i) and (v) are necessary in the previous proposition. Note that Proposition~\ref{prp2} also implies that any function $f:[0,\infty)^2\rightarrow[0,\infty)$ that satisfies conditions (i) and (v), but not (ii), of Proposition~\ref{prp1} cannot be continuous. Indeed, if such a function were continuous, it would have to be one of the possibilities given by Proposition~\ref{prp2}, but each of these satisfies (ii).

Proposition~\ref{prp2} implies Acz\'{e}l's result mentioned above, since if $f:[0,\infty)^2\to \R$ satisfies (i) and (ii) of Proposition~\ref{prp1}, then (i) implies that $f(0,0)=f(0\cdot 1,0\cdot 1)=0\cdot f(1,1)=0$ and from (ii) it follows that $f$ is nonnegative.  Therefore Proposition~\ref{prp2} applies, but the only strictly increasing function provided by Proposition~\ref{prp2} is the one given in \eqref{2}.

Let $f:[0,\infty)^2\to \R$ satisfy (i), (ii), (iv) (in the symmetric form $f(1,0)=f(0,1)=1$), and (v) in Proposition~\ref{prp1}.
If $s,t>0$, then by (i) and (iv), $f(s,t)=tf(s/t,1)\ge tf(0,1)=t>0$. Hence the restriction $f:(0,\infty)^2\to(0,\infty)$ is well defined and satisfies (i), (ii), and (v). We claim that this restriction is also continuous.  To see this, let $s_0,t_0>0$, define $w=\min\{s_0,t_0\}$, and choose
$\varepsilon\in[0,w)$. Then, for
$s\in  (s_0-\varepsilon,s_0+\varepsilon)$ and $t\in
(t_0-\varepsilon,t_0+\varepsilon)$, we use first (ii) and then (i) to obtain
$$
f(s,t)\le f(s_0+\varepsilon,t_0+\varepsilon)\le
f\left(s_0(1+\varepsilon/w),t_0(1+\varepsilon/w)\right)
=(1+\varepsilon/w)f(s_0,t_0).
$$
Similarly,
$f(s,t)\ge (1-\varepsilon/w)f(s_0,t_0)$ and hence
$$
|f(s,t)-f(s_0,t_0)|\le \frac{\varepsilon}{w} f(s_0,t_0).
$$
Therefore $f$ is continuous at $(s_0,t_0)$. Another result of Pearson \cite[Theorem~1]{Pea66} then applies and shows that the
restriction of $f$ to $(0,\infty)^2$ must be of one of five types of functions given there. The condition $f(0,1) = f(1,0) = 1$ and (i) imply that  $f(s,0) = s$ and $f(0,t) = t$ for all $s,t\ge 0$. This and (ii) can be used to rule out all the functions provided by \cite[Theorem~1]{Pea66} except those given by \eqref{2}. This shows that Bohnenblust's theorem is a consequence of \cite[Theorem~1]{Pea66} and also confirms Fleming's statement mentioned above.

The following proposition sheds light on the relation between various types of homogeneity.  We omit the proof, which is an easy adaptation of the argument in \cite[p.~345]{Acz89}.

\begin{prp}\label{quasilem}
Let $f:[0,\infty)^2\to[0,\infty)$ be a continuous function satisfying
$$
f(r s,r t)=g(r)f(s,t),
$$
for all $r,s,t\ge 0$ and some function $g:[0,\infty)\to[0,\infty)$. Then $f\equiv 0$ or $g(r)=r^c$ for some $c\in\R$ and all $r\ge 0$.  If, in addition, $f(0,t)=t$ for all $t> 0$ (or $f(s,0)=s$ for all $s> 0$), then $g(r)=r$ for all $r\ge 0$.
\end{prp}

The inequality
\begin{equation}\label{mul}
\varphi^{-1}\left(\varphi(s_1+s_2) + \varphi(t_1+t_2)\right)\le
\varphi^{-1}\left(\varphi(s_1) + \varphi(t_1)\right)+\varphi^{-1}\left(\varphi(s_2) + \varphi(t_2)\right),
\end{equation}
where $s_1,s_2,t_1,t_2\ge 0$ and $\varphi: [0,\infty)\rightarrow [0,
\infty)$ is continuous and strictly increasing with $\varphi(0)=0$, is known as {\em Mulholland's inequality}.  It was first studied by Mulholland \cite{Mul50} and represents a generalization of Minkowski's inequality to functions other than $\varphi(s)=s^p$, $p\ge 1$.  Mulholland proved that (\ref{mul}) holds if $\varphi(s)=s\exp(\psi(\log s))$ for some continuous, increasing, convex function $\psi$ on $\R$. He gave as particular examples satisfying this condition the functions $\varphi(s)=\sinh s$ and $\varphi(s)=s^{1+a}\exp(bs^c)$ for $a,b,c\ge 0$.

\section{Properties of binary operations and $o$-symmetrizations}\label{properties}

For certain classes $\mathcal{C}$, $\mathcal{D}$ of sets in $\R^n$ with $\mathcal{C}\subset \mathcal{D}$, we seek natural properties to impose on an arbitrary binary operation $*:{\mathcal{C}}^2\rightarrow {\mathcal{D}}$ that force the operation to coincide with a known one.  The investigation is restricted to the cases $\mathcal{D}\subset {\mathcal{K}}^n$ and $\mathcal{D}\subset {\mathcal{S}}^n$.  In the following list, it is assumed that $\mathcal{C}$ is an appropriate class for the property under consideration.  The properties are supposed to hold for all appropriate $K, L, M, N, K_m, L_m\in{\mathcal{C}}$ and for all $r,s\ge 0$.  Properties~10--12 do not play a major role in our investigation, but are nonetheless familiar properties that could be considered in characterizing known operations.  Moreover, since the best-known operations all satisfy these three properties, they act as extra motivation for Property~(13), which we shall see in Section~\ref{Minkowski} can distinguish Minkowski addition from $L_p$ addition for $p>1$.

\medskip

1. (Commutativity) $K*L=L*K$.

2. (Associativity) $K*(L*M)=(K*L)*M$.

3. (Homogeneity of degree $k$) $(rK)*(rL)=r^k(K*L)$.

4. (Distributivity) $(rK)*(sK)=(r+s)K$.

5.  (Identity) $K*\{o\}=K=\{o\}*K$.

6. (Continuity) $K_m\rightarrow M, L_m\rightarrow N\Rightarrow K_m*L_m\rightarrow M*N$ as $m\rightarrow\infty$.

7. ($GL(n)$ covariance) $\phi(K*L)=\phi K*\phi L$ for all $\phi\in GL(n)$.

8. (Projection covariance) $(K*L)|S=(K|S)*(L|S)$ for all $S\in {\mathcal{G}}(n,k)$, $1\le k\le n-1$.

9. (Section covariance) $(K*L)\cap S=(K\cap S)*(L\cap S)$ for all $S\in {\mathcal{G}}(n,k)$, $1\le k\le n-1$.

10. (Monotonicity) $K\subset M, L\subset N \Rightarrow K*L\subset M*N$.

11. (Cancellation) $K*M=L*M\Rightarrow K=L$ and $M*K=M*L \Rightarrow K=L$.

12. (Valuation)  $K\cup L, K\cap L\in {\mathcal{C}}\Rightarrow (K\cup L)*(K\cap L)=K*L$.

13.  (Polynomial volume) ${{\mathcal{H}}^n}(rK*sL)=\sum_{i,j=0}^{m(K,L)} a_{ij}(K,L)r^{i}s^j$, for some real coefficients $a_{ij}(K,L)$ and $m(K,L)\in\N\cup\{0\}$.

\medskip

Of course, continuity (Property~6) is with respect to some suitable metric. Throughout the paper, we shall use the Hausdorff metric when ${\mathcal{D}}\subset{\mathcal{K}}^n$ and otherwise, if  ${\mathcal{D}}\subset{\mathcal{S}}^n$, the radial metric.

In the definitions of projection and section covariance, the stated property is to hold for all $1\le k\le n-1$.  However, the proofs of our results never require $k>4$, and often $k=1$ or $k=1,2$ suffices.

Note that these properties are not independent.  For example, Property~4 implies Property~5, and the following lemma implies that for compact convex sets,  Property 8 follows from Properties~6 and~7.

\begin{lem}\label{CGimpliesP}
Let ${\mathcal{C}}\subset{\mathcal{K}}^n$ be closed under the action of $GL(n)$ and the taking of Hausdorff limits. If $*:{\mathcal{C}}^2\rightarrow {\mathcal{K}}^n$ is continuous and $GL(n)$ covariant, then it is also projection covariant.
\end{lem}

\begin{proof}
Let $S$ be a proper subspace of $\R^n$ and let $\phi$ denote projection onto $S$. Let $(\phi_m)$, $m\in \N$, be a sequence of transformations in $GL(n)$ that converges (in the sense of convergence of $n\times n$ matrices) to $\phi$.  If $K\in {\mathcal{K}}^n$, we claim that $\phi_m K\rightarrow \phi K$ as $m\rightarrow\infty$, in the Hausdorff metric.  To see this, let $u\in S^{n-1}$.  Then, using (\ref{suppchange}) and the continuity of support functions, we have
\begin{eqnarray*}
\lim_{m\rightarrow\infty}|h_{\phi_m K}(u)-h_{\phi K}(u)|&=&
\lim_{m\rightarrow\infty}|h_{K}(\phi_m^t u)-h_{K}(\phi^t u)|=\left|h_{K}(\lim_{m\rightarrow\infty}\phi_m^t u)-h_{K}(\phi^t u)\right|\\
&=& |h_{K}(\phi^t u)-h_{K}(\phi^t u)|=0.
\end{eqnarray*}
The convergence is uniform in $u\in S^{n-1}$ by \cite[Theorem~1.8.12]{Sch93}, so
$$\lim_{m\rightarrow\infty}\|h_{\phi_m K}-h_{\phi K}\|_{\infty}=0,$$
which in view of (\ref{HD}) proves the claim.

For $K, L\in {\mathcal{C}}$, we now have
$$\phi(K*L)=\lim_{m\rightarrow\infty}\phi_m(K*L)
=\lim_{m\rightarrow\infty}(\phi_m K*\phi_m L)
=(\lim_{m\rightarrow\infty} \phi_m K)*(\lim_{m\rightarrow\infty}\phi_m L)
=\phi K*\phi L,$$
as required.
\end{proof}

We shall also consider other properties of operations that are easily stated in words, for example, rotation covariance.  Some of the above properties can be stated in different versions; for example, Property~6 is continuity in both variables separately, and one can impose instead continuity in either variable or joint continuity. Properties~5 and 11 can be stated as one-sided versions.

Various modifications of the above properties can be considered.  For example, we may impose:

3$'$. (Quasi-homogeneity) $(rK)*(rL)=g(r)(K*L)$ for some continuous function $g:[0,\infty]\rightarrow [0,\infty]$.

The following lemma relates quasi-homogeneity with homogeneity of degree 1.

\begin{lem}\label{lemthmcor9}
Suppose that $\mathcal{C}\subset \mathcal{D}$ are classes of sets in $\R^n$ such that $rB^n\in {\mathcal{C}}$ for all $r\ge 0$ and $*:{\mathcal{C}}^2\rightarrow {\mathcal{D}}$ is a quasi-homogeneous operation that satisfies either $K*\{o\}=K$ or $\{o\}*K=K$, for all $K\in {\mathcal{C}}$. Then $*$ is homogeneous of degree 1.
\end{lem}

\begin{proof}
Suppose that $*$ is quasi-homogeneous, for some continuous function $g:[0,\infty]\rightarrow [0,\infty]$.  Suppose that $K*\{o\}=K$ for all $K\in {\mathcal{C}}$ (the proof for the case when $\{o\}*K=K$ is similar).  Then for $r\ge 0$, we have
$$rB^n=(rB^n)\ast \{o\}=(rB^n)\ast \{ro\}=g(r)(B^n\ast \{o\})=g(r)B^n.$$
Thus $g(r)=r$ for $r\ge 0$ and so $*$ is homogeneous of degree 1.
\end{proof}

Another natural modification is:

9$'$. (Affine section covariance) $(K*L)\cap S=(K\cap S)*(L\cap S)$ for all $S\in {\mathcal{A}}(n,k)$, $1\le k\le n-1$, where ${\mathcal{A}}(n,k)$ denotes the set of $k$-dimensional planes in $\R^n$.

However, we shall not find use for Property~9$'$ in this paper.

Analogous properties will be considered of $o$-symmetrizations $\di:{\mathcal{C}}\rightarrow{\mathcal{C}}_s$, for example:

1. (Homogeneity of degree $k$) $\di(rK)=r^k\di K$.

2. (Identity) $\di K=K$ if $K\in {\mathcal{C}}_s$.

3. (Continuity) $K_m\rightarrow K\Rightarrow \di K_m\rightarrow \di K$ as $m\rightarrow\infty$.

4. ($GL(n)$ covariance) $\phi(\di K)=\di(\phi K)$ for all $\phi\in GL(n)$.

5. (Projection covariance) $(\di K)|S=\di(K|S)$ for all $S\in {\mathcal{G}}(n,k)$, $1\le k\le n-1$.

6. (Section covariance) $(\di K)\cap S=\di(K\cap S)$ for all $S\in {\mathcal{G}}(n,k)$, $1\le k\le n-1$.

7. (Monotonicity) $K\subset L \Rightarrow \di K\subset \di L$.

We shall not find use for Property (7) in this paper.

The pertinent remarks above concerning the list of properties of binary operations apply also to these properties of $o$-symmetrizations.  In particular, the following lemma holds.

\begin{lem}\label{diCGimpliesP}
If $\di:{\mathcal{K}}^n\rightarrow {\mathcal{K}}_s^n$ is continuous and $GL(n)$ covariant, then it is also projection covariant.
\end{lem}

\begin{proof}
Let $S$ be a proper subspace of $\R^n$ and let $\phi$ denote projection onto $S$. Let $(\phi_m)$, $m\in \N$, be a sequence of transformations in $GL(n)$ that converges (in the sense of convergence of $n\times n$ matrices) to $\phi$.  If $K\in {\mathcal{K}}^n$, then as in the proof of Lemma~\ref{CGimpliesP}, we have $\phi_m K\rightarrow \phi K$ as $m\rightarrow\infty$, in the Hausdorff metric.  Therefore
$$\phi(\di K)=\lim_{m\rightarrow\infty}\phi_m(\di K)=\lim_{m\rightarrow\infty}\di(\phi_m K)=\di\left(\lim_{m\rightarrow\infty}(\phi_m K)\right)=\di(\phi K).$$
\end{proof}

\section{Examples of binary operations and $o$-symmetrizations}\label{examples}

The properties of known additions in this section are those listed in Section~\ref{properties} for operations $*:{\mathcal{C}}^2\rightarrow {\mathcal{D}}$, where $\mathcal{C}$, $\mathcal{D}$ are classes of sets in $\R^n$ with $\mathcal{C}\subset \mathcal{D}$.  It will always be assumed that $\mathcal{D}\subset {\mathcal K}^n$ or $\mathcal{D}\subset {\mathcal S}^n$, as appropriate, and that $\mathcal{C}$ is an appropriate subclass of $\mathcal{D}$.

\subsection{Minkowski addition}

The vector or Minkowski sum of sets $X$ and $Y$ in $\R^n$ is defined by
$$
X+Y=\{x+y: x\in X, y\in Y\}.
$$
When $K,L\in {\mathcal K}^n$, $K+L$ can be equivalently defined as the compact convex set such that
\begin{equation}\label{supph}
h_{K+L}(u)=h_K(u)+h_L(u),
\end{equation}
for all $u\in S^{n-1}$.  Minkowski addition satisfies all the 13 properties listed in Section~\ref{properties} with $\mathcal{C}={\mathcal K}^n$, except Property~(9), section covariance. (Here, and throughout this section, the homogeneity Property~(3) is with $k=1$.) Some of these are a direct consequence of (\ref{supph}) and the properties of the support function. For Properties~(12) and (13), see \cite[Lemma~3.1.1]{Sch93} and \cite[Theorem~5.1.6]{Sch93}, respectively.

\subsection{$L_p$ addition}\label{subLp}

Let $1< p\le \infty$. Firey \cite{Fir61}, \cite{Fir62} introduced the notion of what is now called the {\em $L_p$ sum} of $K,L\in {\mathcal K}^n_o$.  (The operation has also been called Firey addition, as in \cite[Section~24.6]{BZ}.)  This is the compact convex set $K+_pL$ defined by
$$
h_{K+_pL}(u)^p=h_K(u)^p+h_L(u)^p,
$$
for $u\in S^{n-1}$ and $p<\infty$, and by
$$
h_{K+_{\infty}L}(u)=\max\{h_K(u),h_L(u)\},
$$
for all $u\in S^{n-1}$. Note that $K+_{\infty}L=\conv(K\cup L)$.  The operation of $L_p$ addition satisfies the properties listed in Section~\ref{properties} with $\mathcal{C}={\mathcal K}^n_o$, except Property~(4), distributivity, and Property~(9), section covariance; Property (13) is discussed in Section~\ref{Minkowski}.

Lutwak, Yang, and Zhang \cite{LYZ} extended the previous definition for $1<p<\infty$, as follows.  Let $K$ and $L$ be arbitrary subsets of $\R^n$ and define
\begin{equation}\label{LYZLp}
K+_pL=\left\{(1-t)^{1/p'}x+t^{1/p'}y: x\in K, y\in L, 0\le t\le 1\right\},
\end{equation}
where $p'$ is the H\"{o}lder conjugate of $p$, i.e., $1/p+1/p'=1$.  In \cite{LYZ} it is shown that when $K,L\in {\mathcal K}^n_o$, this definition agrees with the previous one.

Equation (\ref{LYZLp}) makes sense for arbitrary $K, L\in {\mathcal K}^n$. However, the right-hand side of (\ref{LYZLp}) is not in general convex.  To see this, take $K=\{x\}$ and $L=\{y\}$, where $x$ and $y$ do not lie on the same line through the origin.  Then $K+_pL$ is a nonlinear curve that approaches $[x,x+y]\cup [y,x+y]$ as $p\rightarrow 1$ and $[x,y]$ as $p\rightarrow\infty$.  An important exception is given in the following theorem.

\begin{thm}\label{thmLpS}
For each $K\in {\mathcal K}^n$, the set $K+_p(-K)$ defined by (\ref{LYZLp}) with $L=-K$ is convex and hence belongs to ${\mathcal K}^n$.
\end{thm}

\begin{proof}
Let $K\in {\mathcal K}^n$ and let $K^{\dagger}=\conv\{K,o\}$.  We claim that $K+_p(-K)=K^{\dagger}+_p (-K^{\dagger})$.  Once this is proved, the result follows immediately from \cite[Lemma~1.1]{LYZ}, which states that $K+_p L\in {\mathcal K}^n_o$ whenever $K,L\in {\mathcal K}^n_o$.

To prove the claim, it suffices to show that if $x^{(1)},x^{(2)}\in K$ and
$0\le \alpha, \beta, t\le 1$, then there are $y^{(1)},y^{(2)}\in K$ and
$0\le s\le 1$ such that
$$(1-t)^{1/p'}\alpha x^{(1)}-t^{1/p'}\beta
x^{(2)}=(1-s)^{1/p'}y^{(1)}-s^{1/p'}y^{(2)}.$$
Indeed, the inclusion $K^{\dagger}+_p (-K^{\dagger})\subset K+_p(-K)$
then follows and the reverse inclusion is obvious.  We shall seek a solution to the previous equation with $y^{(1)}=(1-\theta)x^{(1)}+\theta x^{(2)}$ and $y^{(2)}=x^{(2)}$, where $0\le \theta\le 1$.  Substituting, we see that it suffices to solve the equations
\begin{equation}\label{first}
(1-s)^{1/p'}(1-\theta)=(1-t)^{1/p'}\alpha
\end{equation}
and
\begin{equation}\label{second}
(1-s)^{1/p'}\theta-s^{1/p'}=-t^{1/p'}\beta
\end{equation}
for $s$ and $\theta$.  Adding (\ref{first}) and (\ref{second}), we obtain
\begin{equation}\label{ff}
f(s)=(1-s)^{1/p'}-s^{1/p'}=(1-t)^{1/p'}\alpha-t^{1/p'}\beta=\gamma,
\end{equation}
say, where $-1\le \gamma\le 1$.  Since the function $f(s)$ is strictly decreasing for $0\le s\le 1$ with $f(0)=1$ and $f(1)=-1$, (\ref{ff}) has a solution for $s$.  Now (\ref{first}) and (\ref{second}) give
$$\theta=\frac{(1-s)^{1/p'}-(1-t)^{1/p'}\alpha}
{(1-s)^{1/p'}}=\frac{s^{1/p'}-t^{1/p'}\beta}{(1-s)^{1/p'}}.$$
In view of the previous equation, it is enough to show that either of the two numerators are nonnegative.  Suppose, on the contrary, that $t\beta^{p'}>s$ and $(1-t)\alpha^{p'}>1-s$.  These two inequalities imply that $t\beta^{p'}>1-(1-t)\alpha^{p'}$ and hence $t(\beta^{p'}-\alpha^{p'})>1-\alpha^{p'}$.  Clearly $\beta\le \alpha$ is not possible, but if $\beta>\alpha$, then
$$t>\frac{1-\alpha^{p'}}{\beta^{p'}-\alpha^{p'}}\ge 1,$$
a contradiction.
\end{proof}

Noting that $h_{K^{\dagger}}(u)=\max\{h_K(u),0\}$ for $u\in S^{n-1}$, the previous result suggests a reasonable definition of the $L_p$ sum of $K,L\in {\mathcal K}^n$ for $1\le p\le\infty$, namely via the equation
\begin{equation}\label{extendedLp}
h_{K+_pL}(u)^p=\max\{h_K(u),0\}^p+\max\{h_L(u),0\}^p,
\end{equation}
for $u\in S^{n-1}$.  Since $K=K^{\dagger}$ when $K\in {\mathcal K}^n_o$, this definition extends the original one.  It can be checked that the extended operation retains all the properties listed above for the original $L_p$ addition, except Property~(5), the identity property, which holds if and only if $o\in K$. We shall return to this extension of $L_p$ addition in Example~\ref{LPex}.

\subsection{$M$-addition}\label{Msub}

Let $M$ be an arbitrary subset of $\R^2$ and define the {\em $M$-sum} $K\oplus_M L$ of arbitrary sets $K$ and $L$ in $\R^n$ by
\begin{equation}\label{Mdef}
K\oplus_M L= \{ a x + b y : x\in K, y\in L, (a,b )\in M\} .
\end{equation}

It appears that $M$-addition was first introduced, for centrally symmetric compact convex sets $K$ and $L$ and a $1$-unconditional convex body $M$ in $\R^2$, by Protasov \cite{Pro97}, motivated by work on the joint spectral radius in the theory of normed algebras.

Note that if $M=\{(1,1)\}$, then $\oplus_M$ is ordinary vector or Minkowski addition, and if
\begin{equation}\label{Mmp}
M=\left\{(a,b)\in [0,1]^2:a^{p'}+b^{p'}=1\right\}=\left\{\left((1-t)^{1/p'},t^{1/p'}\right): 0\le t\le 1\right\},
\end{equation}
where $p>1$ and $1/p+1/p'=1$, then $\oplus_M$ is $L_p$ addition as defined in \cite{LYZ}.
The limiting case $p=1$, $p'=\infty$ gives $M=[e_1,e_1+e_2]\cup [e_2,e_1+e_2]$, and the case $p=\infty$, $p'=1$ corresponds to $M=[e_1,e_2]$ and
$$K\oplus_M L=\{(1-t)x+ty:x\in K,y\in L, 0\le t\le 1\}=\conv(K\cup L).$$
For a choice of $M$ leading to a different extension of $L_p$ addition, see Example~\ref{LPex}.

If $M$ is a compact set in $\R^2$, it follows from the definition (\ref{Mdef}) that $\oplus_M$ maps $\left({\mathcal{C}}^n\right)^2$ to ${\mathcal{C}}^n$, where ${\mathcal{C}}^n$ denotes the class of compact sets in $\R^n$.   It is easy to see that in this case $\oplus_M$ is monotonic, continuous in the Hausdorff metric, and $GL(n)$ covariant (and hence projection covariant, by Lemma~\ref{CGimpliesP}, and homogeneous of degree 1). Protasov \cite{Pro97} proved that if $M$ is a compact convex subset in $\R^2$ that is $1$-unconditional, then $\oplus_M:\left({\mathcal{K}}^n_s\right)^2\rightarrow {\mathcal{K}}^n_s$.  (This proof is omitted in the English translation.)  This and other results on $M$-addition can be found in Section~\ref{Maddition}.

\subsection{Radial and $p$th radial addition}\label{rpadd}

The {\em radial sum} $K\widetilde{+}L$ of $K, L\in {\mathcal S}^n$ can be defined either by
$$
K\widetilde{+}L=\{x\widetilde{+}y: x\in K, y\in L\},$$
where
$$
x\widetilde{+}y=\left\{\begin{array}{ll} x+y &
\mbox{if $x$, $y$, and $o$ are collinear,}\\
o & \mbox{otherwise,}
\end{array}\right.
$$
or by
$$\rho_{K\widetilde{+}L}=\rho_K+\rho_L.
$$
Radial addition satisfies all the 13 properties listed in Section~\ref{properties} with $\mathcal{C}={\mathcal S}^n$, except Property~(8), projection covariance. Many of these are a direct consequence of the previous equation and the properties of the radial function; for example, Property~(7) follows from \cite[(0.34), p.~20]{Gar06}.

More generally, for any $p>0$, we can define the {\em $p$th radial sum} $K\widetilde{+}_pL$ of $K, L\in {\mathcal S}^n$ by
\begin{equation}\label{radialp}
\rho_{K\widetilde{+}_pL}(x)^p=\rho_K(x)^p+\rho_L(x)^p,
\end{equation}
for $x\in \R^n\setminus\{o\}$. If $p<0$, we define $\rho_{K\widetilde{+}_pL}(x)$ as in (\ref{radialp}) when $\rho_K(x),\rho_L(x)>0$, and by $\rho_{K\widetilde{+}_pL}(x)=0$ otherwise.  Of course we can extend these definitions in a consistent fashion by setting $K\widetilde{+}_{-\infty}L=K\cap L$ and $K\widetilde{+}_{\infty}L=K\cup L$.
The operation of $p$th radial addition satisfies Properties (1)--(3), (6), (7), (9), (10), and (12) of Section~\ref{properties} with $\mathcal{C}={\mathcal S}^n$, and if $p>0$, Properties~(5) and (11) also hold.  Property (13) is discussed in Section~\ref{Minkowski}.

Hints as to the origins of these radial additions can be found in \cite[Note~6.1]{Gar06}.

\subsection{Polar $L_p$ addition}\label{polarlpadd}

In \cite{Fir61}, Firey defined the $p$-sum of $K,L\in {\mathcal K}^n_{oo}$ for $-\infty\le p\le -1$ to be $\left(K^o+_{-p} L^o\right)^o$. We shall call this operation from $({\mathcal K}^n_{oo})^2$ to ${\mathcal K}^n_{oo}$ {\em polar $L_p$ addition}, although in view of the relation $h_{K^o}=1/\rho_K$ (see \cite[(0.36), p.~20]{Gar06}), the polar $L_p$ sum of $K$ and $L$ is just $K\widetilde{+}_p L$.   (In \cite[Section~24.6]{BZ}, it is called Firey addition and the special case when $p=-1$ is sometimes called inverse addition, as in \cite{Roc70}.)  Since ${\mathcal K}^n_{oo}\subset {\mathcal S}^n$, polar $L_p$ addition satisfies Properties~(1)--(3), (7), and (9)--(12) of Section~\ref{properties} with $\mathcal{C}={\mathcal K}^n_{oo}$.  It also satisfies Property~(6), continuity, with respect to the Hausdorff metric, but this is lost when lower-dimensional sets are involved.

\begin{thm}\label{thmF}
If $n\ge 2$ and $-\infty\le p\le -1$, polar $L_p$ addition cannot be extended to a continuous operation $*:\left({\mathcal{K}}^n_s\right)^2\rightarrow {\mathcal{K}}^n$.
\end{thm}

\begin{proof}
Let $n\ge 2$ and $-\infty\le p\le -1$ and suppose, to the contrary, that $*:\left({\mathcal{K}}^n_s\right)^2\rightarrow {\mathcal{K}}^n$ is such a continuous extension. Let $K\neq \{o\}\subset e_n^{\perp}$ be an $o$-symmetric compact convex set, and let $\{K_m\}$ be a sequence of $o$-symmetric convex bodies such that $K_m\rightarrow K$ as $m\rightarrow\infty$.  For each $m$, we have
$$\left(K_m^o+_{-p} K_m^o\right)^o=\left(2^{-1/p}K_m^o\right)^o=2^{1/p}K_m,$$
so
\begin{equation}\label{KK}
K*K=\lim_{m\rightarrow\infty}2^{1/p}K_m=2^{1/p}K.
\end{equation}
(This holds true also for the limiting value $p=-\infty$.)
Let $\phi_{\alpha}$ denote a counterclockwise rotation by angle $\alpha$ in the $\{x_1,x_n\}$-plane and let $D$ denote the $o$-symmetric $(n-1)$-dimensional unit ball in $e_1^{\perp}$.  Let $0<\alpha<\pi/2$, let $K=[-e_1,e_1]$, and let $K_m=K+(1/m)D$ and $L_{m,\alpha}=\phi_{\alpha}K_m$, for $m\in \N$, so that $K_m$ and $L_{m,\alpha}$ are spherical cylinders such that $K_m\rightarrow K$ and $L_m\rightarrow \phi_{\alpha}K$ as $m\rightarrow\infty$. Then $K_m^o=\conv\{\pm e_1,mD\}$ and $L_{m,\alpha}^o=\phi_{\alpha} K_m^o$ are double cones such that $K_m^o+_{\infty}L_{m,\alpha}^o=\conv(K_m^o\cup L_{m,\alpha}^o)\rightarrow \R^n$ as $m\rightarrow\infty$ for any fixed $\alpha$.  Hence $(K_m^o+_{\infty}L_{m,\alpha}^o)^o\rightarrow \{o\}$ as $m\rightarrow\infty$.  Since polar $L_p$ sums decrease as $p$ increases, it follows that $(K_m^o+_{-p}L_{m,\alpha}^o)^o\rightarrow \{o\}$ as $m\rightarrow\infty$ for any fixed $\alpha$ and all $-\infty\le p\le -1$. By the continuity assumption, $K*\phi_{\alpha}K=\{o\}$.  Now letting $\alpha\rightarrow 0$, we obtain $K*K=\{o\}$, contradicting (\ref{KK}).
\end{proof}

\subsection{Blaschke addition}\label{blaadd}

Another important binary operation in convex geometry is {\em Blaschke addition} $\sharp$, defined for convex bodies $K$ and $L$ in $\R^n$ by letting $K\,\sharp\,L$ be the unique convex body with centroid at the origin such that
$$S(K\,\sharp \,L,\cdot)=S(K,\cdot)+S(L,\cdot),$$
where $S(K,\cdot)$ denotes the surface area measure of $K$.  See \cite[p.~130]{Gar06} or \cite[p.~394]{Sch93}.  As an operation between convex bodies, Blaschke addition satisfies Properties~(1)--(3), (7), and (12) of Section~\ref{properties}.  The $GL(n)$ covariance is not quite obvious, but a proof is given in \cite{GPS}.  When $n=2$, Blaschke addition is the same as Minkowski addition, up to translation, so it can be extended to a continuous operation between $o$-symmetric compact convex sets in $\R^2$ in this case.  For $n\ge 3$, such an extension does not exist, as we now prove.

\begin{thm}\label{thmB}
If $n\ge 3$, Blaschke addition cannot be extended to a continuous operation $*:\left({\mathcal{K}}^n_s\right)^2\rightarrow {\mathcal{K}}^n$.
\end{thm}

\begin{proof}
For $a>0$, let $K_a=[0,a]^{n-1}\times [-1/2,1/2]$ and $L_a=[-1/2,1/2]^{n-1}\times [0,a^{n-2}]$. Then $S(K_a,\cdot)$ consists of point masses of $a^{n-2}$ at $\pm e_1,\dots,\pm e_{n-1}$ and $a^{n-1}$ at $\pm e_n$, while $S(L_a,\cdot)$ consists of point masses of $a^{n-2}$ at $\pm e_1,\dots,\pm e_{n-1}$ and $1$ at $\pm e_n$.  Therefore $S(K_a\,\sharp\, L_a,\cdot)$ consists of point masses of $2a^{n-2}$ at $\pm e_1,\dots,\pm e_{n-1}$ and $1+a^{n-1}$ at $\pm e_n$.  It follows easily that $K_a\,\sharp\, L_a$ is the coordinate box
$$\left[-\frac{(1+a^{n-1})^{\frac{1}{n-1}}}{2},\frac{(1+a^{n-1})^
{\frac{1}{n-1}}}{2}\right]^{n-1}\times \left[-\frac{a^{n-2}}{(1+a^{n-1})^{\frac{n-2}{n-1}}},\frac{a^{n-2}}{(1+a^{n-1})^
{\frac{n-2}{n-1}}}\right].$$
As $a \to 0$, $K_a\rightarrow [-e_n/2,e_n/2]$, $L_a\rightarrow P$, and $K_a\,\sharp\, L_a\rightarrow P$, where $P$ is the $o$-symmetric $(n-1)$-dimensional coordinate cube $P$ of side length 1 contained in $e_n^{\perp}$.  Let $\phi$ be any rotation that leaves the $x_n$-axis fixed and satisfies $\phi e_i\neq e_j$ for all $i=1,\dots,n-1$ and $j=1,\dots,n$.  Then $\phi K_a\rightarrow [-e_n/2,e_n/2]$ as $a\to 0$.  Also, $S(\phi K_a,\cdot)$ consists of point masses of $a^{n-2}$ at $\pm \phi e_1,\dots,\pm \phi e_{n-1}$ and $a^{n-1}$ at $\pm e_n$, so $S(\phi K_a\,\sharp\, L_a,\cdot)$ consists of point masses of $a^{n-2}$ at $\pm e_1\dots,\pm e_{n-1},\pm \phi e_1,\dots,\pm \phi e_{n-1}$ and $1+a^{n-1}$ at $\pm e_n$.

Minkowski's existence theorem (see, for example, \cite[Theorem~A.3.2]{Gar06}) guarantees the existence of a convex polytope $J$ with centroid at the origin such that
$$
S(J,\cdot)=\sum\{\delta_x: x=\pm e_1\dots,\pm e_n,\pm \phi e_1,\dots,\pm \phi e_{n-1}\},
$$
where $\delta_x$ denotes a point mass of $1$ at $x$.  Then $J$ is an $o$-symmetric cylinder with the $x_n$-axis as its axis, and the formula for $S(J,\cdot)$ shows that the cross-section $Q=J\cap e_n^{\perp}$ has volume $1$ and $4(n-1)$ facets whose volumes are equal. Let $\psi_a$ be the linear transformation of $\R^n$ defined by
$$
\psi_a(y+s e_n)=(1+a^{n-1})^{\frac{1}{n-1}}y+
a^{n-2}(1+a^{n-1})^{-\frac{n-2}{n-1}}s e_n,
$$
for $y\in e_n^{\perp}$ and $s\in\R$.
Then $\psi_aJ$ is also an $o$-symmetric cylinder with the $x_n$-axis as its axis and cross-section $Q_a=\psi_a(J)\cap e_n^{\perp}$ of volume $1+a^{n-1}$.
Moreover, each facet of $\psi_aJ$ parallel to $e_n$ is the product of an $(n-2)$-dimensional face parallel to $e_n^{\perp}$ and an edge parallel to $e_n$. Its volume is therefore the volume of the corresponding parallel facet of $J$, which is $1$, times
$$\left(\left((1+a^{n-1})^{\frac{1}{n-1}}\right)^{n-2}\right)
\left(a^{n-2}(1+a^{n-1})^{-\frac{n-2}{n-1}}\right)=a^{n-2}.$$
Since their volumes and outer unit normals are therefore the same, $\psi_a J=\phi K_a\sharp L_a$, by the uniqueness part of Minkowski's existence theorem. Thus the formula for $\psi_a$ shows that $Q_a\to Q$ and $\phi K_a\sharp L_a\to Q$ as $a\to 0$.

Now the theorem is proved, because if $\sharp$ had a continuous extension $*$ defined on $({\mathcal{K}}_s^n)^2$, we would have $[-e_n/2,e_n/2]*P=P$, $[-e_n/2,e_n/2]*P=Q$, and $P\neq Q$.
\end{proof}

\subsection{$o$-symmetrizations}\label{osymms}

Examples of $o$-symmetrizations are the {\em central symmetral} (see \cite[p.~106]{Gar06}),
$$\Delta K=\frac{1}{2}K+\frac{1}{2}(-K),$$
for $K\in {\mathcal{K}}^n$, the {\em $p$th central symmetral},
$$\Delta_p K=\left(\frac{1}{2}K\right)+_p\left(\frac{1}{2}(-K)\right)
=\frac{1}{2}\left(K+_p(-K)\right),$$
for $K\in {\mathcal{K}}^n$ and $p\ge 1$, the {\em chordal symmetral} (see \cite[p.~196]{Gar06}),
$$\widetilde{\Delta}K=\frac{1}{2}K\widetilde{+}\frac{1}{2}(-K),$$
for $K\in {\mathcal{S}}^n$, and the {\em $p$th chordal symmetral},
$$\widetilde{\Delta}_pK=\left(\frac{1}{2}K\right)\widetilde{+}_p\left(\frac{1}{2}(-K)\right)
=\frac{1}{2}\left(K\widetilde{+}_p(-K)\right),$$
for $K\in {\mathcal{S}}^n$ and $p\neq 0$. For the latter, see \cite[p.~234]{Gar06}, where different notation is used.  Note that the fact that $\Delta_p K\in {\mathcal{K}}^n_s$ for $K\in {\mathcal{K}}^n$ follows from Theorem~\ref{thmLpS} if $L_p$ addition is defined as in \cite{LYZ}, or, equivalently, from the definition given by (\ref{extendedLp}) (see also Example~\ref{LPex}).

The $o$-symmetrizations $\Delta$ and $\Delta_p$ satisfy the $o$-symmetrization properties listed in Section~\ref{properties}, except Property~(6), section covariance, and $\widetilde{\Delta}$ and $\widetilde{\Delta}_p$ satisfy Properties (1), (4), (6), and (7), and also Property~(2) if $p>0$.  Here Property~(1) is with $k=1$.

There are many other important $o$-symmetrizations in convex geometry, for example, the projection body, intersection body, and centroid body operators, usually denoted by $\Pi K$, $IK$, and $\Gamma K$, respectively.  See, for example, \cite{Gar06}. Note that Steiner symmetrization is not an $o$-symmetrization.

\section{$M$-addition and $M$-combination}\label{Maddition}

We begin by generalizing the definition of $M$-addition in (\ref{Mdef}). Let $M$ be an arbitrary subset of $\R^m$ and define the {\em $M$-combination} $\oplus_M (K_1,K_2,\dots, K_m)$ of arbitrary sets $K_1,K_2,\dots,K_m$ in $\R^n$ by
\begin{equation}\label{Mndef}
\oplus_M (K_1,K_2,\dots, K_m)= \left\{\sum_{i=1}^ma_ix^{(i)}: x^{(i)}\in K_i, (a_1,a_2,\dots,a_m)\in M\right\}.
\end{equation}
For $m=2$ and $n\ge 2$, this becomes $M$-addition as in (\ref{Mdef}).

An equivalent definition is
$$
\oplus_M (K_1,K_2,\dots, K_m)= \cup\left\{a_1K_1+a_2K_2+\cdots +a_mK_m : (a_1,a_2,\dots,a_m)\in M\right\}.
$$

Several properties of $M$-combination, natural $m$-ary analogues of those for binary additions listed in Section~\ref{properties}, follow easily from these equivalent definitions.  The $m$-ary operation $\oplus_M$ is monotonic and $GL(n)$ covariant.  If $M$ and $K_i$, $i=1,\dots,m$, are compact, then $\oplus_M (K_1,K_2,\dots, K_m)$ is also compact.  In the various settings below, when $\oplus_M:{\mathcal{C}}^m\to {\mathcal{K}}^n$, where ${\mathcal{C}}={\mathcal{K}}^n$, ${\mathcal{K}}^n_o$, or ${\mathcal{K}}^n_s$, $\oplus_M$ is continuous in the Hausdorff metric and since it is also $GL(n)$ covariant, it is homogeneous of degree 1, rotation covariant, and, by a straightforward extension of Lemma~\ref{CGimpliesP}, projection covariant.

\begin{thm}\label{thmM1}
Let $m\ge 2$ and let $M$ be a subset of $\R^m$.

$\mathrm{(i)}$ If $m\le n$, the operation $\oplus_M$ maps $\left({\mathcal{K}}^n\right)^m$ to ${\mathcal{K}}^n$ if and only if $M\in {\mathcal{K}}^m$ and $M$ is contained in one of the $2^m$ closed orthants of $\R^m$. (The assumption $m\le n$ is needed only to conclude that $M\in {\mathcal{K}}^m$.)

$\mathrm{(ii)}$ If $M\in {\mathcal{K}}^m$, then $\oplus_M$ maps $\left({\mathcal{K}}^n_o\right)^m$ to ${\mathcal{K}}^n_o$ if and only if $M$ is contained in one of the $2^m$ closed orthants of $\R^m$.
\end{thm}

\begin{proof}
(i) Let
$$M'=\{x=(x_1,x_2,\dots,x_n)\in\R^n: (x_1,x_2,\dots,x_m)\in M, x_{m+1}=\cdots=x_n=0\}.$$
(Here we assume that $m\le n$.) If $\oplus_M$ maps $\left({\mathcal{K}}^n\right)^m$ to ${\mathcal{K}}^n$, then
$\oplus_M(\{e_1\},\{e_2\},\dots,\{e_m\})=M'\in {\mathcal{K}}^n$.  Since $M'$ is a copy of $M$ embedded into $\R^n$, we have $M\in {\mathcal{K}}^m$.

Suppose that $M$ is not contained in one of the $2^m$ closed orthants of $\R^m$.  Then there are $c,d>0$ and $i\in\{1,\dots,m\}$ such that $M|\,l_{e_i}=[-ce_i,de_i]$.  Let $K_i=\conv\{o,e_1,e_2,\dots,e_n\}$ and $K_j=\{o\}$ for $1\le j\neq i\le m$.  Then
\begin{eqnarray*}
\oplus_M (K_1,K_2,\dots, K_m)&=&\{a_ix: (a_1,\dots,a_m)\in M, x\in K_i\}\\
&=&\{a_ix: x\in K_i, -c\le a_i\le d\}=-cK_i\cup dK_i
\end{eqnarray*}
is not convex.

Conversely, suppose that $M\in {\mathcal{K}}^m$ is contained in one of the $2^m$ closed orthants of $\R^m$. Let $K_i\in{\mathcal{K}}^n$, $i=1,\dots,m$, and let $w,z\in \oplus_M (K_1,K_2,\dots, K_m)$.  Then there are $x^{(i)},y^{(i)}\in K_i$, $i=1,\dots,m$, and $(a_1,\dots,a_m), (b_1,\dots,b_m)\in M$ such that
$$
w=\sum_{i=1}^ma_ix^{(i)}\quad{\text{and}}\quad z=\sum_{i=1}^mb_iy^{(i)}.
$$
Let $0<t<1$.  We have to show that $(1-t)w+tz\in \oplus_M (K_1,K_2,\dots, K_m)$.

Our assumption on $M$ implies that $a_i$ and $b_i$ have the same sign for all $i=1,\dots,m$.  If at least one of $a_i$ and $b_i$ are nonzero for all $i=1,\dots,m$, then $(1-t)a_i +tb_i\neq 0$ for all $i=1,\dots,m$.  In this case
\begin{eqnarray*}
(1-t)w+tz&=&(1-t)\sum_{i=1}^ma_ix^{(i)}+t\sum_{i=1}^mb_iy^{(i)}=\sum_{i=1}^m \left((1-t)a_ix^{(i)}+tb_iy^{(i)}\right)\\
&=&\sum_{i=1}^m\left(((1-t)a_i +tb_i)\frac{(1-t)a_ix^{(i)}+tb_iy^{(i)}}{(1-t)a_i +tb_i}\right)\in \oplus_M (K_1,K_2,\dots, K_m),
\end{eqnarray*}
since
\begin{equation}\label{ki}
\frac{(1-t)a_ix^{(i)}+tb_iy^{(i)}}{(1-t)a_i +tb_i}=
\frac{(1-t)|a_i|}{(1-t)|a_i| +t|b_i|}x^{(i)}
+\frac{t|b_i|}{(1-t)|a_i| +t|b_i|}y^{(i)}\in K_i,
\end{equation}
for $i=1,\dots,m$, and
$$\left((1-t)a_1 +tb_1,\dots,(1-t)a_m +tb_m\right) =(1-t)(a_1,\dots,a_m) + t(b_1,\dots,b_m)\in M.$$
If $a_i=b_i=0$ for some $i$, we have $(1-t)a_i+tb_i=0$, so the point in $K_i$ in (\ref{ki}) can be replaced by an arbitrary point in $K_i$.

(ii) Let $M\in {\mathcal{K}}^m$. Obviously, if $o\in K_i$ for $i=1,\dots,m$, then $o\in \oplus_M (K_1,K_2,\dots, K_m)$. The result follows directly from this and the proof of (i).
\end{proof}

Note that if the dimension of $M$ is less than $m$, it is possible that $M$ is contained in more than one of the $2^m$ closed orthants of $\R^m$.

The fact that $\oplus_M$ is $L_p$ addition on $\left({\mathcal{K}}^n_o\right)^2$ when $M$ is given by (\ref{Mmp}) shows that in part (ii) of the previous theorem, it is not necessary that $M$ is convex.

\begin{lem}\label{lem1uncond1}
If $M\subset \R^m$ and $K_i\in {\mathcal{K}}^n_s$, $i=1,\dots,m$, then
$$\oplus_M (K_1,K_2,\dots, K_m)=\oplus_{\widehat{M}} (K_1,K_2,\dots, K_m),$$
where $\widehat{M}$ is the 1-unconditional hull of $M$ defined by (\ref{1uncond}).
\end{lem}

\begin{proof}
Since $M\subset \widehat{M}$, $\oplus_M (K_1,K_2,\dots, K_m)\subset \oplus_{\widehat{M}} (K_1,K_2,\dots, K_m)$. For the reverse inclusion, let $z\in \oplus_{\widehat{M}} (K_1,K_2,\dots, K_m)$.  Then there are $x^{(i)}\in K_i$ and $-1\le \alpha_i\le 1$, $i=1,\dots,m$, and $(a_1,\dots,a_m)\in M$ such that
$$z=\sum_{i=1}^m\alpha_ia_ix^{(i)}=\sum_{i=1}^ma_iy^{(i)},$$
where $y^{(i)}=\alpha_ix^{(i)}\in K_i$ because $K_i\in {\mathcal{K}}^n_s$.  Hence $z\in \oplus_M (K_1,K_2,\dots, K_m)$.
\end{proof}

\begin{thm}\label{thmM1symm}
Let $2\le m\le n$ and let $M$ be a compact subset of $\R^m$.  Then $\oplus_M$ maps $\left({\mathcal{K}}^n_s\right)^m$ to ${\mathcal{K}}^n_s$ if and only if the 1-unconditional hull $\widehat{M}$ of $M$ is convex.  (The assumption $m\le n$ is not needed for the ``only if" part.)
\end{thm}

\begin{proof}
Let $K_i\in{\mathcal{K}}^n_s$, $i=1,\dots,m$, and suppose that $\widehat{M}$ is convex.  Since $M$ is compact, it is easy to see that $\widehat{M}$ is also compact.  Let $M'=\widehat{M}\cap [0,\infty)^m$.  Then $\widehat{M'}=\widehat{M}$ because $\widehat{M}$ is 1-unconditional. Therefore, by Lemma~\ref{lem1uncond1}, $\oplus_M (K_1,K_2,\dots, K_m)=\oplus_{M'} (K_1,K_2,\dots, K_m)$.  By Theorem~\ref{thmM1}(i), it follows that $\oplus_M (K_1,K_2,\dots, K_m)\in {\mathcal{K}}^n$.  Also, $\oplus_M (K_1,K_2,\dots, K_m)$ is $o$-symmetric due to the $o$-symmetry of $K_i$, $i=1,\dots,m$, so $\oplus_M$ maps $\left({\mathcal{K}}^n_s\right)^m$ to ${\mathcal{K}}^n_s$.

Conversely, suppose that $\oplus_M$ maps $\left({\mathcal{K}}^n_s\right)^m$ to ${\mathcal{K}}^n_s$. Then
\begin{eqnarray*}
\lefteqn{\oplus_M ([-e_1,e_1],[-e_2,e_2],\dots, [-e_m,e_m])=}\\
&=&
\left\{\sum_{i=1}^ma_i\alpha_ie_i: (a_1,\dots,a_m)\in M, |\alpha_i|
\le 1, i=1,\dots,m\right\}\\
&=&\{(\alpha_1a_1,\dots,\alpha_ma_m,0,\dots,0): (a_1\dots,a_m)\in M, |\alpha_i|
\le 1, i=1,\dots,m\}
\end{eqnarray*}
belongs to ${\mathcal{K}}^n_s$ and is a copy of $\widehat{M}$ in $\R^n$.  Therefore $\widehat{M}$ is convex.
\end{proof}

As was mentioned in Section~\ref{Msub}, the following corollary was proved by Protasov \cite{Pro97} for $1$-unconditional $M$ and $m=2$.

\begin{cor}\label{corM1symm}
Let $m\ge 2$. If $M\in {\mathcal{K}}^m$ and $M$ is either contained in one of the $2^m$ closed orthants of $\R^m$ or $1$-unconditional, then $\oplus_M$ maps $\left({\mathcal{K}}^n_s\right)^m$ to ${\mathcal{K}}^n_s$.
\end{cor}

\begin{proof}
Either of the two hypotheses on $M$ guarantees that $\widehat{M}$ is convex, so the result follows directly from the previous theorem.  (In the case of the first hypothesis, it is also an easy consequence of Theorem~\ref{thmM1}(i).)
\end{proof}

Obviously there are nonconvex sets $M$ such that $\widehat{M}$ is convex and hence $\oplus_M$ maps $\left({\mathcal{K}}^n_s\right)^m$ to ${\mathcal{K}}^n_s$. Indeed, we have already observed in Section~\ref{Msub} that when $M$ is the nonconvex set defined by (\ref{Mmp}), the operation $\oplus_M$ is $L_p$ addition as defined in \cite{LYZ}, and as an operation on $\left({\mathcal{K}}^n_s\right)^2$, this is equivalent to $\oplus_{\widehat{M}}$, where $\widehat{M}$ is the unit ball in $l^n_{p'}$.  However, there are also compact convex sets $M$, even $o$-symmetric ones, such that $\widehat{M}$ is not convex.  A specific example is given by $M=\conv\{(2,1),(-2,-1),(-1,2),(1,-2)\}$; the 1-unconditional hull $\widehat{M}$ of this $o$-symmetric square contains the points $(\pm 2,\pm 1)$ and $(\pm 1,\pm 2)$ but not $(\pm 3/2,\pm 3/2)$, so $\widehat{M}$ is not convex.  Consequently, by Theorem~\ref{thmM1symm}, it is not true in general that $\oplus_M$ maps $\left({\mathcal{K}}^n_s\right)^m$ to ${\mathcal{K}}^n_s$ when $M\in {\mathcal{K}}^n_s$.

\begin{thm}\label{lemM2}
$\mathrm{(i)}$ Let $m\ge 2$ and let $M\in {\mathcal{K}}^m$ be contained in one of the $2^m$ closed orthants of $\R^m$.  Let $\ee_i=\pm 1$, $i=1,\dots,m$, denote the sign of the $i$th coordinate of a point in the interior of this orthant and let
$$M^+=\{(\ee_1a_1,\dots,\ee_ma_m): (a_1,\dots,a_m)\in M\}$$
be the reflection of $M$ contained in $[0,\infty)^m$.  If $K_i\in{\mathcal{K}}^n$, $i=1,\dots,m$, then
\begin{equation}\label{OMsupp}
h_{\oplus_M (K_1,K_2,\dots, K_m)}(x)=h_{M^+}\left(h_{\ee_1K_1}(x),h_{\ee_2K_2}(x),\dots,h_{\ee_mK_m}(x)\right),
\end{equation}
for all $x\in\R^n$.

$\mathrm{(ii)}$ Let $m\ge 2$ and let $M\in {\mathcal{K}}^m$ be $1$-unconditional. If $K_i\in{\mathcal{K}}^n_s$, $i=1,\dots,m$, then
$$
h_{\oplus_M (K_1,K_2,\dots, K_m)}(x)=h_{M}\left(h_{K_1}(x),h_{K_2}(x),\dots,h_{K_m}(x)\right),
$$
for all $x\in\R^n$.
\end{thm}

\begin{proof}
(i) By Theorem~\ref{thmM1}(i), $\oplus_M (K_1,K_2,\dots, K_m)\in {\mathcal{K}}^n$.
If $z\in \oplus_M (K_1,K_2,\dots, K_m)$, then by (\ref{Mndef}), $z=a_1x^{(1)}+\cdots+a_mx^{(m)}$ with $x^{(i)}\in K_i$, $i=1,\dots,m$, and $(a_1,\dots,a_m)\in M$. If $u\in S^{n-1}$, then using (\ref{suppf}), we have
\begin{eqnarray*}
z\cdot u&=&\sum_{i=1}^ma_i(x^{(i)}\cdot u)=\sum_{i=1}^m(\ee_ia_i)(\ee_i x^{(i)}\cdot u)\le \sum_{i=1}^m(\ee_ia_i)h_{\ee_iK_i}(u)\\
&=&(\ee_1a_1,\dots,\ee_ma_m)\cdot \left(h_{\ee_1K_1}(u),h_{\ee_2K_2}(u),\dots,h_{\ee_mK_m}(u)\right)\\
&\le& h_{M^+}\left(h_{\ee_1K_1}(u),h_{\ee_2K_2}(u),\dots,h_{\ee_mK_m}(u)\right).
\end{eqnarray*}
By (\ref{suppf}) again, this shows that $$
h_{\oplus_M (K_1,K_2,\dots, K_m)}(u)\le h_{M^+}\left(h_{\ee_1K_1}(u),h_{\ee_2K_2}(u),\dots,h_{\ee_mK_m}(u)\right).
$$

Now choose $(a_1(u),\dots,a_m(u))\in M$ such that $h_{M^+}\left(h_{\ee_1K_1}(u),h_{\ee_2K_2}(u),\dots,h_{\ee_mK_m}(u)\right)$ is equal to $(\ee_1a_1(u),\dots,\ee_ma_m(u))\cdot \left(h_{\ee_1K_1}(u),h_{\ee_2K_2}(u),\dots,h_{\ee_mK_m}(u)\right)$.
Choose $x^{(i)}\in K_i$ such that $\ee_i x^{(i)}\cdot u=h_{\ee_iK_i}(u)$,
for $i=1,\dots,m$.  Then $h_{M^+}\left(h_{\ee_1K_1}(u),h_{\ee_2K_2}(u),\dots,h_{\ee_mK_m}(u)\right)$ is equal to
$$
(\ee_1a_1(u),\dots,\ee_ma_m(u))\cdot \left(\ee_1 x^{(1)}\cdot  u,\dots,\ee_m x^{(m)}\cdot  u\right)=\left(\sum_{i=1}^ma_i(u)x^{(i)}\right)\cdot u\le  h_{\oplus_M (K_1,K_2,\dots, K_m)}(u).
$$
This proves the result for $x\in S^{n-1}$, and the general case follows by the homogeneity of support functions.

(ii) By Corollary~\ref{corM1symm}, $\oplus_M (K_1,K_2,\dots, K_m)\in {\mathcal{K}}^n_s$. If $z\in \oplus_M (K_1,K_2,\dots, K_m)$, then by (\ref{Mndef}), $z=a_1x^{(1)}+\cdots+a_mx^{(m)}$ with $x^{(i)}\in K_i$, $i=1,\dots,m$, and $(a_1,\dots,a_m)\in M$. If $u\in S^{n-1}$, then using (\ref{suppf}) and $(|a_1|,\dots,|a_m|)\in M$, we have
\begin{eqnarray*}
z\cdot u&=&\sum_{i=1}^ma_i(x^{(i)}\cdot u)\le \sum_{i=1}^m|a_i|h_{K_i}(u)=(|a_1|,\dots,|a_m|)\cdot \left(h_{K_1}(u),h_{K_2}(u),\dots,h_{K_m}(u)\right)\\
&\le& h_{M}\left(h_{K_1}(u),h_{K_2}(u),\dots,h_{K_m}(u)\right).
\end{eqnarray*}
By (\ref{suppf}) again, this shows that $$
h_{\oplus_M (K_1,K_2,\dots, K_m)}(u)\le h_{M}\left(h_{K_1}(u),h_{K_2}(u),\dots,h_{K_m}(u)\right).
$$
The rest of the proof is as in part (i) with $M^+=M$ and $\ee_i=1$ for $i=1,\dots,m$.
\end{proof}

\begin{cor}\label{corthmM1}
Let $m\ge 2$ and let $M$ be a compact convex set in $\R^m$.  Then
\begin{equation}\label{FF}
F(x)=h_{M}\left(h_{K_1}(x),h_{K_2}(x),\dots,h_{K_m}(x)\right),
\end{equation}
for $x\in \R^n$, is a support function whenever $K_i\in {\mathcal{K}}^n$, $i=1,\dots,m$, if and only if $M\subset [0,\infty)^m$.
\end{cor}

\begin{proof}
Let $M\subset [0,\infty)^m$ be a compact convex set and let $K_i\in {\mathcal{K}}^n$, $i=1,\dots,m$.  By Theorem~\ref{thmM1}(i), $\oplus_M(K_1,\dots,K_m)\in {\mathcal{K}}^n$. Then the fact that $F(x)$ is a support function follows from Theorem~\ref{lemM2}(i), since in (\ref{OMsupp}) we have $M^+=M$ and $\ee_i=1$ for $i=1,\dots,m$.

For the converse, suppose that $F(x)$ is a support function whenever $K_i\in {\mathcal{K}}^n$, $i=1,\dots,m$. Fix $j\in \{1,\dots,m\}$ and let $K_j=\conv\{-e_1,-e_2\}$ and $K_i=\{o\}$ for $1\le i\neq j\le m$.  Then $h_{K_j}(e_1)=h_{K_j}(e_2)=0$, $h_{K_j}(e_1+e_2)=-1$, and of course $h_{K_i}\equiv 0$ for $1\le i\neq j\le m$.  Therefore
$$F(e_k)=h_{M}\left(h_{K_1}(e_k),\dots,h_{K_m}(e_k)\right)=h_M(0,\dots,0)=0,$$
for $k=1,2$, and
$$F(e_1+e_2)=h_{M}\left(h_{K_1}(e_1+e_2),\dots,h_{K_m}(e_1+e_2)\right)=h_{M}(-e_j).$$
The subadditivity of $F(x)$, with $x=e_1$, $e_2$, and $e_1+e_2$, now gives $h_M(-e_j)\le 0+0=0$.  Since this holds for each $j=1,\dots,m$, we have $M\subset [0,\infty)^m$.
\end{proof}

\begin{ex}\label{LPex}
{\em Let $1<p\le \infty$ and let
$$
M=\left\{(x_1,x_2)\in [0,1]^2:x_1^{p'}+x_2^{p'}\le 1\right\},
$$
where $p>1$ and $1/p+1/p'=1$.  Then $M$ is the part of the unit ball in $l^2_{p'}$ contained in the closed positive quadrant.  The support function of $M$ is given for $x\in \R^2$ by
$$
h_M(x_1,x_2)=\begin{cases}
\left(\max\{x_1,0\}^p+\max\{x_2,0\}^p\right)^{1/p},& {\text{ if $1<p<\infty$,}}\\
\max\left\{\max\{x_1,0\},\max\{x_2,0\}\right\},& {\text{ if $p=\infty$}}.
\end{cases}
$$
By Theorem~\ref{lemM2} and Corollary~\ref{corthmM1}, for this choice of $M$, we have, for all $K, L\in {\mathcal{K}}^n$ and $x\in \R^n$,
$$
h_{K\oplus_M L}(x)=h_M\left(h_K(x),h_L(x)\right)=\begin{cases}
\left(\max\{h_K(x),0\}^p+\max\{h_L(x),0\}^p\right)^{1/p},& {\text{ if $1<p<\infty$,}}\\
\max\left\{\max\{h_K(x),0\},\max\{h_L(x),0\}\right\},& {\text{ if $p=\infty$}}.
\end{cases}
$$
Thus we have retrieved the extension of $L_p$ addition to sets in ${\mathcal{K}}^n$ via (\ref{extendedLp}) proposed earlier.  In general, this extension is different from that in \cite{LYZ}, but by Theorem~\ref{thmLpS} and its proof, the two extensions coincide in the special case when $L=-K$.  In particular, the $o$-symmetral $\triangle_pK$ can be defined via either extension of $L_p$ addition in an unambiguous way.}
\end{ex}

The situation is different for sets in the class ${\mathcal{K}}^n_o$.  Indeed,
suppose that $F$ as defined in (\ref{FF}) is a support function whenever $K_i\in {\mathcal{K}}^n_o$, $i=1,\dots,m$.  Suppose that $h_M(a)\le 0$ for some $a=(a_1,\dots,a_m)\in (0,\infty)^m$, and let $K_i=a_iB^n$, $i=1,\dots,m$.  Then
$h_{K_i}(u)=a_i$ for all $u\in S^{n-1}$ and $i=1,\dots,m$.  Consequently, $F(u)=h_M(a)\le 0$ for all $u\in S^{n-1}$, which implies that $F\equiv 0$ and hence $M=\{o\}$.  It follows that if $M\neq \{o\}$, then $M\cap [0,\infty)^m$ is a nontrivial compact convex set.  However, it is not necessary that $M\subset [0,\infty)^m$; for example, if $m=2$ and $M=B^2$, then
$$F(x)=\left(h_{K_1}(x)^2+h_{K_2}(x)^2\right)^{1/2},$$
for all $x\in \R^n$, so $F$ is the support function of the $L_2$ sum $K_1+_2K_2$ for all $K_1,K_2\in {\mathcal{K}}^n_o$.

Next, we seek a version of Corollary~\ref{corthmM1} for sets in the class ${\mathcal{K}}^n_s$.

\begin{lem}\label{thm3}
Let $m\ge 2$ and let $M\in {\mathcal{K}}^m_s$.  Then
$$
F(x)=h_{M}\left(h_{K_1}(x),h_{K_2}(x),\dots,h_{K_m}(x)\right),
$$
for $x\in \R^n$, is a support function whenever $K_i\in {\mathcal{K}}^n_s$, $i=1,\dots,m$, if and only if $h_M(s_1,\dots,s_m)$ is increasing in each variable on $[0,\infty)^m$.
\end{lem}

\begin{proof}
Suppose that $h_M(s_1,\dots,s_m)$ is increasing in each variable for $s_i\ge 0$, $i=1,\dots,m$, and let $K_i\in {\mathcal{K}}^n_s$, $i=1,\dots,m$. Clearly $F$ is homogeneous of degree 1.  Also, for $x,y\in \R^n$, the subadditivity of support functions implies that
\begin{eqnarray*}
F(x+y)&=&h_{M}\left(h_{K_1}(x+y),\dots,h_{K_m}(x+y)\right)\\
&\le &h_{M}\left(h_{K_1}(x)+h_{K_1}(y),\dots,h_{K_m}(x)+h_{K_m}(y)\right)\\
&=& h_{M}\left(\left(h_{K_1}(x),\dots,h_{K_m}(x)\right)+
\left(h_{K_1}(y),\dots,h_{K_m}(y)\right)\right)\\
&\le & h_{M}\left(h_{K_1}(x),\dots,h_{K_m}(x)\right)+
h_{M}\left(h_{K_1}(y),\dots,h_{K_m}(y)\right)=F(x)+F(y).
\end{eqnarray*}
This proves that $F$ is a support function.

Suppose that $F$ is a support function. To show that $h_M(s_1,\dots,s_{m})$ is increasing in each variable on $[0,\infty)^m$, we may, without loss of generality, prove that if $0<t_1<t_2$ and $s_i>0$, $i=1,\dots,m-1$, we have
$$
h_M(s_1,s_2\dots,s_{m-1},t_1)\le h_M(s_1,s_2,\dots,s_{m-1},t_2).
$$
In fact, it suffices to assume also that $t_2<2t_1$, since the general case then follows by iteration.  Let $K_i=[-s_ie_1,s_ie_1]$ for $i=1,\dots,m-1$, let $x=(1,1,0\dots,0)\in \R^n$, and let $y=(1,-1,0,\dots,0)\in \R^n$.  Note that $h_{K_i}(x)=h_{K_i}(y)=s_i$ and $h_{K_i}(x+y)=2s_i$, $i=1,\dots,m-1$.
Choose $K_m\in {\mathcal{K}}^n_s$ such that
$h_{K_m}(x)=h_{K_m}(y)=t_2$ and
$h_{K_m}(x+y)=2t_1$. (This is possible since $0<t_1<t_2<2t_1$.)
Then, since $F$ is subadditive, we have
\begin{eqnarray*}
2h_M\left(s_1,s_2,\dots,s_{m-1},t_1\right)
&=&h_M\left(2s_1,2s_2,\dots,2s_{m-1},2t_1\right)\\
&=&h_{M}\left(h_{K_1}(x+y),\dots,h_{K_m}(x+y)\right)\\
&=&F(x+y)\le F(x)+F(y)\\
&=&h_{M}\left(h_{K_1}(x),\dots,h_{K_m}(x)\right)+
h_{M}\left(h_{K_1}(y),\dots,h_{K_m}(y)\right)\\
&=&2h_M\left(s_1,s_2,\dots,s_{m-1},t_2\right),
\end{eqnarray*}
as required.
\end{proof}

Note that it is not necessary that $h_M$ is strictly increasing in each variable.  This is shown by taking $M=\conv\{\pm e_1, \pm e_2,\dots,\pm e_m\}$ (i.e., the unit ball in $l^m_1$), in which case $h_M(s_1,\dots,s_m)=\max\{s_1,\dots,s_m\}$, for $s_1,\dots,s_m\ge 0$.

If $M\in \K^m$ is $1$-unconditional, then $M$ is $o$-symmetric and it follows from Theorem~\ref{lemM2}(ii) and Lemma~\ref{thm3} that $h_M$ is increasing in each variable on $[0,\infty)^m$.  When $m=2$, the next lemma provides a sort of converse statement.

\begin{lem}\label{madd2}
Suppose that $M\in \K_o^2$ is such that $h_M(s,t)$ is increasing in each variable for $s,t\ge 0$.  Then there is an $M'\in \K^2$ that is $1$-unconditional and such that $h_{M'}(s,t)=h_M(s,t)$ for $s,t\ge 0$.
\end{lem}

\begin{proof}
Suppose that $M\in \K_o^2$ is such that $h_M(s,t)$ is increasing in each variable for $s,t\ge 0$.  We claim that among the points in $M$ with the greatest $x_2$-coordinate, there is one contained in $[0,\infty)^2$.  If this is not the case, let $(-a,b)$, $a,b>0$, be the one with the greatest $x_1$-coordinate. Let
$s,t>0$ be such that the tangent line to $M$ orthogonal to $(s,t)$ meets $\partial M$ at $(-c,d)$, where $0<c\le a$ and $0\le d\le b$.  Then
$$h_M(s,t)=(-c,d)\cdot (s,t)=-cs+dt<bt=h_M(0,t),$$
which contradicts the fact that $h_M(s,t)$ is increasing in $s$ for $s\ge 0$.  This proves the claim.  Similarly, one can show that among the points in $M$ with the greatest $x_1$-coordinate, there is one contained in $[0,\infty)^2$.

Let $z$ be the point in $M\cap [0,\infty)^2$ with the greatest $x_2$-coordinate and such that the (possibly degenerate) horizontal line segment $T$ with left endpoint on the $x_2$-axis and right endpoint equal to $z$ has relative interior disjoint from $M$. Similarly, let $z'$ be the point in $M\cap [0,\infty)^2$ with the greatest $x_1$-coordinate and such that the (possibly degenerate) vertical line segment $T'$ with lower endpoint on the $x_1$-axis and upper endpoint equal to $z'$ has relative interior disjoint from $M$.  Let $M'$ be the compact convex set such that $M'\cap [0,\infty)^2=\conv\{M\cap [0,\infty)^2, T,T'\}$ and $M'$ is $1$-unconditional.  By this construction, $h_{M'}(s,t)=h_M(s,t)$ for all $s,t\ge 0$.
\end{proof}

Of special interest in Section~\ref{Symmetrization} is the behavior of the $M$-sum of a compact convex set $K$ in $\R^n$ and its reflection $-K$, particularly when $M\in {\mathcal{K}}^2$ is symmetric in the line $x_1=x_2$.  By Theorem~\ref{thmM1}(i), $K\oplus_M (-K)$ is convex and hence easily seen to be in ${\mathcal{K}}_s^n$, when $M\subset [0,\infty)^2$, and then Theorem~\ref{lemM2}(i) implies that $h_M(h_K(x),h_{-K}(x))$, $x\in \R^n$, is the support function of $K\oplus_M(-K)$.  It is natural to ask if these observations remain true when $M\in {\mathcal{K}}_s^2$.  The following example shows that this is not the case, even if $M$ is 1-unconditional and symmetric in the line $x_1=x_2$ and $o\in K$.

Let $M=[-1,1]^2$ and let $K=[0,1]^2$.  Then $(0,2)=1(0,1)+(-1)(0,-1)$ and $(-1,1)=(-1)(1,0)+(-1)(0,-1)$ both belong to $K\oplus_M(-K)$.  We claim that $(-1/2,3/2)\not\in K\oplus_M(-K)$.  Indeed, otherwise there would be $0\le x_1,x_2,y_1,y_2\le 1$ and $-1\le a,b\le 1$ such that $-1/2=ax_1-by_1$ and $3/2=ax_2-by_2$.  For the first equation to hold, we need $a<0$ or $b>0$, or both, but then the second equation cannot hold.  This proves the claim and shows that $K\oplus_M(-K)$ is not convex.

\section{Operations between $o$-symmetric compact convex or star sets}\label{results}

Here we focus on projection covariant operations $*:\left({\mathcal{K}}^n_s\right)^2\rightarrow {\mathcal{K}}^n$ or section covariant operations $*:\left({\mathcal{S}}^n_s\right)^2\rightarrow {\mathcal{S}}^n$.  We remark at the outset that in this case, the distributivity property $(rK)*(sK)=(r+s)K$ is too strongly tied to Minkowski or radial addition for a nontrivial classification theorem.  Indeed, we have the following easy result.

\begin{thm}\label{thm00}
Let $*:\left({\mathcal{K}}^n_s\right)^2\rightarrow {\mathcal{K}}^n$ (or $*:\left({\mathcal{S}}^n_s\right)^2\rightarrow {\mathcal{S}}^n$) be projection covariant (or section covariant, respectively) and have the distributivity property.  Then $K*L=K+L$ for all $K,L\in {\mathcal{K}}_s^n$ (or $K*L=K\widetilde{+}L$ for all $K,L\in {\mathcal{S}}_s^n$, respectively).
\end{thm}

\begin{proof}
Suppose that $K,L\in {\mathcal{K}}_s^n$ and $u\in S^{n-1}$, and recall that $l_u$ denotes the line through the origin parallel to $u$. Then
\begin{eqnarray*}
(K*L)|\,l_u&=&(K|\,l_u)*(L|\,l_u)=[-ru,ru]*[-su,su]\\
&=&(r[-u,u])*(s[-u,u])=
(r+s)[-u,u]=[-(r+s)u,(r+s)u]\\
&=&[-ru,ru]+[-su,su]=(K|\,l_u)+(L|\,l_u)=(K+L)|\,l_u,
\end{eqnarray*}
for suitable $r,s\ge 0$ (depending on $u$) and all $u\in S^{n-1}$.  But this implies that $h_{K*L}(u)=h_{K+L}(u)$ for all $u\in S^{n-1}$ and hence $K*L=K+L$.  The other case follows in a similar fashion.
\end{proof}

The next two results can be stated in different versions, according to the corresponding versions of the results in Section~\ref{background}.  In both results, we interpret $L_p$ addition for $-\infty\le p<1$, $p\neq 0$, via (\ref{radialp}) and the remarks following it, since when $n=1$, the classes of $o$-symmetric compact convex sets and  $o$-symmetric star sets coincide.

\begin{lem}\label{lem0}
Let $*:\left({\mathcal{K}}^1_s\right)^2\rightarrow {\mathcal{K}}_s^1$ be continuous, homogeneous of degree 1, and associative.  Then either $K*L=\{o\}$, or $K*L=K$, or $K*L=L$, for all $K,L\in {\mathcal{K}}^1_s$, or else  $*=+_p$, i.e., the operation is $L_p$ addition, for some $-\infty\le p\neq 0\le \infty$.
\end{lem}

\begin{proof}
Suppose that for each $s,t\ge 0$,
$$[-s,s]*[-t,t]=[-f(s,t), f(s,t)].$$
Then the function $f:[0,\infty)^2\rightarrow [0,\infty)$ satisfies the hypotheses of Proposition~\ref{prp2}, and the result follows immediately.
\end{proof}

\begin{cor}\label{thmcor9}
Let $*:\left({\mathcal{K}}^1_s\right)^2\rightarrow {\mathcal{K}}_s^1$ be a continuous, quasi-homogeneous, and associative operation that has the identity property. Then $*=+_p$ for some $0< p\le\infty$.
\end{cor}

\begin{proof}
This follows directly from Lemmas~\ref{lemthmcor9} and~\ref{lem0}, bearing in mind that the identity property fails in general when $-\infty\le p<0$.
\end{proof}

To see that none of the assumptions of Lemma~\ref{lem0} nor the first three assumptions of Corollary~\ref{thmcor9} can be omitted, define $*$ by taking the function $f$ in the proof of Lemma~\ref{lem0} to be $f_1$, $f_2$, or $f_5$ as defined after Proposition~\ref{prp1}.
The operation $*$ defined by
$$[-s,s]*[-t,t]=[-st,st],$$
for $s,t\ge 0$ shows that the identity property also cannot be omitted in Corollary~\ref{thmcor9}.

\begin{lem}\label{lem01}
If $n\ge 2$, then $*:\left({\mathcal{K}}^n_s\right)^2\rightarrow {\mathcal{K}}^n$ is projection covariant if and only if there is a homogeneous of degree 1 function $f:[0,\infty)^2\to[0,\infty)$ such that
\begin{equation}\label{441}
h_{K*L}(x)=f\left(h_K(x),h_L(x)\right),
\end{equation}
for all $K,L\in {\mathcal{K}}^n_s$ and all $x\in \R^n$.
\end{lem}

\begin{proof}
Suppose that $*:\left({\mathcal{K}}^n_s\right)^2\rightarrow {\mathcal{K}}^n$ is projection covariant and let $u\in S^{n-1}$. Then for any two $o$-symmetric compact convex sets $K$ and $L$ in $\R^n$, we have
\begin{equation}\label{eqb}
(K*L)|\,l_u=(K|\,l_u)*(L|\,l_u).
\end{equation}
One consequence of this is that if $I$ and $J$ are $o$-symmetric closed intervals in $l_u$, we must have $I*J\subset l_u$. Hence there are functions $f_u, g_u:[0,\infty)^2\to\R$ such that $-g_u\le f_u$ and
such that
\begin{equation}\label{eqn1}
[-su,su]*[-tu,tu]=[-g_u(s,t)u,f_u(s,t)u],
\end{equation}
for all $s,t\ge 0$.

Let $0\le \alpha\le 1$ and choose $v\in S^{n-1}$ such that $u\cdot v=\alpha$.  Using (\ref{eqb}) with $K=[-su,su]$, $L=[-tu,tu]$, and $l_u$ replaced by $l_v$, and (\ref{eqn1}), we obtain
\begin{eqnarray*}
\alpha[-g_{u}(s,t)v,f_{u}(s,t)v]&=&[-g_{u}(s,t)u,f_{u}(s,t)u]|\,l_v\\
&=&([-su,su]*[-tu,tu])|\,l_v=([-su,su]|\,l_v)*([-tu,tu]|\,l_v)\\
&=&[-\alpha sv,\alpha sv]*[-\alpha tv,\alpha tv]=
[-g_{v}(\alpha s,\alpha t)v,f_{v}(\alpha s,\alpha t)v],
\end{eqnarray*}
for all $s,t\ge 0$.  Therefore
\begin{equation}\label{fandg}
f_v(\alpha s,\alpha t)=\alpha f_u(s,t)\quad{\text{and}}\quad g_v(\alpha s,\alpha t)=\alpha g_u(s,t),
\end{equation}
for all $s,t\ge 0$.  Interchanging $u$ and $v$ in the first equation in (\ref{fandg}), we also have
\begin{equation}\label{alphav}
f_u(\alpha s,\alpha t)=\alpha f_v(s,t)
\end{equation}
and hence
$$f_u(\alpha^2 s,\alpha^2 t)=\alpha f_v(\alpha s,\alpha t)=\alpha^2f_u(s,t),$$
for all $s,t\ge 0$.  Setting $r=\alpha^2$, we get
\begin{equation}\label{rrr}
f_u(rs,rt)=rf_u(s,t),
\end{equation}
for $0\le r\le 1$ and $s,t\ge 0$.  Replacing $s$ and $t$ by $s/r$ and $t/r$, respectively, yields
\begin{equation}\label{1r}
f_u(s/r,t/r)=(1/r)f_u(s,t),
\end{equation}
for $0<r\le 1$ and $s,t\ge 0$.  From (\ref{rrr}) and (\ref{1r}), it follows that $f_u$ is homogeneous of degree 1.

Now fix $u\in S^{n-1}$. Let $v\in S^{n-1}$ be such that $u\cdot v>0$ and choose $0<\alpha <1$ such that $\alpha=u\cdot v$.  Then from (\ref{alphav}) and the homogeneity of $f_u$, we obtain
$$\alpha f_v(s,t)=f_u(\alpha s,\alpha t)=\alpha f_u(s,t),$$
for all $s,t\ge 0$.  This shows that $f_v=f_u$ for all such $v$ and consequently $f_u=f$, say, is independent of $u$.

Applying the same arguments to the second equation in (\ref{fandg}), we see that $g_u=g$, say, is also homogeneous of degree 1 and independent of $u$. Now from (\ref{eqb}) and (\ref{eqn1}) with $f_u=f$ and $g_u=g$, we obtain
\begin{eqnarray*}
[-h_{K*L}(-u)u,h_{K*L}(u)u]&=&(K*L)|\,l_u=(K|\,l_u)*(L|\,l_u)\\
&=&[-h_{K}(u)u,h_{K}(u)u]*[-h_{L}(u)u,h_{L}(u)u]\\
&=&[-g\left(h_K(u),h_L(u)\right)u,f\left(h_K(u),h_L(u)\right)u].
\end{eqnarray*}
Comparing the second coordinates in the previous equation, we conclude that
\begin{equation}\label{441a}
h_{K*L}(u)=f\left(h_K(u),h_L(u)\right),
\end{equation}
for all $u\in S^{n-1}$. (Note that in view of the equality of the first coordinates and the $o$-symmetry of $K$ and $L$, we must in fact have $g(s,t)=f(s,t)\ge 0$ for all $s,t\ge 0$.)

Let $r\ge 0$.  Then (\ref{441a}) and the homogeneity of support functions imply that
$$h_{K*L}(ru)=rh_{K*L}(u)=
rf\left(h_{K}(u),h_{L}(u)\right)
=f\left(rh_{K}(u),rh_{L}(u)\right)
=f\left(h_{K}(ru),h_{L}(ru)\right),$$
and (\ref{441}) follows.

For the converse, let $S\in {\mathcal{G}}(n,k)$, $1\le k\le n-1$, be a subspace and let $x\in \R^n$.  Using (\ref{hproj}) and (\ref{441}), we obtain
$$h_{(K*L)|S}(x)=h_{K*L}(x|S)=f\left(h_K(x|S),h_L(x|S)\right)=
f\left(h_{K|S}(x),h_{L|S}(x)\right)=h_{(K|S)*(L|S)}(x),$$
establishing the projection covariance of $*$.
\end{proof}

In view of (\ref{441}), we obtain the following corollary.

\begin{cor}\label{corsymm}
If $n\ge 2$ and $*:\left({\mathcal{K}}^n_s\right)^2\rightarrow {\mathcal{K}}^n$ is projection covariant, then in fact $*:\left({\mathcal{K}}^n_s\right)^2\rightarrow {\mathcal{K}}_s^n$.
\end{cor}

Another easy consequence of Lemma~\ref{lem01} is that if $n\ge 2$, then any projection covariant operation $*:\left({\mathcal{K}}^n_s\right)^2\rightarrow {\mathcal{K}}^n$ must be both homogeneous of degree 1 and rotation covariant.  However, an even stronger conclusion will be drawn in Corollary~\ref{contaff}.

\begin{thm}\label{thm12}
Let $n\ge 2$. An operation $*:\left({\mathcal{K}}^n_s\right)^2\rightarrow {\mathcal{K}}^n$ is projection covariant if and only if it can be defined by
\begin{equation}\label{nn}
h_{K*L}(x)=h_{M}\left(h_K(x),h_L(x)\right),
\end{equation}
for all $K,L\in {\mathcal{K}}^n_s$ and $x\in\R^n$, or equivalently by
\begin{equation}\label{Meq}
K*L=K\oplus_M L,
\end{equation}
where $M$ is a $1$-unconditional compact convex set in $\R^2$.  Moreover, $M$ is uniquely determined by $*$.
\end{thm}

\begin{proof}
By Lemma~\ref{lem01}, an operation defined by (\ref{nn}) is projection covariant.

Let $*:\left({\mathcal{K}}^n_s\right)^2\rightarrow {\mathcal{K}}^n$ be projection covariant. By Lemma~\ref{lem01}, (\ref{441}) holds for some homogeneous of degree 1 function $f:[0,\infty)^2\to[0,\infty)$.

Let $K_0=[-e_1,e_1]$, $L_0=[-e_2,e_2]$, and $S=\lin\{e_1,e_2\}$.  From the projection covariance of $*$, we have
$$(K_0*L_0)|S=(K_0|S)*(L_0|S)=K_0*L_0,$$
so $K_0*L_0\subset S$.  Identifying $S$ with $\R^2$ in the natural way, we let $M=K_0*L_0$.  Then for $x=(x_1,\dots,x_n)\in \R^n$ with $x_1,x_2\ge 0$, (\ref{441}) with $K=K_0$ and $L=L_0$ yields
\begin{equation}\label{Mf}
h_M(x_1,x_2)=h_{K_0*L_0}(x)=f(h_{K_0}(x),h_{L_0}(x))=f(|x_1|,|x_2|)=f(x_1,x_2).
\end{equation}
Since $h_K(x),h_L(x)\ge 0$ for all $K,L\in {\mathcal{K}}^n_s$ and all $x\in\R^n$, (\ref{nn}) follows directly from (\ref{441}) and (\ref{Mf}).  Moreover, from its definition and Corollary~\ref{corsymm}, $M$ is an $o$-symmetric compact convex set.

By (\ref{nn}) and Lemma~\ref{thm3} with $m=2$, $h_M(s,t)$ is increasing in each variable for $s,t\ge 0$.  By Lemma~\ref{madd2}, there exists a 1-unconditional set $M'\in \mathcal{K}^2$ such that $h_{M'}=h_M$ on $[0,\infty)^2$. Hence we can assume that $M$ has this property. Now (\ref{Meq}) follows from (\ref{nn}) via Theorem~\ref{lemM2}(ii) with $m=2$. The equivalence of (\ref{nn}) and (\ref{Meq}) is also a consequence of Theorem~\ref{lemM2}(ii) with $m=2$.

Let $x_1,x_2\ge 0$. Then (\ref{nn}) with $K=x_1B^n$ and $L=x_2B^n$ yields
$h_{K*L}(u)=h_M(x_1,x_2)$, for $u\in S^{n-1}$. This shows that $M\cap [0,\infty)^2$, and since $M$ is 1-unconditional, $M$ itself, is uniquely determined by the operation $*$.
\end{proof}

\begin{cor}\label{contaff}
Let $n\ge 2$. An operation $*:\left({\mathcal{K}}^n_s\right)^2\rightarrow {\mathcal{K}}^n$  is projection covariant if and only if it is continuous and $GL(n)$ covariant (and hence homogeneous of degree 1).
\end{cor}

\begin{proof}
If $*$ is continuous and $GL(n)$ covariant, then it is projection covariant by Lemma~\ref{CGimpliesP} and homogeneous of degree 1.  Since $\oplus_M:\left({\mathcal{K}}^n_s\right)^2\rightarrow {\mathcal{K}}_s^n$ is continuous and $GL(n)$ covariant, the converse follows from Theorem~\ref{thm12}.
\end{proof}

\begin{cor}\label{BlaPol}
Neither polar $L_p$ addition, for $n\ge 2$ and $-\infty\le p\le -1$, nor Blaschke addition, for $n\ge 3$, can be extended to a projection covariant operation $*:\left({\mathcal{K}}^n_s\right)^2\rightarrow {\mathcal{K}}^n$.
\end{cor}

\begin{proof}
This is an immediate consequence of Theorems~\ref{thmF} and~\ref{thmB} and Corollary~\ref{contaff}.
\end{proof}

The special case of the next theorem when $*=\oplus_M$ for some $1$-unconditional compact convex set $M$ in $\R^2$ was proved earlier by Protasov \cite{Pro99}, by a fairly intricate direct argument.  Rather than appealing to \cite{Pro99}, however, we prefer to utilize the more general results of Section~\ref{background}.

\begin{thm}\label{thm1}
If $n\ge 2$, then $*:\left({\mathcal{K}}^n_s\right)^2\rightarrow {\mathcal{K}}^n$ is projection covariant and associative if and only if $*=\oplus_M$, where either $M=\{o\}$, or $M=[-e_1,e_1]$, or $M=[-e_2,e_2]$, or $M$ is the unit ball in $l^2_p$ for some $1\le p\le \infty$; in other words, if and only if either $K*L=\{o\}$, or $K*L=K$, or $K*L=L$, for all $K,L\in {\mathcal{K}}^n_s$, or else $*=+_p$ for some $1\le p\le\infty$.
\end{thm}

\begin{proof}
Note that by Corollary~\ref{corsymm}, we have $*:\left({\mathcal{K}}^n_s\right)^2\rightarrow {\mathcal{K}}_s^n$, so the associativity property makes sense. By Theorem~\ref{thm12}, (\ref{nn}) holds for some $1$-unconditional compact convex set $M$ in $\R^2$.  The support function $h_M(s,t)$ is continuous and homogeneous of degree 1. Since $*$ is associative, (\ref{nn}) implies that
$$
h_M\left(h_K(x),h_M\left(h_L(x),h_N(x)\right)\right)=h_M\left(h_M\left
(h_K(x),h_L(x)\right),h_N(x)\right),
$$
for all $K,L,N\in \K^n_s$ and $x\in \R^n$.  Setting $K=rB^n$, $L=sB^n$, and $N=tB^n$, for $r,s,t\ge 0$, we obtain
$$
h_M(r,h_M(s,t))=h_M(h_M(r,s),t),
$$
i.e., $h_M$ is associative. It follows that $h_M$ satisfies the hypotheses of Proposition \ref{prp2} and so must be of one of the forms listed there.  The case $p<1$ is excluded because $h_M$ is a support function.  The remaining possibilities are that $M=\{o\}$, or $M=[-e_1,e_1]$, or $M=[-e_2,e_2]$, or $M$ is the unit ball in $l^2_p$ for some $p\ge 1$.
\end{proof}

The three elementary exceptional cases in Theorem~\ref{thm1} can be eliminated by imposing the identity property in addition to those assumed there, i.e., projection covariance and associativity.  The following corollary assumes continuity and $GL(n)$ covariance, but by Corollary~\ref{contaff}, could equivalently be stated with the assumption of projection covariance instead.

\begin{cor}\label{cor10}
If $n\ge 2$, the operation $*:\left({\mathcal{K}}^n_s\right)^2\rightarrow {\mathcal{K}}^n$ is continuous, $GL(n)$ covariant, associative, and has the identity property if and only if $*=+_p$ for some $1\le p\le\infty$.
\end{cor}

\begin{proof}
By Lemma~\ref{CGimpliesP}, if $*$ is continuous and $GL(n)$ covariant, then it is also projection covariant and Theorem~\ref{thm1} applies.  The identity property assumption eliminates the other possibilities for $*$ in the statement of Theorem~\ref{thm1}.  The converse is clear.
\end{proof}

Various examples of operations $*:\left({\mathcal{K}}^n_s\right)^2\rightarrow {\mathcal{K}}^n$ can be obtained by defining
\begin{equation}\label{suppdef}
h_{K*L}(x)=f\left(h_K(x),h_L(x)\right),
\end{equation}
for all $K, L\in {\mathcal{K}}_s^n$ and $x\in \R^n$, where $f:[0,\infty)^2\rightarrow \R$.  However, the options for the function $f$ are already severely limited by the fact that the right-hand side of (\ref{suppdef}) must be a support function.  It is not enough that $f$ is nonnegative and homogeneous of degree 1; for example, we cannot take $f=f_5$ as defined after Proposition~\ref{prp1}.

\begin{ex}\label{ex2}
{\em The force of Theorem~\ref{thm1} is apparent when considering the possibility of taking
$$
f(s,t)=\sinh^{-1}\left(\sinh s+\sinh t\right),
$$
for $s,t\ge 0$.  Then (\ref{suppdef}) takes the form
$$h_{K*L}(x)=\sinh^{-1}\left(\sinh h_K(x)+\sinh h_L(x)\right),$$
for all $K, L\in {\mathcal{K}}_s^n$ and $x\in \R^n$.
The subadditivity of $h_{K*L}$ then follows from the subadditivity of $h_K$ and $h_L$, together with the fact that the $\sinh$ function is increasing and satisfies Mulholland's inequality (\ref{mul}).  However, $h_{K*L}$ is not homogeneous of degree 1.  This can be proved directly, but since the resulting operation would be associative and projection covariant, it is already a consequence of Theorem~\ref{thm1}.}
\end{ex}

In fact, from Lemma~\ref{lem01} and Theorem~\ref{thm12} we know that if (\ref{suppdef}) holds, where $f$ is homogeneous of degree 1, then $*=\oplus_M$ for some $1$-unconditional compact convex set $M$. The next example is of this type.

\begin{ex}\label{ex1}
{\em To show that associativity cannot be dropped in Theorem~\ref{thm1} and Corollary~\ref{cor10}, we can take in (\ref{suppdef}) the function
$$f(s,t)=f_6(s,t)=\frac12(s+t)+
\frac12\left(s^2+t^2\right)^{1/2}.$$
The function $f_6$ is homogeneous of degree 1 and the resulting operation
$*:\left({\mathcal{K}}^n_s\right)^2\rightarrow {\mathcal{K}}^n$, defined by
$$h_{K*L}(x)=\frac12\left(h_K(x)+h_L(x)\right)+\frac12\left(h_K(x)^2+h_L(x)^2\right)^{1/2}
,$$
for all $K, L\in {\mathcal{K}}_s^n$ and $x\in \R^n$,
satisfies all the other hypotheses of those results.  This corresponds to $M$-addition with $M=(1/2)[-1,1]^2+(1/2)B^2$. More generally, let $\lambda_i>0$, $i=1,\dots,m$, satisfy $\sum_{i=1}^m\lambda_i=1$.
Let $p_i\ge 1$, let $B^2_{p_i}$ be the unit ball in $l^2_{p_i}$, $i=1,\dots,m$, and define
$$M=\lambda_1B^2_{p_1}+\cdots +\lambda_mB^2_{p_m}.$$
Then the operation $\oplus_M$ is continuous, $GL(n)$ covariant, and has the identity property, but is not associative.}
\end{ex}

\begin{ex}\label{ex4}
{\em Define
$$K*L=\begin{cases}
K\cap L,& {\text{ if $K\neq \{o\}$ and $L\neq \{o\}$,}}\\
K\cup L,& {\text{ if $K=\{o\}$ or $L=\{o\}$}},
\end{cases}
$$
for all $K,L\in {\mathcal{K}}^n_s$.  This operation is $GL(n)$ covariant, associative, and has the identity property, but it is neither continuous nor projection covariant (take $K=[-e_1,e_1]$ and $L=[-e_2,e_2]$ and consider projections onto the $x_1$-axis).  This shows that projection covariance cannot be omitted in Theorem~\ref{thm1} and continuity cannot be dropped in Corollary~\ref{cor10}.}
\end{ex}

\begin{ex}\label{ex5}
{\em Let $n\ge 2$ and define
$$K*L=\left(\frac{{{\mathcal{H}}^n}(K)^{1/n}+{{\mathcal{H}}^n}(L)^{1/n}}{\kappa_n^{1/n}}\right)B^n,
$$
for all $K,L\in {\mathcal{K}}^n_s$.  This operation is continuous, homogeneous of degree 1, rotation covariant, and associative, but is not projection covariant and does not have the identity property.  This also shows that projection covariance cannot be omitted in Theorem~\ref{thm1}.}
\end{ex}

\begin{ex}\label{ex5555}
{\em Let $n\ge 2$ and let $F:{\mathcal{K}}^n_s\to {\mathcal{K}}^n_s$ be such that $F(K)$ is the set obtained by rotating $K$ by an angle equal to its volume ${{\mathcal{H}}^n}(K)$ around the origin in the $\{x_1,x_2\}$-plane.  Note that since ${{\mathcal{H}}^n}(F(K))={{\mathcal{H}}^n}(K)$, the map $F$ is injective and so $F^{-1}$ is defined.  Of course, $F^{-1}$ rotates by an angle $-{{\mathcal{H}}^n}(K)$ instead. Now define
\begin{equation}\label{eq5555}
K*L=F^{-1}\left(F(K)+F(L)\right),
\end{equation}
for all $K,L\in {\mathcal{K}}^n_s$.  It is easy to check that $*$ is continuous, associative, homogeneous of degree 1, rotation covariant, and moreover has the identity property.  However, $*$ is not projection covariant or $GL(n)$ covariant.  This is rather clear from Theorem~\ref{thm1} and Corollary~\ref{cor10}, but an explicit example can be constructed as follows.

Let $K=[-1/2,1/2]\times [-\pi^{1/(n-1)}/2, \pi^{1/(n-1)}/2]^{n-1}$ and
$L=[-(1/4)e_1,(1/4)e_1]$ be an $o$-symmetric coordinate box and line segment in the $x_1$-axis in $\R^n$, respectively.  Since ${{\mathcal{H}}^n}(K|\,l_{e_1})={{\mathcal{H}}^n}(L|\,l_{e_1})=0$, we have $F(K|\,l_{e_1})=K|\,l_{e_1}$ and $F(L|\,l_{e_1})=L|\,l_{e_1}$ and hence
$$(K|\,l_{e_1})*(L|\,l_{e_1})=(K+L)|\,l_{e_1}=[-3/4, 3/4].$$
Also, ${{\mathcal{H}}^n}(K)=\pi$ and ${{\mathcal{H}}^n}(L)=0$, so $F(K)=K$ and $F(L)=L$.  Therefore
$$F(K)+F(L)=K+L =[-3/4,3/4] \times \left[-\pi^{1/(n-1)}/2, \pi^{1/(n-1)}/2\right]^{n-1}$$
and hence ${{\mathcal{H}}^n}\left(F(K)+F(L)\right)=3\pi/2$.  Therefore
$$K*L=\left[-\frac{\pi^{1/(n-1)}}{2}, \frac{\pi^{1/(n-1)}}{2}\right]\times [-3/4,3/4] \times \left[-\frac{\pi^{1/(n-1)}}{2}, \frac{\pi^{1/(n-1)}}{2}\right]^{n-2}$$
so
$$(K*L)|\,l_{e_1}=\left[-\pi^{1/(n-1)}/2, \pi^{1/(n-1)}/2\right]
\neq (K|\,l_{e_1})*(L|\,l_{e_1}).$$
This example also shows that projection covariance cannot be omitted in Theorem~\ref{thm1} and moreover that $GL(n)$ covariance is essential for Corollary~\ref{cor10}.}
\end{ex}

Note that the trivial operations $K*L=\{o\}$, or $K*L=K$, or $K*L=L$, for all $K,L\in {\mathcal{K}}^n_s$, show that the identity property alone cannot be omitted in Corollary~\ref{cor10}.

Next, we consider operations on pairs of star sets.  The case $n=1$ is already dealt with in Lemma~\ref{lem0}, since ${\mathcal{S}}^1_s={\mathcal{K}}^1_s$, and $L_p$ addition and $p$th radial addition coincide.

\begin{lem}\label{lem01arbsections}
Suppose that $n\ge 2$ and that $*:\left({\mathcal{S}}^n_s\right)^2\rightarrow {\mathcal{S}}^n$ is rotation and section covariant.  Then $*:\left({\mathcal{S}}^n_s\right)^2\rightarrow {\mathcal{S}}_s^n$.
\end{lem}

\begin{proof}
Let $u\in S^{n-1}$, and let $\phi_u$ be a rotation such that $\phi_u(u)=-u$.  Note that $\phi_u l_u=l_u$.  If $K, L\in {\mathcal{S}}^n_s$, then
\begin{eqnarray*}
\phi_u\left((K*L)\cap l_u\right)& =& \left(\phi_u(K*L)\right)\cap l_u = (\phi_u K * \phi_u L)\cap l_u\\
&=&((\phi_u K)\cap l_u) * ((\phi_u L)\cap l_u) = (K\cap l_u) * (L\cap l_u) = (K*L)\cap l_u.
\end{eqnarray*}
Thus $(K*L)\cap l_u\in {\mathcal{S}}_s^n$ for all $u\in S^{n-1}$, so $K*L\in {\mathcal{S}}_s^n$.
\end{proof}

\begin{thm}\label{thm1a}
If $n\ge 2$, then $*:\left({\mathcal{S}}^n_s\right)^2\rightarrow {\mathcal{S}}^n$ is continuous, homogeneous of degree $1$, rotation and section covariant, and associative if and only if either $K*L=\{o\}$, or $K*L=K$, or $K*L=L$, for all $K,L\in {\mathcal{S}}^n_s$, or else $*=\widetilde{+}_p$ for some $-\infty\le p\le\infty$ with $p\neq 0$.
\end{thm}

\begin{proof}
If $*:\left({\mathcal{S}}^n_s\right)^2\rightarrow {\mathcal{S}}^n$ is rotation and section covariant, then by Lemma~\ref{lem01arbsections}, we have $*:\left({\mathcal{S}}^n_s\right)^2\rightarrow {\mathcal{S}}_s^n$, so associativity makes sense.  Let $u\in S^{n-1}$. Since $*$ is section covariant, for any two $o$-symmetric star sets $K$ and $L$ in $\R^n$, we have
\begin{equation}\label{seqb}
(K*L)\cap l_u=(K\cap l_u)*(L\cap l_u).
\end{equation}
One consequence of this is that if $I$ and $J$ are $o$-symmetric closed intervals in $l_u$, we must have $I*J\subset l_u$. Hence, for each $u\in S^{n-1}$, there is a function $f_u:[0,\infty)^2\to[0,\infty)$
such that
\begin{equation}\label{seqn1}
[-su,su]*[-tu,tu]=[-f_u(s,t)u,f_u(s,t)u],
\end{equation}
for all $s,t\ge 0$. Let $\phi$ be a rotation. Then
$\phi u\in S^{n-1}$ and, for $s,t\ge 0$, we use (\ref{seqn1}) and the rotation covariance of $*$ to obtain
\begin{eqnarray*}
[-f_{\phi u}(s,t)\phi u,f_{\phi u}(s,t)\phi u]&=&[-s\phi u,s\phi u]*[-t\phi u,t\phi u]\\
&=&\phi[-s u,s u]*\phi[-t u,tu]=\phi([-s u,s u]*[-t u,tu])\\
&=&\phi[-f_u(s,t) u,f_u(s,t) u]=[-f_u(s,t) \phi u,f_u(s,t)\phi u].
\end{eqnarray*}
We conclude that $f_{\phi u}(s,t)=f_u(s,t)$ for $s,t\ge 0$.
Since $\phi$ was an arbitrary rotation, we conclude that $f_u=f$, say, is independent of $u$.  Now from (\ref{seqb}) and (\ref{seqn1}) with $f_u=f$, we obtain
\begin{eqnarray*}
[-\rho_{K*L}(u)u,\rho_{K*L}(u)u]&=&(K*L)\cap l_u=(K\cap l_u)*(L\cap l_u)\\
&=&[-\rho_{K}(u)u,\rho_{K}(u)u]*[-\rho_{L}(u)u,\rho_{L}(u)u]\\
&=&[-f\left(\rho_K(u),\rho_L(u)\right)u,f\left(\rho_K(u),\rho_L(u)\right)u].
\end{eqnarray*}
Hence
\begin{equation}\label{442}
\rho_{K*L}(u)=f\left(\rho_K(u),\rho_L(u)\right),
\end{equation}
for all $K,L\in {\mathcal{S}}^n_s$, all $u\in S^{n-1}$, and some function $f:[0,\infty)^2\rightarrow [0,\infty)$.  Let $r,s,t\ge 0$ and $u\in S^{n-1}$ and let $K=sB^n$ and $L=tB^n$.  Then (\ref{442}) and the homogeneity of $*$ imply that
\begin{eqnarray*}
f(rs,rt)&=&f\left(r\rho_K(u),r\rho_L(u)\right)=f\left(\rho_{rK}(u),\rho_{rL}(u)\right)\\
&=&\rho_{rK*rL}(u)=\rho_{r(K*L)}(u)=r\rho_{K*L}(u)=rf\left(\rho_K(u),\rho_L(u)\right)=rf(s,t),
\end{eqnarray*}
so $f$ is homogeneous of degree 1.  The associativity of $f$ is proved by the same argument used for $h_M$ in the proof of Theorem~\ref{thm1}.  Let $s_m,t_m\ge 0$, $m\in \N$, be such that $(s_m,t_m)\rightarrow (s,t)$ as $k\rightarrow\infty$.  With $K_m=s_mB^n$, $L_m=t_mB^n$, $K=sB^n$, and $L=tB^n$, we obtain
$$f(s_m,t_m)=f\left(\rho_{K_m}(u),\rho_{L_m}(u)\right)=\rho_{K_m*L_m}(u)\rightarrow
\rho_{K*L}(u)=f\left(\rho_{K}(u),\rho_{L}(u)\right)=f(s,t),$$
as $m\rightarrow\infty$, for all $u\in S^{n-1}$, by the continuity of $*$.  Hence $f$ is continuous.  It follows that $f$ satisfies the hypotheses of Proposition \ref{prp2} and so must be of one of the forms listed there.  Then (\ref{442}) implies that $*$ is one of the operations listed in the statement of the theorem.

The converse is clear.
\end{proof}

A characterization of $p$th radial addition is obtained by adding the identity property to the hypotheses of Theorem~\ref{thm1a}.

Note that as in Corollary~\ref{thmcor9}, the assumption of homogeneity of degree 1 in Theorem~\ref{thm1a} can be weakened to quasi-homogeneity if the identity property is added to the hypotheses.

Various examples of operations $*:\left({\mathcal{S}}^n_s\right)^2\rightarrow {\mathcal{S}}^n$ can be obtained by defining $K*L$ by
$$\rho_{K*L}(u)=f\left(\rho_K(u),\rho_L(u)\right),$$
for $u\in S^{n-1}$, where $f:[0,\infty)^2\rightarrow [0,\infty)$.  Here either the function $f$ is homogeneous of degree 1, or one extends the definition to $\R^n\setminus\{o\}$ by setting
$$\rho_{K*L}(ru)=\frac{1}{r}\rho_{K*L}(u),$$
for $u\in S^{n-1}$ and $r>0$. Examples showing that homogeneity, continuity, and associativity cannot be dropped separately in Theorem~\ref{thm1a} are obtained by taking $f=f_1$, $f=f_2$, or $f=f_5$, respectively, as given after Proposition~\ref{prp1}.  (Note that taking $f=f_2$ yields an operation that is not the same as that defined in Example~\ref{ex4}; to see the difference, take $K=[-e_1,e_1]$ and $L=[-e_2,e_2]$.)   All these operations also have the identity property. Example~\ref{ex5} shows that section covariance cannot be omitted, but this operation does not have the identity property; an obvious modification of Example~\ref{ex5555}, where (\ref{eq5555}) is replaced by $K*L=F^{-1}(F(K)\widetilde{+}F(L))$ for all $K,L\in {\mathcal{S}}^n_s$, serves the same purpose and also has the identity property.  Finally, the following example proves the necessity of the hypothesis of rotation covariance.

\begin{ex}\label{ex3}
{\em To show that rotation covariance cannot be omitted from the hypotheses of Theorem~\ref{thm1a}, let $n\ge 2$ and let $p:S^{n-1}\rightarrow [1,\infty)$ be any function that is continuous but not rotation invariant. (For example, when $n=2$, we can take $p(\theta)=\sin\theta+2$ for $0\le \theta<2\pi$.)  Then define $p(ru)=p(u)$ for all $r>0$. Define
$$\rho_{K*L}(x)=\left(\rho_K(x)^{p(x)}+\rho_L(x)^{p(x)}\right)^{1/p(x)},$$
for all $K, L\in {\mathcal{S}}_s^n$ and $x\in\R^n\setminus\{o\}$.  Then $*$ is continuous, homogeneous of degree 1, associative, section covariant, and has the identity property, but it is not rotation covariant.}
\end{ex}

\section{Classification of $o$-symmetrizations}\label{Symmetrization}

\begin{lem}\label{lem01s}
Let $H$ be the closed half-plane $H=\{(s,t)\in\R^2: -s\le t\}$. If $n\ge 2$, the $o$-symmetrization $\di:{\mathcal{K}}^n\rightarrow{\mathcal{K}}^n_s$ is projection covariant if and only if there is a homogeneous of degree 1 function $f:H\to[0,\infty)$, symmetric in its variables, such that
\begin{equation}\label{441s}
h_{\di K}(x)=f\left(h_{K}(x),h_{-K}(x)\right),
\end{equation}
for all $K\in {\mathcal{K}}^n$ and all $x\in \R^n$.
\end{lem}

\begin{proof}
Let $\di:{\mathcal{K}}^n\rightarrow{\mathcal{K}}^n_s$ be projection covariant and let $u\in S^{n-1}$. For $K\in {\mathcal{K}}^n$, we have
\begin{equation}\label{eqbs}
(\di K)|\,l_u=\di(K|\,l_u).
\end{equation}
One consequence of this is that if $I$ is a closed interval in $l_u$, we must have $\di I\subset l_u$.  Thus there is a function $f_u:H\to[0,\infty)$
such that
\begin{equation}\label{eqn1s}
\di [-su,tu]=[-f_u(s,t)u,f_u(s,t)u],
\end{equation}
whenever $(s,t)\in H$.

Let $0\le \alpha\le 1$ and choose $v\in S^{n-1}$ such that $u\cdot v=\alpha$.  Using (\ref{eqbs}) with $K=[-su,tu]$ and $l_u$ replaced by $l_v$, and (\ref{eqn1s}), we obtain
\begin{eqnarray*}
\alpha[-f_{u}(s,t)v,f_{u}(s,t)v]&=&[-f_{u}(s,t)u,f_{u}(s,t)u]|\,l_v
=(\di[-su,tu])|\,l_v\\
&=&\di([-su,tu]|\,l_v)=\di[-\alpha sv,\alpha tv]=
[-f_{v}(\alpha s,\alpha t)v,f_{v}(\alpha s,\alpha t)v],
\end{eqnarray*}
for all $(s,t)\in H$.  Therefore $f_v(\alpha s,\alpha t)=\alpha f_u(s,t)$, for all $(s,t)\in H$.  Now exactly as in the proof of Lemma~\ref{lem01}, we conclude that $f_u$ is homogeneous of degree 1 and further that $f_u=f$, say, is independent of $u$.

Now from (\ref{eqbs}) and (\ref{eqn1s}) with $f_u=f$, we obtain
\begin{eqnarray*}
[-h_{\di K}(u)u,h_{\di K}(u)u]&=&(\di K)|\,l_u=\di (K|\,l_u)=\di[-h_{K}(-u)u,h_{K}(u)u]\\
&=&[-f\left(h_{-K}(u),h_K(u)\right)u,f\left(h_{-K}(u),h_K(u)\right)u],
\end{eqnarray*}
for all $u\in S^{n-1}$. This yields
\begin{equation}\label{hdknew}
h_{\di K}(u)=f\left(h_{-K}(u),h_{K}(u)\right),
\end{equation}
for all $u\in S^{n-1}$.

Let $(s,t)\in H$ and choose $K\in {\mathcal{K}}^n$ and $u\in S^{n-1}$ such that $h_{-K}(u)=s$ and $h_K(u)=t$.  Then by (\ref{hdknew}), we have
\begin{eqnarray*}
f(s,t)&=&f\left(h_{-K}(u),h_{K}(u)\right)=h_{\di K}(u)=h_{\di K}(-u)\\
&=&f\left(h_{-K}(-u),h_{K}(-u)\right)=f\left(h_{K}(u),h_{-K}(u)\right)=f(t,s),
\end{eqnarray*}
so $f$ is symmetric in its variables and (\ref{441s}) holds for $x\in S^{n-1}$.
Then the homogeneity of $f$ can be used, just as it was in the proof of Lemma~\ref{lem01}, to obtain (\ref{441s}).

An argument analogous to that in the last paragraph of the proof of Lemma~\ref{lem01} proves the converse.
\end{proof}

\begin{thm}\label{thm2s}
If $n\ge 2$, the $o$-symmetrization $\di:{\mathcal{K}}^n\rightarrow{\mathcal{K}}^n_s$ is projection covariant if and only if there is a compact convex set $M$ in $\R^2$, symmetric in the line $x_1=x_2$, such that
\begin{equation}\label{ns}
h_{\di K}(x)=h_{M}\left(h_K(x),h_{-K}(x)\right),
\end{equation}
for all $K\in {\mathcal{K}}^n$ and all $x\in\R^n$.
\end{thm}

\begin{proof}
Let $\di:{\mathcal{K}}^n\rightarrow{\mathcal{K}}^n_s$ be projection covariant.
Lemma~\ref{lem01s} implies that there is a homogeneous of degree 1 function $f:H\to[0,\infty)$, symmetric in its variables, such that (\ref{441s}) holds.

Let $K_0=[-e_2,e_1]$. Then for $x=(x_1,\dots,x_n)\in\R^n$ with $(x_1,x_2)\in H$, we have $-x_1\le x_2$ and hence
$$h_{K_0}(x)=\max\{-x\cdot e_2,x\cdot e_1\}=\max\{-x_2,x_1\}=x_1$$ and similarly $h_{-K_0}(x)=x_2$.   Let $S=\lin\{e_1,e_2\}$ and note that the projection covariance of $\di$ implies that
$$(\di K_0)|S=\di(K_0|S)=\di K_0,$$
so $\di K_0\subset S$. Identifying $S$ with $\R^2$ in the natural way, we let $M=\di K_0$. Then (\ref{441s}) with $K=K_0$ yields
\begin{equation}\label{dif}
h_M(x_1,x_2)=h_{\di K_0}(x)=f\left(h_{K_0}(x),h_{-K_0}(x)\right)=f(x_1,x_2),
\end{equation}
whenever $(x_1,x_2)\in H$.  Since $(h_{K}(x),h_{-K}(x))\in H$ for all $K\in {\mathcal{K}}^n$ and all $x\in\R^n$, (\ref{ns}) follows directly from (\ref{441s}) and (\ref{dif}).  Finally, $f$ is symmetric in its variables by Lemma~\ref{lem01s}, so the symmetry of $M$ in the line $x_1=x_2$ is a consequence of (\ref{dif}).

The converse is clear.
\end{proof}

\begin{cor}\label{contaffs}
Let $n\ge 2$. An $o$-symmetrization $\di:{\mathcal{K}}^n\rightarrow{\mathcal{K}}^n_s$ is projection covariant if and only if it is continuous and $GL(n)$ covariant (and hence homogeneous of degree 1).
\end{cor}

\begin{proof}
If $\di$ is continuous and $GL(n)$ covariant, then it is projection covariant by Lemma~\ref{diCGimpliesP}.  If $\di$ is projection covariant, then both the continuity and the $GL(n)$ covariance (using the formula \cite[(0.27), p.~18]{Gar06} for the change in a support function under a linear transformation) are easy consequences of (\ref{ns}).
\end{proof}

Theorem~\ref{thm2s} raises the question as to which compact convex sets $M$ in $\R^2$, symmetric in the line $x_1=x_2$, are such that the right-hand side of (\ref{ns}) is a support function for every $K\in {\mathcal{K}}^n$.  A partial answer is provided by Corollary~\ref{corthmM1}, which implies that this is true if in addition $M\subset [0,\infty)^2$.  It is natural to ask when it is true if in addition $M$ is 1-unconditional.

We first observe that it is true when $M$ is the unit ball in $l^2_1$, that is, $M=\conv\{\pm e_1,\pm e_2\}$. In this case the right-hand side of (\ref{ns}) becomes
$$h_{M}\left(h_K(x),h_{-K}(x)\right)=\max\{|h_K(x)|,|h_{-K}(x)|\},$$
for all $x\in \R^n$ and it is a routine exercise to show that this is a support function.

However, the right-hand side of (\ref{ns}) is not a support function when $M$ is the unit ball in $l^2_p$ and $1<p\le \infty$.  To see this, suppose the contrary, and let $K=\conv\{-e_1,-e_2,-e_1-e_2\}$, so that $h_{K}(e_1)=h_{K}(e_2)=0$, $h_{K}(e_1+e_2)=-1$, $h_{-K}(e_1)=h_{-K}(e_2)=1$, and $h_{-K}(e_1+e_2)=2$.  The subadditivity of $h_M(h_{K}(x),h_{-K}(x))$, with $x=e_1$, $e_2$, and $e_1+e_2$, yields
$$h_M(-1,2)\le 2h_M(0,1).$$
Now let $M$ be the unit $l_p^2$ ball, so that $h_M(s,t)=(|s|^{p'}+|t|^{p'})^{1/p'}$ and $1\le p'<\infty$.  Then the previous inequality implies that $1+2^{p'}\le 2^{p'}$, which is false.

The following corollary characterizes the central symmetral operator (or, equivalently, the difference body operator, where the difference body $DK$ is defined by $DK=K+(-K)=2\Delta K$).

\begin{cor}\label{corthm2s}
If $n\ge 2$, the $o$-symmetrization $\di:{\mathcal{K}}^n\rightarrow{\mathcal{K}}^n_s$ is projection covariant and translation invariant, i.e., $\di(K+z)=\di K$ for all $z\in\R^n$, if and only if there is a $\lambda\ge 0$ such that $\di K=\lambda \Delta K$.
\end{cor}

\begin{proof}
Let $K\in {\mathcal{K}}^n$. By (\ref{ns}), we have
$$
h_{\di (K+z)}(x)=h_{M}\left(h_{K+z}(x),h_{-K-z}(x)\right)=
h_{M}\left(h_{K}(x)+x\cdot z,h_{-K}(x)-x\cdot z\right),$$
for all $x,z\in \R^n$.  Choosing $z$ so that $x\cdot z=\left(h_{-K}(x)-h_K(x)\right)/2$, we obtain
$$
h_{\di K}(x)=h_{\di (K+z)}(x)=
h_{M}\left(\frac{h_{K}(x)+h_{-K}(x)}{2},\frac{h_{K}(x)+h_{-K}(x)}{2}\right)
=\lambda h_{\Delta K}(x),$$
where $\lambda=h_{M}(1,1)$, for all $x\in \R^n$.
\end{proof}

Neither projection covariance nor translation invariance can be omitted in the previous result.  Indeed, the $o$-symmetrization defined as in Example~\ref{LPex} by $\di K=\Delta_p K$ for $K\in {\mathcal{K}}^n$ is projection covariant but not translation invariant when $p>1$, while the $o$-symmetrization defined by $\di K=B^n$ for each $K\in {\mathcal{K}}^n$ is translation invariant but not projection covariant.

The next result is obtained in the same fashion as Theorem~\ref{thm1a}, with the symmetry of $f$ proved as in Lemma~\ref{lem01s}.

\begin{thm}\label{thm3s}
If $n\ge 2$, the $o$-symmetrization $\di:{\mathcal{S}}^n\rightarrow{\mathcal{S}}^n_s$ is homogeneous of degree 1 and rotation and section covariant if and only if there is a function $f:[0,\infty)^2\rightarrow [0,\infty)$, symmetric in its variables, such that
$$
\rho_{\di K}(x)=f\left(\rho_K(x),\rho_{-K}(x)\right),
$$
for all $K\in {\mathcal{S}}^n$ and all $x\in\R^n$.
\end{thm}

\section{Operations between arbitrary compact convex or star sets}\label{arbresults}

We begin with the following result, that can be deduced immediately from Corollary~\ref{corsymm}.

\begin{cor}\label{corrlem01arb}
Suppose that $n\ge 2$ and that $*:\left({\mathcal{K}}^n\right)^2\rightarrow {\mathcal{K}}^n$ (or $*:\left({\mathcal{K}}^n_o\right)^2\rightarrow {\mathcal{K}}^n$) is such that its restriction to $\left({\mathcal{K}}^n_s\right)^2$ is projection covariant.  Then $*:\left({\mathcal{K}}^n_s\right)^2\rightarrow {\mathcal{K}}_s^n$.
\end{cor}

Many consequences of the previous simple result and those from Section~\ref{results} could be stated.  For example, Theorem~\ref{thm1} and Corollary~\ref{corrlem01arb} yield the following corollary.

\begin{cor}\label{newcorthm1}
Let $n\ge 2$, and let $*:\left({\mathcal{K}}^n\right)^2\rightarrow {\mathcal{K}}^n$ (or $*:\left({\mathcal{K}}^n_o\right)^2\rightarrow {\mathcal{K}}^n$) be such that its restriction to the $o$-symmetric sets is projection covariant and associative.  Then this restriction must be either $K*L=\{o\}$, or $K*L=K$, or $K*L=L$, for all $K,L\in {\mathcal{K}}^n_s$, or else $*=+_p$ for some $1\le p\le\infty$.
\end{cor}

\begin{cor}\label{cor10arb}
Let $n\ge 2$, and let $*:\left({\mathcal{K}}^n\right)^2\rightarrow {\mathcal{K}}^n$ (or $*:\left({\mathcal{K}}^n_o\right)^2\rightarrow {\mathcal{K}}^n$) be such that its restriction to the $o$-symmetric sets is continuous, $GL(n)$ covariant, associative, and has the identity property.  Then this restriction must be $L_p$ addition for some $1\le p\le\infty$.
\end{cor}

\begin{ex}\label{ex1dim}
{\em The previous corollary does not hold when $n=1$.  Indeed, let $1\le p\neq q\le \infty$ and define $*:\left({\mathcal{K}}^1\right)^2\rightarrow {\mathcal{K}}^1$ by
$$[-s_1,t_1]*[-s_2,t_2]=\left[-\left(|s_1|^p+|s_2|^p\right)^{1/p},
\left(|t_1|^q+|t_2|^q\right)^{1/q}\right],$$
for all $s_1,t_1,s_2,t_2$ with $-s_1\le t_1$ and $-s_2\le t_2$.  Then $*$ satisfies the hypotheses of Corollary~\ref{cor10arb} but its restriction to the $o$-symmetric intervals is not $L_p$ addition for any $1\le p\le\infty$.}
\end{ex}

\begin{ex}\label{NotLp}
{\em Note that even when $n\ge 2$, an operation $*:\left({\mathcal{K}}^n\right)^2\rightarrow {\mathcal{K}}^n$ that is continuous, $GL(n)$ covariant, and associative need not itself be $L_p$ addition.  For example, let $\di:{\mathcal{K}}^n\rightarrow{\mathcal{K}}^n_s$ be an $o$-symmetrization that is continuous and $GL(n)$ covariant, and that has the identity property  $\di K=K$ if $K\in {\mathcal{K}}^n_s$.  For some $1\le p\le \infty$ and all $K,L\in {\mathcal{K}}^n$, define
$$K*L=\di K+_p \di L.$$
Then $*$ is clearly continuous and $GL(n)$ covariant, and it is easy to check that $*$ is also associative.  A more specific example is obtained by taking $\di K=\Delta_qK$ for $K\in {\mathcal{K}}^n$, $q\ge 1$ (defined for $q>1$ as in Example~\ref{LPex}).  The simplest operation in this class is given by $p=q=1$, i.e.,
$$K*L=\Delta K+\Delta L.$$
Of course, this is ordinary Minkowski addition when restricted to the $o$-symmetric sets.}
\end{ex}

Nevertheless, results such as Corollaries~\ref{newcorthm1} and~\ref{cor10arb} show that respectable projection covariant operations between arbitrary compact compact sets must actually be $L_p$ addition (or $p$th radial addition, respectively) when restricted to $o$-symmetric sets.  The assumption of projection covariance is crucial.  Indeed, Corollary~\ref{BlaPol} shows that neither polar $L_p$ addition, for $n\ge 2$ and $-\infty\le p\le -1$, nor Blaschke addition, for $n\ge 3$, can be extended to projection covariant operations, even between $o$-symmetric compact convex sets.

\begin{lem}\label{lem017}
Let $H$ be the closed half-plane $H=\{(s,t)\in\R^2: -s\le t\}$.
If $n\ge 2$, then $*:\left({\mathcal{K}}^n\right)^2\rightarrow {\mathcal{K}}^n$ (or $*:\left({\mathcal{K}}^n_o\right)^2\rightarrow {\mathcal{K}}^n$) is projection covariant if and only if there is a homogeneous of degree 1 function $f:H^2\to\R$ (or $f:[0,\infty)^4\to\R$, respectively) such that
\begin{equation}\label{n7}
h_{K*L}(x)=f\left(h_K(-x),h_K(x), h_L(-x), h_L(x)\right),
\end{equation}
for all $K,L\in {\mathcal{K}}^n$ (or $K,L\in{\mathcal{K}}^n_o$, respectively) and all $x\in \R^n$.
\end{lem}

\begin{proof}
Suppose that $*:\left({\mathcal{K}}^n\right)^2\rightarrow {\mathcal{K}}^n$  is projection covariant.  (The proof when $*:\left({\mathcal{K}}^n_o\right)^2\rightarrow {\mathcal{K}}^n$ is a straightforward modification of what follows.) Let $u\in S^{n-1}$.  Since $*$ is covariant under projection onto $l_u$, for any two compact convex sets $K$ and $L$ in $\R^n$, we have
\begin{equation}\label{eqb7}
(K*L)|\,l_u=(K|\,l_u)*(L|\,l_u).
\end{equation}
One consequence of this is that if $I$ and $J$ are closed intervals in $l_u$, we must have $I*J\subset l_u$. Hence there are functions $f_u, g_u:H^2\to\R$
such that $-g_u\le f_u$ and
\begin{equation}\label{eqn17}
[-s_1u,t_1u]*[-s_2u,t_2u]=[-g_u(s_1,t_1,s_2,t_2)u,f_u(s_1,t_1,s_2,t_2)u],
\end{equation}
for all $s_1,t_1,s_2,t_2$ with $-s_1\le t_1$ and $-s_2\le t_2$, i.e., for $(s_1,t_1,s_2,t_2)\in H^2$.

Let $0\le \alpha\le 1$ and choose $v\in S^{n-1}$ such that $u\cdot v=\alpha$.  Using (\ref{eqb7}) with $K=[-s_1u,t_1u]$, $L=[-s_2u,t_2u]$, and $l_u$ replaced by $l_v$, and (\ref{eqn17}), we obtain
\begin{eqnarray*}
\alpha[-g_{u}(s_1,t_1,s_2,t_2)v,f_{u}(s_1,t_1,s_2,t_2)v]&=&
[-g_{u}(s_1,t_1,s_2,t_2)u,f_{u}(s_1,t_1,s_2,t_2)u]|\,l_v\\
&=&([-s_1u,t_1u]*[-s_2u,t_2u])|\,l_v\\
&=&([-s_1u,t_1u]|\,l_v)*([-s_2u,t_2u]|\,l_v)\\
&=&[-\alpha s_1v,\alpha t_1v]*[-\alpha s_2v,\alpha t_2v]\\
&=&
[-g_{v}(\alpha s_1,\alpha t_1, \alpha s_2,\alpha t_2)v,
f_{v}(\alpha s_1,\alpha t_1, \alpha s_2,\alpha t_2)v],
\end{eqnarray*}
for all $(s_1,t_1,s_2,t_2)\in H^2$.  Therefore
$$f_v(\alpha s_1,\alpha t_1, \alpha s_2,\alpha t_2)=\alpha f_u(s_1,t_1,s_2,t_2)\quad{\text{and}}\quad g_v(\alpha s_1,\alpha t_1, \alpha s_2,\alpha t_2)=\alpha g_u(s_1,t_1,s_2,t_2),$$
for all $(s_1,t_1,s_2,t_2)\in H^2$. Now exactly as in the proof of Lemma~\ref{lem01}, we conclude that both $f_u$ and $g_u$ are homogeneous of degree 1 and further that $f_u=f$ and $g_u=g$, say, are independent of $u$.  Thus we have
\begin{equation}\label{comm7}
[-s_1u,t_1u]*[-s_2u,t_2u]=[-g(s_1,t_1,s_2,t_2)u,f(s_1,t_1,s_2,t_2)u],
\end{equation}
for all $u\in S^{n-1}$ and $(s_1,t_1,s_2,t_2)\in H^2$. Then (\ref{eqb7}) and (\ref{comm7}) yield
\begin{eqnarray*}
\lefteqn{[-h_{K*L}(-u)u,h_{K*L}(u)u]}\\
&=&(K*L)|\,l_u=(K|\,l_u)*(L|\,l_u)=[-h_{K}(-u)u,h_{K}(u)u]*[-h_{L}(-u)u,h_{L}(u)u]\\
&=&[-g\left(h_K(-u),h_K(u), h_L(-u), h_L(u)\right)u,f\left(h_K(-u),h_K(u), h_L(-u),h_L(u)\right)u].
\end{eqnarray*}
Comparing the second coordinates in the previous equation, we deduce (\ref{n7}) for $x\in S^{n-1}$.  (Note that in view of the equality of the first coordinates, we must in fact have $g(s_1,t_1,s_2,t_2)=f(t_1,s_1,t_2,s_2)$ for all $(s_1,t_1,s_2,t_2)\in H^2$.)  As in the proof of Theorem~\ref{thm12}, using the homogeneity of $f$ and of support functions, we easily obtain (\ref{n7}) for all $x\in \R^n$.

An argument analogous to that in the last paragraph of the proof of Lemma~\ref{lem01} proves the converse.
\end{proof}

Our next goal is to show that the function $f$ in the previous lemma can be taken to be the support function of a closed convex set.  Unfortunately, the method used in the transition from Lemma~\ref{lem01} to Theorem~\ref{thm12} works only when $n\ge 4$ (see Remark~\ref{remark}), and otherwise another route has to be followed.

\begin{thm}\label{thm275}
If $n\ge 2$, then $*:\left({\mathcal{K}}^n\right)^2\rightarrow {\mathcal{K}}^n$ (or $*:\left({\mathcal{K}}^n_o\right)^2\rightarrow {\mathcal{K}}^n$) is projection covariant if and only if there is a nonempty closed convex set $M$ in $\R^4$ such that
\begin{equation}\label{n77}
h_{K*L}(x)=h_M\left(h_K(-x),h_K(x), h_L(-x), h_L(x)\right),
\end{equation}
for all $K,L\in {\mathcal{K}}^n$ (or $K,L\in{\mathcal{K}}^n_o$, respectively) and $x\in\R^n$.
\end{thm}

\begin{proof}
Suppose that $*:\left({\mathcal{K}}^n\right)^2\rightarrow {\mathcal{K}}^n$ is projection covariant. By Lemma~\ref{lem017}, there is a homogeneous of degree 1 function $f:H^2\to\R$ such that (\ref{n7}) holds, where $H$ is as in the statement of Lemma~\ref{lem017}.

We aim to prove that $f$ is subadditive on $H^2$.  To this end, let $a,b\in H^2$, so that $-a_1\le a_2$, $-a_3\le a_4$, $-b_1\le b_2$, and $-b_3\le b_4$.  Let $K_0=[-a_1e_1-b_1e_2,a_2e_1+b_2e_2]$ and $L_0=[-a_3e_1-b_3e_2,a_4e_1+b_4e_2]$ be line segments in $\R^n$. Then
\begin{equation}\label{zz1}
h_{K_0}(-e_1)=a_1,\quad h_{K_0}(e_1)=a_2,\quad h_{K_0}(-e_2)=b_1,\quad{\text{and}}\quad
h_{K_0}(e_2)=b_2.
\end{equation}
(For example,
$$h_{K_0}(-e_1)=\max\{(-a_1e_1-b_1e_2)\cdot (-e_1),(a_2e_1+b_2e_2)\cdot (-e_1)\}=
\max\{a_1,-a_2\}=a_1.)$$
Similarly, we obtain
\begin{equation}\label{zz2}
h_{K_0}(-e_1-e_2)=a_1+b_1,\quad h_{K_0}(e_1+e_2)=a_2+b_2,
\end{equation}
\begin{equation}\label{zz3}
h_{L_0}(-e_1)=a_3,\quad h_{L_0}(e_1)=a_4, \quad h_{L_0}(-e_2)=b_3,\quad h_{L_0}(e_2)=b_4,
\end{equation}
\begin{equation}\label{zz4}
h_{L_0}(-e_1-e_2)=a_3+b_3,\quad{\text{and}}\quad h_{L_0}(e_1+e_2)=a_4+b_4.
\end{equation}
Therefore by (\ref{n7}) and the subadditivity of $h_{{K_0}*{L_0}}$, we have
\begin{eqnarray*}
f(a+b)\!\!\!
&=&\!\!\!f\left(h_{K_0}(-e_1-e_2),h_{K_0}(e_1+e_2),h_{L_0}(-e_1-e_2),
h_{L_0}(e_1+e_2)\right)\\
&=&\!\!\!h_{{K_0}*{L_0}}(e_1+e_2)\le h_{{K_0}*{L_0}}(e_1)+h_{{K_0}*{L_0}}(e_2)\\
&=&\!\!\!f\left(h_{K_0}(-e_1),h_{K_0}(e_1),h_{L_0}(-e_1),
h_{L_0}(e_1)\right)+f\left(h_{K_0}(-e_2),h_{K_0}(e_2),h_{L_0}(-e_2),
h_{L_0}(e_2)\right)\\
&=&\!\!\!f(a)+f(b).
\end{eqnarray*}
This proves that $f$ is subadditive on $H^2$.

Our next goal is to show that $f$ is continuous on $H^2$. We first claim that for each $r>0$, there is an $R=R(r)>0$ such that $K*L\subset RB^n$ whenever $K,L\subset rB^n$.  Indeed, we know that the function $f$ is sublinear and hence convex on $H^2$.  Since $f$ is also finite on $H^2\cap [-r,r]^4$, it has a maximum, $R$, say, on this set, attained at one of the vertices.  Because $K,L\subset rB^n$, we have
$$(h_K(-u),h_K(u),h_L(-u),h_L(u))\in H^2\cap [-r,r]^4,$$
for all $u\in S^{n-1}$. Then (\ref{n7}) implies that $h_{K*L}(u)\le R$
for all $u\in S^{n-1}$, proving the claim.

Now suppose that $a\in H^2$ and let $\ee>0$.  If $a\neq o$, let $r=2\max_{i=1,\dots,4}\{|a_i|\}$ and if $a=o$, let $r=1$.  In each case, let $R$ be the corresponding radius defined in the previous paragraph. Choose $0<\theta<\pi/2$ small enough to ensure that with $x=e_1$ and $x'=\cos\theta e_1+\sin\theta e_2$, we have $R|x-x'|<\ee$.  Let
$${K_1}=\left[-a_1e_1+\left(\frac{a_1\cos\theta-a_1'}{\sin\theta}\right)e_2,
a_2e_1+\left(\frac{-a_2\cos\theta+a_2'}{\sin\theta}\right)e_2\right]=[p,q],$$
say, and
$${L_1}=\left[-a_3e_1+\left(\frac{a_3\cos\theta-a_3'}{\sin\theta}\right)e_2,
a_4e_1+\left(\frac{-a_4\cos\theta+a_4'}{\sin\theta}\right)e_2\right]=[v,w],$$
say, be line segments in $\R^n$.  It is easy to check that
\begin{equation}\label{hhh}
(h_{K_1}(-x),h_{K_1}(x),h_{L_1}(-x),h_{L_1}(x))=a\quad{\text{and}}\quad (h_{K_1}(-x'),h_{K_1}(x'),h_{L_1}(-x'),h_{L_1}(x'))=a'.
\end{equation}
The latter equation shows that $a'\in H^2$.
We claim that ${K_1},{L_1}\subset rB^n$ if $a'$ is sufficiently close to $a$.  To see this, note that ${K_1},{L_1}\subset cB^n$, where
\begin{equation}\label{maxc}
c=\max\{|p|,|q|,|v|,|w|\}=\max_{i=1,\dots,4}\left\{\left(a_i^2+a_i'^2
-2a_ia_i'\cos\theta\right)^{1/2}/\sin\theta\right\}.
\end{equation}
If $a=o$, then $c=\max_{i=1,\dots,4}\{|a_i'|/\sin\theta\}<1=r$ for $a'$ sufficiently close to $a$.  If $a\neq o$, note that as $a'\rightarrow a$, the right-hand expression in (\ref{maxc}) approaches $$\max_{i=1,\dots,4}\{|a_i|/\cos(\theta/2)\}
\le\sqrt{2}\max_{i=1,\dots,4}\{|a_i|\}<r,$$
proving the claim. Then ${K_1}*{L_1}\subset RB^n$ and from this, (\ref{hhh}), (\ref{n7}), and the subadditivity of $h_{{K_1}*{L_1}}$, we get
$$f(a)-f(a')=h_{{K_1}*{L_1}}(x)-h_{{K_1}*{L_1}}(x')
\le h_{{K_1}*{L_1}}(x-x')\le R|x-x'|<\ee,$$
for all $a'$ sufficiently close to $a$.  The same bound applies to $f(a')-f(a)$, establishing the continuity of $f$.

Because $f$ is sublinear on $H^2$, it is also convex on $H^2$. Extend $f$ to a function $\overline f : \R^4\to \R\cup\{\infty\}$ by defining $\overline f(x)=\infty$, for $x\notin H^2$. Then $\overline f$ is convex and proper (i.e., not identically $\infty$).  Since $f$ is continuous on $H^2$, $\overline f$ is lower semi-continuous and hence closed (i.e., the epigraph $\{(x,t)\in \R^{5}:t\ge \overline f(x), x\in \R^4, t\in \R\}$ of $\overline f$ is closed; see \cite[p.~52]{Roc70}).  By \cite[Theorem~13.2]{Roc70}, there is a nonempty closed convex set $M\subset\R^4$ with support function $h_M=\overline f$ and therefore $h_M=f$ on $H^2$.  This yields (\ref{n77}) for this case.

The case when $*:\left({\mathcal{K}}^n_o\right)^2\rightarrow {\mathcal{K}}^n$ is projection covariant is handled in a similar fashion. By Lemma~\ref{lem017}, there is a homogeneous of degree 1 function $f:[0,\infty)^4\to\R$ such that (\ref{n7}) holds, and it suffices to prove that $f$ is subadditive and continuous on $[0,\infty)^4$.  For the former, let $a,b\in [0,\infty)^4$ and define $K_0'=[-a_1,a_2]\times [-b_1,b_2]\times\{o\}$ and $L_0'=[-a_3,a_4]\times [-b_3,b_4]\times\{o\}$, rectangles in the $\{x_1,x_2\}$-plane in $\R^n$.  Note that $K_0',L_0'\in {\mathcal{K}}^n_o$.  One readily verifies that (\ref{zz1})--(\ref{zz4}) hold with $K_0$ and $L_0$ replaced by $K_0'$ and $L_0'$, respectively, allowing the proof of the subadditivity of $f$ to go through as above.  For the continuity of $f$, let $a\in [0,\infty)^4$ and follow the proof above, replacing $K_1$ and $L_1$ by $K_1'=\conv\{K_1,o\}$ and $L_1'=\conv\{L_1,o\}$, respectively. Note that $K_1',L_1'\in {\mathcal{K}}^n_o$. Then (\ref{hhh}) holds with $K_1$ and $L_1$ replaced by $K_1'$ and $L_1'$, respectively, and it is clear that $K_1',L_1'\subset rB^n$ with the same values of $r$ used above.  This allows the proof of the continuity of $f$ and the conclusion to go through as before.

The converse is clear.
\end{proof}

Suppose that $f$ is a continuous and convex (and hence sublinear) function defined on a closed convex cone $C$ in $\R^n$ with apex at the origin.  If it were possible to extend $f$ to a continuous and convex function on $\R^n$, then we could conclude from the proof of Theorem~\ref{thm275} that the set $M$ could be taken to be a compact convex set for all $n\ge 2$.  But \cite[Theorem~2.2]{IgaP98} states that if $\inte C\neq \emptyset$, then unless $C=\R^n$, there is always such an $f$ for which this extension is not possible.
When $n\ge 4$, we can avoid this difficulty, as the following remark demonstrates.

\begin{rem}\label{remark}
{\em When $n\ge 4$, the set $M$ in Theorem~\ref{thm275} can be taken to be a compact convex set, and moreover the proof can be shortened considerably.  To be specific, when $*:\left({\mathcal{K}}^n\right)^2\rightarrow {\mathcal{K}}^n$ is projection covariant and $n\ge 4$, we may set $K_0=[-e_1,e_2]$ and $L_0=[-e_3,e_4]$ and define $M=K_0*L_0$. Identifying $S=\lin\{e_1,e_2,e_3,e_4\}$ with $\R^4$ in the natural way, it follows that by its very definition, $M$ is a compact convex set in $\R^4$. Using the same argument as in the proof of Theorem~\ref{thm12}, we apply (\ref{n7}) with $K=K_0$ and $L=L_0$ to obtain
\begin{equation}\label{MMf}
h_M(x_1,x_2,x_3,x_4)=h_{K_0*L_0}(x)=
f\left(h_{K_0}(-x),h_{K_0}(x), h_{L_0}(-x), h_{L_0}(x)\right)=f(x_1,x_2,x_3,x_4),
\end{equation}
whenever $(x_1,x_2,x_3,x_4)\in H^2$.  Since $(h_K(-x),h_K(x), h_L(-x), h_L(x))\in H^2$ for all $K,L\in {\mathcal{K}}^n$ and $x\in\R^n$, (\ref{n77}) follows directly from (\ref{n7}) and (\ref{MMf}). When $*:\left({\mathcal{K}}^n_o\right)^2\rightarrow {\mathcal{K}}^n$ is projection covariant, we can take $K_0=\conv\{o,-e_1,e_2\}$ and $L_0=\conv\{o,-e_3,e_4\}$ instead.

Observe that in order to obtain (\ref{MMf}), we require projection covariance for subspaces of dimension 4, whereas for Theorem~\ref{thm275}, only projections onto lines are used.}
\end{rem}

\begin{cor}\label{contaff7}
Let $n\ge 2$. An operation $*:\left({\mathcal{K}}^n\right)^2\rightarrow {\mathcal{K}}^n$ (or $*:\left({\mathcal{K}}^n_o\right)^2\rightarrow {\mathcal{K}}^n$) is projection covariant if and only if it is continuous and $GL(n)$ covariant (and hence homogeneous of degree 1).
\end{cor}

\begin{proof}
If $*$ is continuous and $GL(n)$ covariant, then it is projection covariant by Lemma~\ref{CGimpliesP}.  Suppose that $*$ is projection covariant.  The proof of Theorem~\ref{thm275} shows that $h_M$ is continuous on $H^2$ or $[0,\infty)^4$, as appropriate. The continuity of $*$ follows from this, (\ref{n77}), and (\ref{HD}), while the $GL(n)$ covariance of $*$ is an easy consequence of (\ref{n77}) and (\ref{suppchange}).
\end{proof}

As was the case with Theorem~\ref{thm12}, in Theorem~\ref{thm275} there will generally be more than one set $M$ giving rise to the same operation $*$ via (\ref{n77}).

Theorem~\ref{thm275} raises the question as to which closed convex sets $M$ in $\R^4$ are such that the right-hand side of (\ref{n77}) is a support function for all $K, L\in {\mathcal{K}}^n$.  A partial answer is provided by Corollary~\ref{corthmM1}, which implies that this is true if  $M$ is a compact convex subset of $[0,\infty)^4$.  In this case, we have
$$K*L=\oplus_M(K,-K,L,-L),$$
by Lemma~\ref{lemM2}(i) and (\ref{n77}).

\begin{thm}\label{thm276}
Suppose that $n\ge 2$.  The operation $*:\left({\mathcal{K}}^n\right)^2\rightarrow {\mathcal{K}}^n$ is projection covariant and has the identity property if and only if it is Minkowski addition.
\end{thm}

\begin{proof}
By Theorem~\ref{thm275}, there is a closed convex set $M$ in $\R^4$ such that (\ref{n77}) holds.  We first claim that
\begin{equation}\label{verynew}
M\subset\{(x_1,1+x_1,x_3,1+x_3)\in \R^4: x_1,x_3\le 0\},
\end{equation}
a quadrant of a 2-dimensional plane in $\R^4$ containing $(0,1,0,1)$.  To see this, let $a\ge 1$, let $K=[e_1,ae_1]$, and let $L=\{o\}$.  If $x=-e_1$, we have $h_K(x)=-1$, $h_K(-x)=a$, and of course $h_L(\pm x)=0$.  Therefore, by (\ref{n77}),
$$-1=h_K(x)=h_{K*\{o\}}(x)=h_M(a,-1,0,0).$$
Consequently, we have $h_M(a,-1,0,0)=-1<0$ for all $a\ge 1$.  In particular,
\begin{equation}\label{1new}
h_M(1,-1,0,0)=-1.
\end{equation}
Taking $a=1$, i.e., $K=\{e_1\}$, but now $x=e_1$, we obtain
\begin{equation}\label{2new}
h_M(-1,1,0,0)=1.
\end{equation}
Equations (\ref{1new}) and (\ref{2new}) imply that
$$M\subset\{x\in \R^4: x_2=1+x_1\}.$$
Now suppose that $h_M(e_1)\ge c>0$.  Then there is a $z\in M$ such that $z\cdot e_1\ge c$ and
$$h_M(a,-1,0,0)=h_M(ae_1-e_2)\ge (ae_1-e_2)\cdot z\ge ac-z\cdot e_2,$$
which is positive for sufficiently large $a$.  This contradiction shows that $h_M(e_1)\le 0$. In exactly the same way, using $\{o\}*L=L$ for all $L\in {\mathcal K}^n$, we can show that
$$M\subset\{x\in \R^4: x_4=1+x_3\}$$
and $h_M(e_3)\le 0$.  It follows that (\ref{verynew}) holds and the claim is proved.

Next, we claim that $(0,1,0,1)\in M$.  Indeed, if this is not the case, then (\ref{verynew}) and the fact that $M$ is closed yield
$$\alpha=\sup\{x_1+x_3: (x_1,1+x_1,x_3,1+x_3)\in M\}<0.$$
Suppose that $K\in {\mathcal{K}}^n_o$.  Then $h_K(-x), h_K(x)\ge 0$ and we have, by (\ref{n77}),
\begin{eqnarray}
h_{K*K}(x)&=&h_M\left(h_K(-x),h_K(x), h_K(-x), h_K(x)\right)\nonumber\\
&=&\sup\{z_1h_K(-x)+(1+z_1)h_K(x)+z_3h_K(-x)+(1+z_3)h_K(x):\nonumber\\
& & (z_1,1+z_1,z_3,1+z_3)\in M\}\nonumber\\
&=&\sup\{(z_1+z_3)\left(h_K(-x)+h_K(x)\right)+2h_K(x): (z_1,1+z_1,z_3,1+z_3)\in M\}\nonumber\\
&=& \alpha h_K(-x)+(2+\alpha)h_K(x).\label{55}
\end{eqnarray}
Let
$$K=\conv\{o,-e_1+e_2,-e_1-e_2\}$$
and note that $K\in {\mathcal{K}}^n_o$.  Let $x=e_1+e_2\in \R^n$ and $y=e_1-e_2\in \R^n$.  Then
$h_K(x)=h_K(y)=h_K(x+y)=0$ and $h_K(-x)=h_K(-y)=h_K(-x-y)=2$.  By (\ref{55}),
$h_{K*K}(x+y)=2\alpha$ and $h_{K*K}(x)=h_{K*K}(y)=2\alpha$.  But since $\alpha<0$, this implies
$$h_{K*K}(x+y)>h_{K*K}(x)+h_{K*K}(y),$$
contradicting the subadditivity of $h_{K*K}$. This proves the claim.

Now let $K,L\in {\mathcal{K}}^n$ and $x\in\R^n$.  Then
\begin{eqnarray}
h_{K*L}(x)&=&h_M\left(h_K(-x),h_K(x), h_L(-x), h_L(x)\right)\nonumber\\
&=&\sup\{z_1h_K(-x)+z_2h_K(x)+z_3h_L(-x)+z_4h_L(x): (z_1,z_2,z_3,z_4)\in M\}\nonumber\\
&\le &\sup\{z_1h_K(-x)+(1+z_1)h_K(x)+z_3h_L(-x)+(1+z_3)h_L(x): z_1,z_3\le 0\}\nonumber\\
&\le & h_K(x)+h_L(x),\label{56}
\end{eqnarray}
since $h_K(-x)+h_K(x)\ge 0$ and $h_L(-x)+h_L(x)\ge 0$.  On the other hand, $(0,1,0,1)\in M$, so
\begin{eqnarray*}
h_M\left(h_K(-x),h_K(x), h_L(-x), h_L(x)\right)&\ge & (0,1,0,1)\cdot \left(h_K(-x),h_K(x), h_L(-x), h_L(x)\right)\\
&=& h_K(x)+h_L(x),
\end{eqnarray*}
which implies by (\ref{n77}) and (\ref{56}) that $*$ is Minkowski addition.

The converse is clear.
\end{proof}

\begin{cor}\label{cor277}
Suppose that $n\ge 2$.  The operation $*:\left({\mathcal{K}}^n\right)^2\rightarrow {\mathcal{K}}^n$ is continuous, $GL(n)$ covariant, and has the identity property if and only if it is Minkowski addition.
\end{cor}

Theorem~\ref{thm276} and Corollary~\ref{cor277} do not hold for operations $*:\left({\mathcal{K}}^n_o\right)^2\rightarrow {\mathcal{K}}^n$ satisfying the same hypotheses, since $L_p$ addition, for example, is of this type.  (More generally, the operations defined in Example~\ref{ex1} also satisfy these hypotheses.)

The extension of $L_p$ addition given in Example~\ref{LPex} is continuous and $GL(n)$ covariant, and hence projection covariant by Lemma~\ref{CGimpliesP}, but does not have the identity property.  The operation $*$ defined in Example~\ref{ex4} is valid for $K,L\in {\mathcal{K}}^n$, and this
is $GL(n)$ covariant and has the identity property, but is neither continuous nor projection covariant.  Together with the following example, we see that none of the hypotheses of Theorem~\ref{thm276} or Corollary~\ref{cor277} can be omitted.

\begin{ex}\label{ex555}
{\em Define
$$K*L=\left(1+{{\mathcal{H}}^n}(L)\right)K+\left(1+{{\mathcal{H}}^n}(K)\right)L,
$$
for all $K,L\in {\mathcal{K}}^n$.  This operation is continuous and has the identity property, but is not $GL(n)$ covariant.}
\end{ex}

Let $H$ be the closed half-plane $H=\{(s,t)\in\R^2: -s\le t\}$ and let $n\ge 2$. If $*:\left({\mathcal{K}}^n\right)^2\rightarrow {\mathcal{K}}^n$ (or $*:\left({\mathcal{K}}^n_o\right)^2\rightarrow {\mathcal{K}}^n_o$) is projection covariant and associative, then one can show that the set $M$ in Theorem~\ref{thm275} satisfies
$$h_M\left(s_1,t_1,h_M\left(t_2,s_2,t_3,s_3\right),h_M\left(s_2,t_2,s_3,t_3\right)\right)=
h_M\left(h_M\left
(t_1,s_1,t_2,s_2\right),h_M\left
(s_1,t_1,s_2,t_2\right),s_3,t_3\right),
$$
for all $(s_i,t_i)\in H$ (or $(s_i,t_i)\in [0,\infty)^2$, respectively),  $i=1,2,3$.  Moreover, the previous displayed equation can be recast as a vector associativity equation (see \cite[(1), Section~8.2.3, p.~370]{Acz66} in the case when the dimension $m=2$). However, relatively little seems to be known about this vector associativity equation, and in particular, no result corresponding to Proposition~\ref{prp2} is available.

We now turn to operations on pairs of star sets. The following result can be deduced immediately from Lemma~\ref{lem01arbsections}.

\begin{cor}\label{corrlem01arbsections}
Suppose that $n\ge 2$ and that $*:\left({\mathcal{S}}^n\right)^2\rightarrow {\mathcal{S}}^n$ is such that its restriction to $\left({\mathcal{S}}^n_s\right)^2$ is rotation and section covariant.  Then $*:\left({\mathcal{S}}^n_s\right)^2\rightarrow {\mathcal{S}}_s^n$.
\end{cor}

Theorem~\ref{thm1a} and Corollary~\ref{corrlem01arbsections} provide the following corollary.

\begin{cor}\label{newcorthm1asections}
Let $n\ge 2$, and let $*:\left({\mathcal{S}}^n\right)^2\rightarrow {\mathcal{S}}^n$ be such that its restriction to the $o$-symmetric sets is continuous, homogeneous of degree $1$, rotation and section covariant, and associative.  Then this restriction must be either $K*L=\{o\}$, or $K*L=K$, or $K*L=L$, for all $K,L\in {\mathcal{S}}^n_s$, or else $*=\widetilde{+}_p$ for some $-\infty\le p\le\infty$ with $p\neq 0$.
\end{cor}

\begin{thm}\label{lem01788}
If $n\ge 2$, then $*:\left({\mathcal{S}}^n\right)^2\rightarrow {\mathcal{S}}^n$ is rotation and section covariant if and only if there is a function $f:[0,\infty)^4\to\R$ such that
\begin{equation}\label{n788}
\rho_{K*L}(u)=f\left(\rho_K(-u),\rho_K(u), \rho_L(-u), \rho_L(u)\right),
\end{equation}
for all $K,L\in {\mathcal{S}}^n$ and all $u\in S^{n-1}$.  The operation $*$ is in addition homogeneous of degree 1 if and only if $f$ is too and (\ref{n788}) holds for $u\in \R^n\setminus \{o\}$.
\end{thm}

\begin{proof}
Suppose that $*:\left({\mathcal{S}}^n\right)^2\rightarrow {\mathcal{S}}^n$ is rotation and section covariant.  Let $u\in S^{n-1}$. For any two $o$-symmetric star sets $K$ and $L$ in $\R^n$, we have
\begin{equation}\label{seqb88}
(K*L)\cap l_u=(K\cap l_u)*(L\cap l_u).
\end{equation}
One consequence of this is that if $I$ and $J$ are $o$-symmetric closed intervals in $l_u$, we must have $I*J\subset l_u$. Hence there are functions $f_u, g_u:[0,\infty)^4\to\R$
such that
\begin{equation}\label{eqn1788}
[-s_1u,t_1u]*[-s_2u,t_2u]=[-g_u(s_1,t_1,s_2,t_2)u,f_u(s_1,t_1,s_2,t_2)u],
\end{equation}
for all $s_1,t_1,s_2,t_2\ge 0$.

Let $\phi$ be a rotation. Just as in the proof of Theorem~\ref{thm1a}, we use the rotation covariance of $*$ to obtain $f_{\phi u}(s_1,t_1,s_2,t_2)=f_u(s_1,t_1,s_2,t_2)$ for $s_1,t_1,s_2,t_2\ge 0$, and
conclude that $f_u=f$, say, is independent of $u$.  Similarly, $g_u=g$, say, is independent of $u$. Now from (\ref{seqb88}) and (\ref{eqn1788}) with $f_u=f$ and $g_u=g$, we obtain
\begin{eqnarray*}
\lefteqn{[-\rho_{K*L}(-u)u,\rho_{K*L}(u)u]}\\
&=&(K*L)\cap l_u=(K\cap l_u)*(L\cap l_u)=[-\rho_{K}(-u)u,\rho_{K}(u)u]*
[-\rho_{L}(-u)u,\rho_{L}(u)u]\\
&=&[-g\left(\rho_K(-u),\rho_K(u), \rho_L(-u), \rho_L(u)\right)u,f\left(\rho_K(-u),\rho_K(u), \rho_L(-u),\rho_L(u)\right)u].
\end{eqnarray*}
Comparing the second coordinates in the previous equation, we deduce (\ref{n788}).  (Note that in view of the equality of the first coordinates, we must in fact have $g(s_1,t_1,s_2,t_2)=f(t_1,s_1,t_2,s_2)$ for all $s_1,t_1,s_2,t_2\ge 0$.)

Suppose that $*$ is also homogeneous of degree 1. As in the proof of Theorem~\ref{thm1a}, this extra property and (\ref{n788}) yield the homogeneity of $f$ and then it is easy to show that (\ref{n788}) holds for all $u\in \R^n\setminus\{o\}$.

The converses are clear.
\end{proof}

It is natural to ask whether the appropriate analog of Theorem~\ref{thm276} holds for star sets, that is, must an operation $*:\left({\mathcal{S}}^n\right)^2\rightarrow {\mathcal{S}}^n$ that is homogeneous of degree 1, rotation and section covariant, and has the identity property be radial addition?  The answer is negative, as is shown by defining
$$\rho_{K*L}(x)=\rho_K(x)+\rho_L(x)+\sqrt{\rho_K(\pm x)\rho_L(\pm x)},$$
for all $x\in \R^n\setminus\{o\}$, for any particular choice of the plus and minus signs.  These examples also show that the analog of Corollary~\ref{cor277} fails to hold.

\section{Operations having polynomial volume}\label{Minkowski}

In this section we examine operations $*:\left({\mathcal{K}}^n_s\right)^2\rightarrow {\mathcal{K}}^n_s$ (or $*:\left({\mathcal{S}}^n_s\right)^2\rightarrow {\mathcal{S}}^n_s$) that have polynomial volume, that is,
\begin{equation}\label{Minknew}
{{\mathcal{H}}^n}(rK*sL)=\sum_{i,j=0}^{m(K,L)} a_{ij}(K,L)r^{i}s^j,
\end{equation}
for some real coefficients $a_{ij}(K,L)$, some $m(K,L)\in\N\cup\{0\}$, and all $K,L\in {\mathcal{K}}^n_s$ (or all $K,L\in {\mathcal{S}}^n_s$, respectively) and $r,s\ge 0$.

\begin{lem}\label{polhomoglem}
Let $*:\left({\mathcal{K}}^n_s\right)^2\rightarrow {\mathcal{K}}_s^n$ (or $*:\left({\mathcal{S}}^n_s\right)^2\rightarrow {\mathcal{S}}_s^n$) be homogeneous of degree $1$ and have polynomial volume.  Then
\begin{equation}\label{Mink}
{{\mathcal{H}}^n}(rK*sL)=\sum_{i=0}^n a_i(K,L)r^{n-i}s^i,
\end{equation}
for some real coefficients $a_i(K,L)$ and all $K,L\in {\mathcal{K}}^n_s$ (or all $K,L\in {\mathcal{S}}^n_s$, respectively) and $r,s\ge 0$.  Moreover, $a_0(K,L)=a_0(K)\ge 0$, $a_n(K,L)=a_n(L)\ge 0$, and
$a_i(rK,L)=r^{n-i}a_i(K,L)$ and $a_i(K,sL)=s^{i}a_i(K,L)$ for $i=0,\dots,n$ and all $r,s\ge 0$.
\end{lem}

\begin{proof}
Let $K,L\in {\mathcal{K}}^n_s$ (or all $K,L\in {\mathcal{S}}^n_s$, respectively) and $r,s\ge 0$.  For any $t\ge 0$, the assumptions and (\ref{Minknew}) imply that
\begin{eqnarray*}
\sum_{i,j=0}^{m(K,L)} a_{ij}(K,L)r^{i}s^jt^{i+j}&=&
{{\mathcal{H}}^n}((tr)K*(ts)L)=
{{\mathcal{H}}^n}(t(rK*sL))\\
&=&t^n{{\mathcal{H}}^n}(rK*sL)=t^{n}\sum_{i,j=0}^{m(K,L)} a_{ij}(K,L)r^{i}s^j.
\end{eqnarray*}
Comparing coefficients of powers of $t$, we obtain $a_{ij}(K,L)=0$ whenever $i+j\neq n$.  This proves (\ref{Mink}).

By (\ref{Mink}), ${{\mathcal{H}}^n}(rK*\{o\})={{\mathcal{H}}^n}(rK*0L)=a_0(K,L)r^n$ is independent of $L$, so we may write $a_0(K,L)=a_0(K)$ and similarly $a_n(K,L)=a_n(L)$.  Moreover,
${{\mathcal{H}}^n}(rK*sL)={{\mathcal{H}}^n}(1(rK)*sL)$ implies by (\ref{Mink}) that
$$\sum_{i=0}^n a_i(K,L)r^{n-i}s^i=\sum_{i=0}^n a_i(rK,L)1^{n-i}s^i,$$
so $a_i(rK,L)=r^{n-i}a_i(K,L)$ and $a_i(K,sL)=s^{i}a_i(K,L)$ for $i=0,\dots,n$ and all $r,s\ge 0$.
\end{proof}

Ordinary Minkowski addition is not the only operation that satisfies (\ref{Mink}).  Indeed, if $F, G:{\mathcal{K}}^n_s\rightarrow {\mathcal{K}}^n_s$ are homogeneous of degree 1, and
\begin{equation}\label{con}
K*L=F(K)+G(L),
\end{equation}
then
$${{\mathcal{H}}^n}(rK*sL)={{\mathcal{H}}^n}\left(F(rK)+G(sL)\right)={{\mathcal{H}}^n}\left(rF(K)+sG(L)\right)=
\sum_{i=0}^n V_i\left(F(K),G(L)\right)r^{n-i}s^i,$$
by Minkowski's theorem for mixed volumes (see, for example, \cite[Theorem~A.3.1]{Gar06}).
Here $V_i\left(F(K),G(L)\right)$ denotes the mixed volume of $n-i$ copies of $F(K)$ and $i$ copies of $G(L)$.  It follows that in (\ref{Mink}) we may take $a_i(K,L)=V_i\left(F(K),G(L)\right)$.  Moreover, the operation $*$ defined by (\ref{con}) is homogeneous of degree 1, and it is continuous and rotation invariant if both $F$ and $G$ are, respectively.  If $F=G$ and $F(F(K)+F(L))=F(K)+F(L)$ for all $K,L\in {\mathcal{K}}^n_s$, then $*$ is also associative.  An operation with all these properties has already been given in Example~\ref{ex5}; here $F=G$ is given by $F(K)=\left({{\mathcal{H}}^n}(K)/\kappa_n\right)^{1/n}B^n$, and the operation is neither projection covariant nor closely related to $L_p$ addition.

The following easy result shows that when $n=1$, the polynomial volume property is extremely strong.

\begin{thm}\label{thm11}
Let $*:\left({\mathcal{K}}^1_s\right)^2\rightarrow {\mathcal{K}}_s^1$ be homogeneous of degree $1$ and have polynomial volume.  Then there are constants $a,b\ge 0$ such that
\begin{equation}\label{conMink}
K*L=aK+bL,
\end{equation}
for all $K,L\in {\mathcal{K}}_s^1$.  Hence $*$ is Minkowski addition if it also has the identity property.
\end{thm}

\begin{proof}
Since $n=1$, Lemma~\ref{polhomoglem} implies that ${{\mathcal{H}}^1}(rK*sL)=a_0(K)r+a_1(L)s$, for all $K,L\in {\mathcal{K}}_s^1$ and $r,s\ge 0$.  Therefore $rK*sL$ is an $o$-symmetric interval of length $a_0(K)r+a_1(L)s$.  When $K=L=B^1$, this gives
\begin{eqnarray*}
[-r,r]*[-s,s]&=&rB^1*sB^1=[-\left(a_0(B^1)r+a_1(B^1)s\right)/2,\left(
a_0(B^1)r+a_1(B^1)s\right)/2]\\
&=&[-(ar+bs),(ar+bs)]=a[-r,r]+b[-s,s],
\end{eqnarray*}
where $a=a_0(B^1)/2$ and $b=a_1(B^1)/2$, for all $r,s\ge 0$.
\end{proof}

Note that when $n=1$, the apparently stronger property (\ref{conMink}) is actually equivalent to (\ref{con}).

The first author and Mathieu Meyer convinced themselves during discussions in 1996 that the following theorem is true, but followed a rather different route and did not publish the result.

\begin{thm}\label{thm5}
Let $n\ge 1$ and let $-\infty\le p\neq 0\le \infty$ if $n=1$ and $1\le p\le \infty$ if $n\ge 2$. The operation $+_p:\left({\mathcal{K}}^n_s\right)^2\rightarrow {\mathcal{K}}_s^n$ does not have polynomial volume unless $p=1$. (Here we interpret the cases when $-\infty\le p<0$ and $p=\infty$ as in Proposition~\ref{prp2}.)
\end{thm}

\begin{proof}
Suppose that $+_p$ has polynomial volume.  Since $+_p$ is homogeneous of degree 1, \eqref{Mink} holds, by Lemma~\ref{polhomoglem}.  Let $K\in {\mathcal{K}}_s^n$.  Suppose that $-\infty< p\neq 0< \infty$, and $1\le p<\infty$ if $n\ge 2$. Then
$$h_{rK+_p\,sK}(u)^p=h_{rK}(u)^p+h_{sK}(u)^p=(r^p+s^p)h_K(u)^p,$$
for all $u\in S^{n-1}$, so $rK+_p\,sK=(r^p+s^p)^{1/p}K$ for $r,s\ge 0$.  Therefore, assuming ${{\mathcal{H}}^n}(K)>0$, (\ref{Mink}) implies that
$$(r^p+s^p)^{n/p}=\sum_{i=0}^nc_ir^{n-i}s^i,$$
where $c_i$, $i=0,\dots,n$, are constants and where $r,s\ge 0$.  In particular,
\begin{equation}\label{spn}
(1+s^p)^{n/p}=\sum_{i=0}^nc_is^i,
\end{equation}
for all $s\ge 0$.  If $p=\infty$, we get instead that
$$\max\{1,s\}^n=\sum_{i=0}^nc_is^i,$$
for all $s\ge 0$, which is clearly impossible, and the case $p=-\infty$ can be dismissed in a similar fashion.

We claim that (\ref{spn}) implies that $p\in \N$ and $p$ divides $n$.  To see this, note first that if $p<0$, then the left-hand side of (\ref{spn}) converges to 1 as $s\to\infty$. This implies that $c_1=\cdots=c_n=0$. But then the left-hand side of (\ref{spn}) is constant, a contradiction. Assume, therefore, that $p>0$.  The binomial expansion of the left-hand side of (\ref{spn}) yields
\begin{equation}\label{spn2}
\sum_{j=0}^\infty\binom{n/p}{j}s^{pj}=\sum_{i=0}^n c_is^i.
\end{equation}
for $s\in [0,1)$.

If $p\in\N$ and $p$ does not divide $n$, then the left-hand side of (\ref{spn2}) does not terminate, since $\binom{n/p}{j}\neq 0$ for $j\in\N$. Note that $pj\in\N$ and therefore a contradiction follows from the uniqueness of power series representations.

It remains to consider the case when $p>0$ and $p\notin \N$.  Either $p>n$ or there is a $k\in\{0,\ldots,n-1\}$ such that $p\in (k,k+1)$.  Set $s=0$ in (\ref{spn2}) to get $c_0=1$ and hence
\begin{equation}\label{eqt1}
\sum_{j=1}^\infty\binom{n/p}{j}s^{pj}=\sum_{i=1}^n c_is^i.
\end{equation}
If $k=0$, then $p\in(0,1)$ and we deduce from \eqref{eqt1} that
$$
\frac{n}{p}+\binom{n/p}{2}s^p+\cdots=\sum_{i=1}^n c_is^{i-p}.
$$
Evaluating this equation at $s=0$ yields $n/p=0$, a contradiction. Now let $1\le k\le n-1$, so that $p>k\ge 1$. From \eqref{eqt1}, we get
$$
\frac{n}{p}s^{p-1}+\binom{n/p}{2}s^{2p-1}+\cdots=\sum_{i=1}^n c_is^{i-1}.
$$
Evaluating this equation at $s=0$ yields $c_1=0$. Suppose that $c_1=\cdots=c_i=0$ for $i<k\le n-1$. Then
$$
\frac{n}{p}s^{p-i-1}+\binom{n/p}{2}s^{2p-i-1}+\cdots= c_{i+1}+c_{i+2}s+\cdots.
$$
Evaluating this equation at $s=0$ yields $c_{i+1}=0$. Thus we conclude that $c_1=\cdots=c_k=0$ and therefore
$$
\frac{n}{p}+\binom{n/p}{2}s^p+\cdots=c_{k+1}s^{k+1-p}+\cdots.
$$
Evaluating this equation at $s=0$ yields $n/p=0$, a contradiction.

Finally, if $p>n$, we arrive as above at $c_1=\cdots=c_n=0$. Since the left-hand side of \eqref{spn} is not constant, we again obtain a contradiction.  This proves the claim.

More work is needed to exclude the possibility that $1\neq p\in \N$ and $p$ divides $n$.  To this end, let $p>1$ and let $p'$ denote the H\"{o}lder conjugate of $p$, so that $1/p+1/p'=1$. Let
$$D=\left\{x=(x_1,\dots,x_n)\in \R^n: \sum_{i=1}^n\left(\frac{|x_i|}{a_i}\right)^p\le 1\right\},$$
where $a_i>0$, $i=1,\dots,n$, so that $D$ is a linear image of the unit ball in $l_p^n$ under a diagonal matrix with entries $a_i$, $i=1,\dots,n$. We claim that
$$h_D(u)=\left(\sum_{i=1}^n(a_i|u_i|)^{p'}\right)^{1/p'},$$
for $u=(u_1,\dots,u_n)\in S^{n-1}$.  To see this, note that by symmetry we may restrict vectors to the positive orthant.  Then, using H\"{o}lder's inequality and the definition (\ref{suppf}) of the support function, we obtain for $u=(u_1,\dots,u_n)\in S^{n-1}$ and $x=(x_1,\dots,x_n)\in D$,
$$u\cdot x=\sum_{i=1}^nu_ix_i\le \left(\sum_{i=1}^n(a_iu_i)^{p'}\right)^{1/p'}
\left(\sum_{i=1}^n\left(\frac{x_i}{a_i}\right)^p\right)^{1/p}\le \left(\sum_{i=1}^n(a_iu_i)^{p'}\right)^{1/p'}.$$
Now taking
$$\left(\frac{x_i}{a_i}\right)^p=\left(\sum_{i=1}^n(a_iu_i)^{p'}\right)^{-1}(a_iu_i)^{p'},$$
for $i=1,\dots,n$, it is easy to check that $x\in D$ and equality holds in the previous inequalities.  This proves the claim.

Now let
$$K=\left\{x=(x_1,\dots,x_n)\in \R^n:\sum_{i=1}^n|x_i|^{p'}\le 1\right\}$$
and
$$L=\left\{x=(x_1,\dots,x_n)\in \R^n:|x_1/2|^{p'}+\sum_{i=2}^n|x_i|^{p'}\le 1\right\},$$
that is, $K$ is the unit ball in $l^n_{p'}$ and $L$ is a linear image of it under a diagonal matrix with entries $2,1,\dots,1$.  Then we have
\begin{eqnarray*}
h_{K+_psL}(u)&=&\left(\sum_{i=1}^n|u_i|^{p}+s^p\left(|2u_1|^{p}
+\sum_{i=2}^n|u_i|^{p}\right)\right)^{1/p}\\
&=&\left((1+2^ps^p)|u_1|^{p}+\sum_{i=2}^n(1+s^p)|u_i|^{p}\right)^{1/p},
\end{eqnarray*}
for $u=(u_1,\dots,u_n)\in S^{n-1}$, so $K+_psL$ is a linear image of the unit ball in $l^n_{p'}$ under a diagonal matrix with entries $(1+2^ps^p)^{1/p},(1+s^p)^{1/p},\dots,(1+s^p)^{1/p}$.  It follows that
$${{\mathcal{H}}^n}(K+_psL)=c(1+2^ps^p)^{1/p}(1+s^p)^{(n-1)/p},$$
where $c$ is a constant depending only on $n$ and $p$. Therefore (\ref{Mink}) implies that
$$(1+2^ps^p)(1+s^p)^{n-1}=\left(c_0+c_1s+\cdots+c_ns^n\right)^p=q(s)^p,$$
say, for some $c_i$, $i=0,\dots,n$.  Let $q(s)=r_1(s)^{m_1}\cdots r_k(s)^{m_k}$
be a factorization into powers of irreducible factors, any two of which are relatively prime. Then
$$
(1+2^ps^p)(1+s^p)^{n-1}=r_1(s)^{m_1p}\cdots r_k(s)^{m_kp}.
$$
The polynomials $1+2^ps^p$ and $1+s^p$ do not have a common divisor, since
any divisor must also divide $s^p$ and hence be of the form $s^i$, for some $i=1,\ldots,p$. But
$s^i$ is not a divisor of $1+s^p$. Since $\R[s]$ is a unique factorization domain, we deduce that $r_1(s)^{m_1p}$, say, is a divisor of $1+2^ps^p$. This implies that $m_1p\,\text{deg}(r_1)\le p$ and thus $m_1=\text{deg}(r_1)=1$. Writing $r_1(s)=a+bs$, where $a,b\neq 0$, we obtain
$$r_1(s)^p=\sum_{i=0}^p\binom{p}{i}a^ib^{p-i}s^{p-i}=\lambda(1+2^ps^p),$$
for some $\lambda\neq 0$. If $p>1$, this implies that $a=0$ or $b=0$, a contradiction.
This shows that $p=1$.
\end{proof}

\begin{cor}\label{cor6}
Let $*:\left({\mathcal{K}}^n_s\right)^2\rightarrow {\mathcal{K}}_s^n$ be an associative operation that has polynomial volume.  If $n=1$,  assume that $*$ is also homogeneous of degree 1. If $n\ge 2$, assume that $*$ is also projection covariant.  Then either $K*L=\{o\}$, or $K*L=K$, or $K*L=L$, for all $K,L\in {\mathcal{K}}^n_s$ or $*$ is Minkowski addition.
\end{cor}

\begin{proof}
When $n=1$, the result follows easily from Theorem~\ref{thm11} and the associativity assumption.  For $n\ge 2$, we appeal instead to
Theorems~\ref{thm1} and~\ref{thm5}.
\end{proof}

None of the assumptions in the previous corollary can be omitted.  Indeed,
the operation defined by $K*L=2K+L$, for all $K,L\in {\mathcal{K}}^n_s$, shows that the associativity assumption in the previous corollary cannot be omitted.  The operation in Example~\ref{ex5555} is easily seen to have polynomial volume (note that the map $F$ is homogeneous of degree 1 and $F^{-1}$ preserves volume), so projection covariance is necessary when $n\ge 2$.  Finally, $L_p$ addition is associative and projection covariant when $n\ge 2$, but does not have polynomial volume, by Theorem~\ref{thm5}.

The following corollary is a direct consequence of the previous one and provides a characterization of Minkowski addition.

\begin{cor}\label{cor12}
Let $n\ge 2$.  The operation $*:\left({\mathcal{K}}^n_s\right)^2\rightarrow {\mathcal{K}}_s^n$ is continuous, $GL(n)$ covariant, associative, and has the identity property and polynomial volume if and only if it is Minkowski addition.
\end{cor}

\begin{thm}\label{thm7}
Let $-\infty\le p\neq 0\le \infty$. The operation $\widetilde{+}_p:\left({\mathcal{S}}^n_s\right)^2\rightarrow {\mathcal{S}}_s^n$ has polynomial volume if and only if $p\in\N$ and $p$ divides $n$. (Here we interpret the cases when $-\infty\le p<0$ and $p=\infty$ as in Proposition~\ref{prp2}.)
\end{thm}

\begin{proof}
Suppose that $\widetilde{+}_p$ has polynomial volume. Since $\widetilde{+}_p$ is homogeneous of degree 1, \eqref{Mink} holds, by Lemma~\ref{polhomoglem}. Let $K\in {\mathcal{S}}_s^n$.  Suppose that $-\infty< p\neq 0< \infty$. Then
$$\rho_{rK\widetilde{+}_p\,sK}(u)^p=\rho_{rK}(u)^p+\rho_{sK}(u)^p=(r^p+s^p)\rho_K(u)^p,$$
for all $u\in S^{n-1}$, so $rK\widetilde{+}_p\,sK=(r^p+s^p)^{1/p}K$ for $r,s\ge 0$.  As in the proof of Theorem~\ref{thm5}, we take $r=1$, conclude that
$$(1+s^p)^{n/p}=\sum_{i=0}^nc_is^i,$$
for all $s\ge 0$, and deduce that $p\in \N$ and $p$ divides $n$.  The cases $p=\pm\infty$ can also be excluded as in the proof of Theorem~\ref{thm5}.

Now assume that $p\in \N$ and $p$ divides $n$. For arbitrary $K, L\in {\mathcal{S}}_s^n$, we have
$$\rho_{rK\widetilde{+}_p\,sL}(u)^p=(r\rho_K(u))^p+(s\rho_L(u))^p,$$
where $n=mp$, say, $m\in \N$.  So
\begin{eqnarray*}
{{\mathcal{H}}^n}(rK\widetilde{+}_p\,sL)&=&\frac{1}{n}\int_{S^{n-1}}\rho_{rK\widetilde{+}_psL}(u)^n\,du\\
&=&\frac{1}{n}\int_{S^{n-1}}\left((r\rho_K(u))^p+(s\rho_L(u))^p\right)^m\,du\\
&=&\frac{1}{n}\int_{S^{n-1}}\sum_{j=0}^m
\binom{m}{j}\left(r\rho_K(u)\right)^{(m-j)p}\left(s\rho_L(u)\right)^{jp}\,du\\
&=&\sum_{i=0}^n a_i(K,L)r^{n-i}s^i,
\end{eqnarray*}
where
$$
a_i(K,L)=\left\{\begin{array}{ll}
\binom{n/p}{i/p}\frac{1}{n}\int_{S^{n-1}}\rho_K(u)^{n-i}\rho_L(u)^i\,du, \hspace{.15in}{\text{ if $i=jp$, $j=0,\dots,n/p$}}\\
0,\hspace{2.25in} {\text{otherwise}}.
\end{array}\right.
$$
Therefore (\ref{Mink}) holds, as required.
\end{proof}

\begin{cor}\label{cor8}
Let $*:\left({\mathcal{S}}^n_s\right)^2\rightarrow {\mathcal{S}}_s^n$ be an associative operation that has polynomial volume.  If $n=1$ assume that $*$ is also homogeneous of degree $1$. If $n\ge 2$, assume that $*$ is also rotation and section covariant.  Then either $K*L=\{o\}$, or $K*L=K$, or $K*L=L$, or $K*L=K\widetilde{+}_p\,L$, where $p\in \N$ and $p$ divides $n$, for all $K,L\in {\mathcal{S}}^n_s$.
\end{cor}

\begin{proof}
This follows directly from Theorem~\ref{thm11} (or Corollary~\ref{cor6}) when $n=1$ and from Theorems~\ref{thm1a} and~~\ref{thm7} when $n\ge 2$.
\end{proof}

\end{document}